\documentclass[a4paper,11pt]{article}

\usepackage[T1]{fontenc}
\usepackage{ae,aecompl}
\usepackage{amsfonts}
\usepackage{amssymb}
\usepackage{amsmath}
\usepackage{amsthm}
\usepackage{verbatim}
\usepackage{hyperref}
\usepackage{xcolor}
\usepackage{xspace}
\usepackage{enumitem}
\usepackage{tikz}

\usepackage[top=1in, bottom=1.25in, left=1.25in, right=1.25in]{geometry}

\setitemize{noitemsep,topsep=5pt,parsep=5pt,partopsep=5pt}
\setenumerate{noitemsep,topsep=5pt,parsep=5pt,partopsep=5pt}

\theoremstyle{plain}
\newtheorem{theorem}{Theorem}[section]
\newtheorem{lemma}[theorem]{Lemma}
\newtheorem{corollary}[theorem]{Corollary}
\newtheorem{proposition}[theorem]{Proposition}
\theoremstyle{definition}
\newtheorem{definition}[theorem]{Definition}
\newtheorem{remark}[theorem]{Remark}
\newtheorem{example}[theorem]{Example}

\DeclareMathOperator{\Agt}{\mathbb{A}gt} %Set of agents
\DeclareMathOperator{\St}{St} %Set of states
\DeclareMathOperator{\Prop}{\Pi} %Set of proposition symbols
\DeclareMathOperator{\Act}{Act} %Set of actions
\DeclareMathOperator{\paths}{\mathsf{paths}} %Set of paths from certain state with certain strategy
\newcommand{\void}{\ensuremath{\mathsf{any}\xspace}}
\newcommand{\card}{\ensuremath{\mathsf{card}\xspace}}
\newcommand{\bd}{\ensuremath{\mathsf{BD}\xspace}}

\newcommand{\rbb}{\ensuremath{\mathsf{RBB}\xspace}}
\newcommand{\avact}{\ensuremath{\mathsf{action}\xspace}}
\newcommand{\tlimbound}{\Gamma} %Upper bound for allowed time limits

\newcommand{\llangle}{{\langle\hspace{-0.8mm}\langle}}
\newcommand{\rrangle}{{\rangle\hspace{-0.8mm}\rangle}}

\RequirePackage{ifthen}

\newboolean{reportversion}

\setboolean{reportversion}{false}

\newcommand\refhereortechnicalreport{
\ifthenelse{\boolean{reportversion}}{%when in technical report
\!Section \ref{technicalreport}}
{%when in AAMAS paper
\!\!the technical report \cite{technicalreport}}
}

\newcommand{\anote}[1]{}
\newcommand{\afnote}[1]{}
\newcommand{\rnote}[1]{}
\newcommand{\rfnote}[1]{}
\newcommand{\vnote}[1]{}
\newcommand{\vfnote}[1]{}
\newcommand{\acta}{\ensuremath{\alpha\xspace}}

\newcommand{\coop}[2][]{\langle\!\langle{#2}\rangle\!\rangle_{_{\!\mathit{#1}}}\,}

\newcommand{\atlx}{\mathord \mathsf{X}\, }
\newcommand{\atlf}{\mathord \mathsf{F}\, }
\newcommand{\atlg}{\mathord \mathsf{G}\, }
\newcommand{\atlu}{\, \mathsf{U} \, }
\newcommand{\atlr}{\, \mathsf{R} \, }
\newcommand{\X}{\atlx}
\newcommand{\F}{\atlf}
\newcommand{\G}{\atlg}
\newcommand{\U}{\atlu}
\newcommand{\Rl}{\atlr}

\newcommand{\atla}{\varphi}
\newcommand{\atlb}{\psi}

\newcommand{\Logicname}[1]{\ensuremath{\mathsf{#1}}}
\newcommand{\CTL}{\Logicname{CTL}\xspace}

\newcommand{\ATL}{\Logicname{ATL}\xspace}
\newcommand{\ATLs}{\Logicname{ATL^*}\xspace}
\newcommand{\ATLplus}{\Logicname{ATL^+}\xspace}
\newcommand{\GTS}{\Logicname{GTS}\xspace}
\newcommand{\CGM}{\Logicname{CGM}\xspace}
\newcommand{\IF}{\Logicname{IF}\xspace}

\newcommand{\vrf}{\ensuremath{\mathbf{V}\xspace}}
\newcommand{\ovrf}{\ensuremath{\overline{\vrf}\xspace}}

\newcommand{\ctr}{\ensuremath{\mathbf{C}\xspace}}
\newcommand{\octr}{\ensuremath{\overline{\mathbf{C}}\xspace}}
\newcommand{\pl}{\ensuremath{\mathbf{P}\xspace}}
\newcommand{\opl}{\ensuremath{\mathbf{\overline{P}}\xspace}}
\newcommand{\pos}{\ensuremath{\mathsf{Pos}\xspace}}
\newcommand{\setpos}{\ensuremath{\mathsf{POS}\xspace}}
\newcommand{\tlim}{\ensuremath{\gamma\xspace}}
\newcommand{\initlim}{\ensuremath{\gamma_{0}\xspace}}
\newcommand{\initst}{\ensuremath{q_{0}\xspace}}
\newcommand{\qzero}{\ensuremath{q_{in}\xspace}}
\newcommand{\evalgame}{\ensuremath{\mathcal{G}\xspace}}
\newcommand{\embgame}{\ensuremath{\mathbf{g}\xspace}}

\newcommand{\step}{\ensuremath{\mathsf{step}\xspace}}

\newcommand{\strev}{\ensuremath{\Sigma\xspace}}
\newcommand{\cnf}{\ensuremath{\mathsf{Conf}\xspace}}
\newcommand{\win}{\ensuremath{\mathsf{win}\xspace}}
\newcommand{\lose}{\ensuremath{\mathsf{lose}\xspace}}

\newcommand{\defstyle}{\textbf}

\newcommand{\bbN}{\ensuremath{\mathbb{N}}}

\begin{document}

% Page heads
\markboth{V. Goranko, A. Kuusisto and R. R\"onnholm}{Game-Theoretic Semantics for Alternating-Time Temporal Logic}

% Title portion
\title{Game-Theoretic Semantics \\ for Alternating-Time Temporal Logic\footnote{This paper is a substantially extended and revised version of the conference paper \cite{GKR-AAMAS2016}.}
}

\date{}

\author{
  Valentin Goranko\footnote{\texttt{valentin.goranko@philosophy.su.se}}\\
  Stockholm University\\
  Sweden
%\texttt{first1.last1@xxxxx.com}
  \and
  Antti Kuusisto\footnote{\texttt{kuusisto@uni-bremen.de}}\\
  University of Bremen\\
  Germany
%\texttt{first2.last2@xxxxx.com}
  \and
  Raine R\"{o}nnholm\footnote{\texttt{raine.ronnholm@uta.fi}}\\
  University of Tampere\\
  Finland
%\texttt{first2.last2@xxxxx.com}
}

\maketitle

\begin{abstract}
We introduce several versions of game-theoretic semantics (\GTS) for Alternating-Time Temporal Logic (\ATL). In \GTS, truth is defined in terms of existence of a winning strategy in a semantic evaluation game, and thus the game-theoretic perspective appears in the framework of \ATL on two semantic levels: on the object level in the standard semantics of the strategic operators, and on the meta-level where game-theoretic logical semantics is applied to \ATL. We unify these two perspectives into semantic evaluation games specially designed for \ATL.
The game-theoretic perspective enables us to identify new variants of the semantics of \ATL based on limiting the time resources available to the verifier and falsifier in the semantic evaluation game. We introduce and analyse an \emph{unbounded} and \emph{(ordinal) bounded} \GTS and prove these to be equivalent to the standard (Tarski-style) compositional semantics. We show that in bounded \GTS, truth of \ATL formulae can always be determined in finite time, i.e., without constructing infinite paths. We also introduce a non-equivalent \emph{finitely bounded} semantics and argue that it is natural from both logical and game-theoretic perspectives.
\end{abstract}

%%%%%%%%%%%%%%%%%%%%%%%%%%%%%
%%%%%%%%%%%%%%%%%%%%%%%%%%%%%
\section{Introduction}
%%%%%%%%%%%%%%%%%%%%%%%%%%%%%
%%%%%%%%%%%%%%%%%%%%%%%%%%%%%

\emph{Alternating-Time Temporal Logic} \ATL was introduced in \cite{AHK02} as a multi-agent extension of the branching-time temporal logic \CTL.  The semantics of \ATL is defined over  \emph{multi-agent transition systems}, also known as  \emph{concurrent game models}, in which agents take simultaneous actions at the current state and the resulting collective action
determines the subsequent transition to a successor state. 
The logic \ATL and its extension \ATLs have gradually become the most popular logical formalisms for reasoning about strategic abilities of agents in synchronous multi-agent systems.

\emph{Game-theoretic semantics} (\GTS) of logical languages has a rich 
history going back to Hintikka \cite{hinti}, Lorenzen \cite{lore} and others.
For an overview of the topic,
see \cite{HintikkaSandu97}.
%
%originally introduced by Hinitikka \cite{hinti}
%
In \GTS, truth of a logical formula $\varphi$ is determined in a
formal \emph{debate} between two players,
\emph{Eloise} and \emph{Abelard}. Eloise is trying to
verify $\varphi$, while Abelard is opposing her and  trying to falsify it. 
Each logical connective in the language is associated with a related rule in the game.
The framework of \GTS has turned out to be particularly
useful for the purpose of defining variants of semantic approaches to different logics.
For example, the \IF-logic of Hintikka and Sandu \cite{sandu} is an extension of first-order logic which was originally developed using \GTS.
Also, the game-theoretic approach to semantics has led to new methods for solving 
decision problems of logics, e.g., via
using parity games for the $\mu$-calculus. The close connection between   
between $\mu$-calculus model checking and parity games was first
discussed in \cite{jutlawhatever}.

%In this paper we introduce versions of game-theoretic semantics for \ATL and relate them with compositional semantics.
%%
%The game-theoretic perspective appears in the framework of \ATL on two levels: on object level, in its standard compositional semantics of the strategic operators, and on a meta-level, where the generic game-theoretic logical semantics is applied to \ATL. To unify these two perspectives we  design special evaluation games for \ATL. The game-theoretic approach enables us to identify essentially new variants of the semantics of \ATL, based on limiting the time resources afforded to the verifier and falsifier to win the evaluation game. We introduce and analyse \emph{unbounded}, \emph{ordinal-bounded}, which we show to be equivalent to the standard compositional semantics, as well as the non-equivalent \emph{finitely bounded} semantics, which we argue to be quite natural from both logical and game-theoretic perspective. 
%

Here we introduce game-theoretic semantics for \ATL.
In our framework, the rules corresponding to strategic operators involve
scenarios where Eloise and Abelard are both controlling (leading) 
coalitions of agents with opposing objectives.
The perspective offered by \GTS enables us to  
develop novel approaches to \ATL based on different time resources available  
to the players.
In \emph{unbounded} \GTS, a coalition trying to verify an  
until-formula is allowed to continue without a time limit, the price  
of an infinite play being a loss in the game. 
In \emph{bounded} \GTS, the coalition must commit to
\emph{finishing in finite time} by submitting an  
\emph{ordinal number} in the beginning of the game. That ordinal controls  
the available time resources in the game and \emph{guarantees a finite play}.   
Notably, even safety games (for always-formulae) are evaluated in  
finite time, and thus the bounded and unbounded
approaches to \GTS are conceptually very different.
However, despite the differences between the two semantics, we show that they are in fact equivalent to the standard compositional
(i.e., Tarski-style) semantics of \ATL, and therefore also to each other.
We also introduce a restricted variant of the bounded \GTS,
called \emph{finitely bounded} \GTS, where the ordinals controlling the time flow must always be finite. This is a particularly simple system of semantics where the
players will always announce the ultimate (always finite) duration of the game before  
the game begins. 
We show that the finitely bounded \GTS is equivalent to the  
standard \ATL semantics on \emph{image-finite models}, and therefore provides an
alternative approach to \ATL sufficient for most practical purposes.
It is worth noting here that the difference between the finitely bounded
and unbounded semantics is conceptually directly linked to the
difference between \emph{for-loops} and \emph{while-loops}. 
Since the finitely bounded semantics is new, we also develop an equivalent 
(over all models) Tarski-style semantics for it.

For all systems of game-theoretic semantics studied here, we show that
positional strategies suffice in the perfect information setting for \ATL.
In the framework of unbounded semantics, this means that strategies depend on the current state only.
In the case of bounded and finitely bounded semantics, strategies may additionally depend on the value of the ordinal guiding the time flow of the game.
As a by-product of our semantic investigations, we also introduce and
study a range of new game-theoretic
notions that are of interest also outside \ATL, 
such as, e.g., \emph{timed strategies}, \emph{$n$-canonical strategies}, 
\emph{$\infty$-canonical strategies}.

There are several related works concerning both \ATL and its extensions as well as game-theoretic semantics. We mention here some such papers.\footnote{We thank an anonymous reviewer for pointing out many of these references.}
\begin{itemize}
%\begin{enumerate}

\item Imposing explicit time bounds on temporal operators in logical
specification languages has often been done
syntactically in the context of \emph{quantitative temporal reasoning}, see, e.g., \cite{EmersonMSS92}. Typically, that amounts, e.g., to replacing $\F q$ by a disjunction $q \lor \F q\lor  ... \lor \F^{k} q$, abbreviated by $\F^{\leq k} q$, for $k \in \bbN$, that imposes an explicit and uniform bound of $k$ steps for satisfying the eventuality $q$. 
%to which $\F^{\leq k}$ applies. 
%Unlike these approaches, ours is purely semantic. 

\item Imposing time bounds in the semantics of temporal operators has also been proposed in several contexts, most notably in relation to \emph{finitary fairness} \cite{toplas/AlurH98} 
and \emph{promptness}
in LTL
%from the reviewer: 
\cite{Kupferman2007} as well as in omega-regular and parity games 
\cite{tocl/ChatterjeeHH09},
\cite{atva/AlmagorHK10},
%\cite{fijalkow_et_al:LIPIcs:2012:3853}, 
%\cite{LMCS/abs-1207-0663},  
\cite{lpar/MogaveroMS13},
%also to: 
\cite{aamas/AminofMMR16}. 
The idea of promptness (in LTL) is to replace the 
unbounded eventuality operator $\F$ by a ``bounded-eventually'' operator 
$\F_{\textbf{p}}$ for which an upper bound for the satisfaction of the eventuality is explicitly specified in the semantics.
%An essential difference with our approach, besides different motivations, is that these apply mainly %to LTL and impose the time bounds in the semantics as a parameter of the evaluation context, %whereas in our work they are only implicitly imposed by the quantification in the semantic condition. 

\item  \GTS-like approaches have been used to solve decision problems of, e.g.,
fragments of Strategy Logic (especially with respect to the so-called ``behavioral semantics'') in  
\cite{aamas/MogaveroMS14},
\cite{aaai/CermakLM15}. 

\item The article \cite{mucalc} discusses an alternative game-theoretic semantics for
the modal $\mu$-calculus based on imposing time-bounds on the fixed-point operators.

%\vadd{
\item In the recent work 
\cite{GorankoKR17}, we introduce and study the variant of CTL with finitely bounded semantics. 
We show that it is a proper sublogic (in terms of validities) of standard CTL. We also show that it lacks the finite model property, but its satisfiability problem is nevertheless decidable (shown via finitary tableaux methods) and still EXPTIME-complete. In \cite{FBS-ATL-2018}, we extend these results for \ATL and also provide an infinitary complete axiomatization for  ATL with finitely bounded semantics.
%}
\end{itemize}
%\end{enumerate}

%In summary, while each of these works bears 

%
The main contributions of this paper are twofold: the development of 
game-theoretic semantics for \ATL and the introduction of new resource-sensitive versions of logics for multi-agent strategic reasoning. The latter relates conceptually to the study of other resource-bounded versions of \ATL, see, e.g., \cite{AlechinaLNR11}, \cite{MonicaNP11}, \cite{AlechinaBLN15}.
We note that some of our technical results could be obtained using alternative methods from coalgebraic modal logic. 
We discuss this issue in more detail in the concluding section~\ref{sec: Conclusion}.

The structure of the paper is as follows. 
After the preliminaries in Section \ref{sec: Preliminaries}, 
we develop the bounded and unbounded \GTS in Section \ref{sec:Game-theoretic semantics}. We analyse these frameworks in Section~\ref{sec:Analysing evaluation games}, where we show, inter alia, that the two game-theoretic systems, bounded and unbounded, are equivalent. In Section \ref{sec:Comparing game-theoretic and compositional semantics}, we compare the game-theoretic and
standard Tarski-style semantics and establish the above-mentioned
equivalences between them. After brief concluding remarks, we provide an appendix section \ref{sec:Appendix} with some technical proofs. 

%
%This paper is the full version of the conference paper \cite{GKR-AAMAS2016}.
%

%\cite{},
%\cite{},
%\cite{},
%\cite{}.

%%%%%%%%%%%%%%%%%%%%%%%%%%%%%
%%%%%%%%%%%%%%%%%%%%%%%%%%%%%
\section{Preliminaries}
\label{sec: Preliminaries}

%\subsection{
%\paragraph{Concurrent game models}
%
Here we define concurrent game models as well as the syntax and standard compositional semantics of \ATL. For detailed background on these, see 
\cite{AHK02} or \cite{GorDrim06}.

\begin{definition}
A \defstyle{concurrent game model} ($\mathit{\CGM}$)
%\[
$\mathcal{M}$ is a tuple
$$(\Agt, \St, \Prop, \Act, d, o, v)$$
%
%
%
%\]
which consists of the following non-empty sets:

-- \defstyle{agents} $\Agt = \{1,\dots,k\}$,

-- \defstyle{states} $\St$, 

-- \defstyle{proposition symbols} $\Prop$,

-- \defstyle{actions} $\Act$, 
\\
and the following functions: 
 
-- an \defstyle{action function} $d:\Agt\times\St\rightarrow(\mathcal{P}(\Act)\setminus\{\emptyset\})$, which assigns to each agent at each state a non-empty set of actions available to the agent at that state;   

-- a \defstyle{transition function} $o$, which assigns to each state $q\in\St$ and \defstyle{action profile} $\vec{\acta}$ at $q$ (where $\vec{\acta}$ is a tuple of actions $\vec{\acta} = (\alpha_1,\dots,\alpha_k)$ such that $\alpha_i\in d(i,q)$ for each $i\in\Agt$), 
a unique \defstyle{outcome state} $o(q, \vec{\acta})$;    

-- and a \defstyle{valuation function} $v:\Prop\rightarrow\mathcal{P}(\St)$. 
\end{definition}

Sets of agents $A\subseteq\Agt$ are also called \defstyle{coalitions}. The complement  $\overline{A}= \Agt\setminus A$ of a coalition $A$ is called the \defstyle{opposing coalition} (of $A$). 
We also define the set of action tuples that are available to a coalition $A$ at a state $q\in\St$: 
\[\avact(A,q):=\{(\alpha_i)_{i\in A}\mid \alpha_i\in d(i,q) \text{ for each } i\in A\}.\]

\begin{definition}
\label{def:CGM}
Let $\mathcal{M} = (\Agt, \St, \Prop, \Act, d, o, v)$ be a concurrent game model.
A \defstyle{(positional) strategy}\footnote{Unless otherwise specified, a `strategy' hereafter will mean a positional and deterministic strategy.} for an agent $a\in\Agt$ is a function $s_a: \St\rightarrow\Act$ such that $s_a(q)\in d(a,q)$ for each $q\in\St$. 
A \defstyle{collective strategy} $S_A$ for $A\subseteq\Agt$ is a tuple of individual strategies, one for each agent in $A$. 
A \defstyle{path}  in $\mathcal{M}$ is a sequence of states $\Lambda$ 
%that can be formed by subsequent transitions, i.e. 
such that
$\Lambda[n\!+\!1] = o(\Lambda[n],\vec{\acta})$ for
some action profile
$\vec{\acta}$ for $\Lambda[n]$,
where $\Lambda[n]$ is the $n$-th state in $\Lambda$ $(n\in\mathbb{N})$. The function 
$\paths(q,S_A)$ returns the set of all paths generated when the agents in $A$ play according to $S_A$, beginning from the state $q$.
\end{definition}
%
%
%
%\subsection{
%\paragraph{The alternating-time temporal logic ATL}
%%%%%%Techrep%%%%%%%%
%
%

\defstyle{Alternating-time temporal logic} \ATL,
introduced in \cite{AHK02}, is a logic, suitable for specifying and verifying qualitative objectives of players and coalitions in concurrent game models.
The main syntactic construct of \ATL is a formula of type $\coop{A} \varphi$, intuitively meaning
that 
%\begin{quote}
\textit{the coalition $A$ has a collective strategy to guarantee the satisfaction of the objective $\varphi$ on every play enabled by that strategy.}
The syntax of \ATL is
defined as follows\footnote{The operator $\atlr$ (Release) was not part of the original syntax of \ATL but has been commonly added later.}
\[
	\atla::= p \mid \neg \atla \mid (\atla\vee \atla) \mid \coop{A} \atlx \atla 
	\mid \coop{A}\atla \atlu \atla \mid \coop{A}\atla \atlr \atla
\]
where $A\subseteq \Agt$ and $p\in\Prop$.
Other Boolean connectives are defined as usual,
and the combined operators $\coop{A} \atlf \atla$ and $\coop{A} \atlg \atla$ are defined respectively by $\coop{A}  \top \atlu \atla$ and $\coop{A} \bot \atlr \atla$. 
%
%
%
%\techrep{To keep the notation lighter, we will list the members of $C$  in $\coop{C}$  without %using $\{ \}$.} 
%
\begin{definition}\label{def: compositional semantics for ATL}
Let $\mathcal{M} = (\Agt, \St, \Prop, \Act, d, o, v)$ be a \CGM, $q\in\St$ a 
state
and $\varphi$ an \ATL-formula. Truth of $\varphi$ in $\mathcal{M}$ at $q$, denoted by $\mathcal{M},q\models\varphi$, is defined as follows:
\begin{itemize}[leftmargin=*]
\item $\mathcal{M},q\models p$  iff  $q\in v(p)$\ \ $\mathrm{(\text{for }}p\in\Prop\mathrm{)}$.
\item $\mathcal{M},q\models \neg\psi$ iff $\mathcal{M},q\not\models\psi$.
\item $\mathcal{M},q\models \psi\vee\theta$ iff $\mathcal{M},q\models\psi$ or $\mathcal{M},q\models\theta$.
\item $\mathcal{M},q\models\coop{A} \X\psi$ iff there exists a
collective strategy $S_A$ such that for each $\Lambda\in\paths(q,S_A)$, we have $\mathcal{M},\Lambda[1]\models\psi$.
\item $\mathcal{M},q\models\coop{A}\psi \U\theta$ iff there exists a
collective strategy $S_A$ such that for each $\Lambda\in\paths(q,S_A)$, there is $i\geq 0$ such that $\mathcal{M},\Lambda[i]\models\theta$ and $\mathcal{M},\Lambda[j]\models\psi$ \,for every $j < i$.
\item $\mathcal{M},q\models\coop{A}\psi \Rl\theta$ iff there exists a
collective strategy $S_A$ such that for each $\Lambda\in\paths(q,S_A)$ and $i\geq 0$, we have $\mathcal{M},\Lambda[i]\models\theta$ or there is $j < i$\, such that $\mathcal{M},\Lambda[j]\models\psi$.
\end{itemize}
\end{definition}
%

%
%We define conjunction and temporal operators $F$ and $G$ as abbreviations in the standard way:
%\begin{align*}
%	\psi\wedge\theta &:= \neg(\neg\psi\vee\neg\theta) \\
%	\coop{A} F\psi &:= \coop{A} \top U\psi \\
%	\coop{A} G\psi &:= \coop{A} \bot R\psi.
%\end{align*}
%
%%%%%%%%%%%%%%%%%%%%%%%%%%%%
%%%%%%%%%%%%%%%%%%%%%%%%%%%%
\section{Game-theoretic semantics}
\label{sec:Game-theoretic semantics}
%%%%%%%%%%%%%%%%%%%%%%%%%%%%
%%%%%%%%%%%%%%%%%%%%%%%%%%%%
%
%
%
In this section we introduce three versions of evaluation games for \ATL: 
\emph{unbounded, (ordinal) bounded, and
finitely bounded evaluation games}. 
These games will ultimately enable us to define three different versions of game-theoretic semantics for \ATL.

\subsection{Unbounded evaluation games}
Given a \CGM $\mathcal{M}$,  a state $\qzero$ and a formula $\varphi$, the
\defstyle{evaluation game} $\evalgame(\mathcal{M}, \qzero, \varphi)$ is, intuitively, a 
debate between two opponents, \defstyle{Eloise} ({\bf E})  and \defstyle{Abelard} ({\bf A}),
about whether the formula $\varphi$ is true at the state $\qzero$ in the model $\mathcal{M}$. Eloise claims that $\varphi$ is true, so she adopts (initially) the role of a \defstyle{verifier} in the game, and Abelard tries to prove the formula false, so he is (initially) the \defstyle{falsifier}. These roles can swap in the course of the game when negations are encountered in the formula to be evaluated.  Truth of an \ATL formula $\varphi$ in $\mathcal{M}$ at $q_{\mathit{in}}$
will be defined as existence of a winning strategy for $\bf{E}$ in  
the game $\evalgame(\mathcal{M}, \qzero, \varphi)$. 

%The games has \emph{positions} of type
%
%$\pos=({\pl},q,\psi)$, where $\pl$  in the debate has the interpretation that
%
%\textit{the player ${\pl}$ is claiming that the formula $\psi$ holds in the state $q$}. 

We will often use the following notation: if $\pl\in\{{\bf A},{\bf E}\}$, 
then ${\opl}$ denotes the \defstyle{opponent} of ${\pl}$,    
i.e., ${\opl}\in \{{\bf A},{\bf E}\}\setminus\{\pl\}$.
\begin{definition}\label{def: Unbounded GTS}
Let $\mathcal{M} = (\Agt, \St, \Prop, \Act, d, o, v)$ be a \CGM, $\qzero\!\in\!\St$ and $\varphi$ an \ATL-formula. The \defstyle{unbounded evaluation~game} 
%for $\mathcal{M}$, $\qzero$ and $\varphi$}, 
$\evalgame(\mathcal{M}, \qzero, \varphi)$ 
between the players \defstyle{\bf A} and \defstyle{\bf E}  is defined as follows.
%roles for the players: the \defstyle{verifier} \vrf and the \defstyle{falsifier} \fals,
%
\begin{itemize}[leftmargin=*]
\item A \defstyle{position} of the game is a tuple $\pos\!=\!({\pl},q,\psi)$ where 
${\pl}\in\{\text{\bf A},\text{\bf E}\}$, 
%${\pl}\in\{\vrf,\fals\}$, 
$q\in\St$ and $\psi$ is a subformula of $\varphi$. 
The \defstyle{initial position} of the game is $\pos_0:=(\text{\bf E},\qzero,\varphi)$.
\item 
In every position $({\pl},q,\psi)$, the player
\pl\ is called the \defstyle{verifier} 
and \opl\ the \defstyle{falsifier} for that position.
\item Each position of the game is associated with a  rule.
The rules for positions where the related formula is
either a proposition symbol or has a Boolean connective as its
main connective, are defined as follows.
\end{itemize}
\begin{enumerate}[leftmargin=*]
\item If $\pos_i=({\pl},q,p)$, where $p\in\Prop$, 
then $\pos_i$ is called an \defstyle{ending position} of the evaluation game.
If $q\in v(p)$, then ${\pl}$ wins the evaluation game. Else ${\opl}$ wins.
%
%This called the \defstyle{ending position}.
%
\item Let $\pos_i=({\pl},q,\neg\psi)$. 
The game then moves to the next position $\pos_{i+1}=({\opl},q,\psi)$.
\item Let $\pos_i=({\pl},q,\psi\vee\theta)$. 
Then the player ${\pl}$ decides whether the next position $\pos_{i+1}$ is $({\pl},q,\psi)$ or $\pos_{i+1}=({\pl},q,\theta)$.
\end{enumerate}
%
%
%\medskip

In order to deal with the strategic
operators, we now define a \defstyle{one step game},
denoted by $\step({\pl},A,q)$, where $A\subseteq\Agt$. This game consists of the following two actions.
\begin{enumerate}[leftmargin=30pt]
\item[i)] First ${\pl}$ chooses an action $\alpha_i\in d(i,q)$ for each $i\in A$. 
\item[ii)] Then ${\opl}$ chooses an action $\alpha_i\in d(i,q)$ for each $i\in\overline{A}$.
\end{enumerate}
The \defstyle{resulting state} of the one step game $\step({\pl},A,q)$ is the  state 
$$q':=o(q,\alpha_1,\dots,\alpha_k)$$ arising from the combined action of the agents.
We now use the one step game to define how the evaluation game proceeds from 
positions where the formula is of type $\coop{A} \X\psi$: 

%denoted by $\out(\step({\pl},A,q))$. 
%
\begin{enumerate}[leftmargin=*]
\item[(4)] Let $\pos_i=({\pl},q,\coop{A} \X\psi)$.
The next position $\pos_{i+1}$ is $({\pl},q',\psi)$, where $q'$ is the resulting state of $\step({\pl},A,q)$.
\end{enumerate}

The rules for the other strategic operators are obtained by iterating the one step game.
For this purpose, we now define the \defstyle{embedded game}
$$\mathbf{G} := \embgame(\vrf,\ctr,A,\initst,\psi_{\ctr},\psi_{\octr}),$$
where both $\vrf, \ctr \in \{{\bf E},{\bf A}\}$,
$A$ is a coalition, $\initst$ a state,
and $\psi_{\ctr}$ and $\psi_{\octr}$ are formulae.
The player $\vrf$ is called the \defstyle{verifier}
(of the embedded game) and $\ctr$ the \defstyle{controller}.
These players may, but need not be, the same.
We let $\ovrf$ and $\octr$  denote the opponents of $\vrf$ and $\ctr$, respectively.

The embedded game $\mathbf{G}$ starts from the \defstyle{initial state} $\initst$ and proceeds
from any state $q$ according to the following rules, applied in the order given below,
until an \defstyle{exit position} is reached.

\begin{enumerate}[leftmargin=25pt]
\item[i)]
$\ctr$  may end the game at the exit position $(\vrf,q,\psi_{\ctr})$.
\item[ii)] $\octr$
%\item[ii] The other player $\octr\in \{{\bf E},{\bf A}\}\setminus\{\ctr\}$
%
may end the game at the exit position $(\vrf,q,\psi_{\octr})$.
\item[iii)] If the game has not ended by the above rules, the one step game $\step(\vrf,A,q)$ is played to produce a resulting state $q'$. The embedded game is continued from the state $q'$. 
\end{enumerate}

If the \emph{embedded} game $\mathbf{G}$ continues an infinite number of rounds, the controller $\ctr$ loses the entire \emph{evaluation} game $\evalgame(\mathcal{M},\qzero,\varphi)$. 
Else, the evaluation game resumes from the exit position of the embedded game.
%(Note indeed that the evaluation game always continues after the end of an embedded game.)

\medskip

We now define the rules of the evaluation game for the  remaining
strategic operators as follows:
\begin{enumerate}[leftmargin=*]
\item[(5)] Let $\pos_i=({\pl},q,\coop{A}\psi \U\theta)$. 
The next position $\pos_{i+1}$ is the exit position of the embedded game $\embgame({\pl},{\pl},A,q,\theta,\psi)$. (Note the order of the formulae $\theta$ and $\psi$.)

\item[(6)] Let $\pos_i=({\pl},q,\coop{A}\psi \Rl\theta)$. 
The next position $\pos_{i+1}$ is the exit position of the embedded game $\embgame({\pl},{\opl},A,q,\theta,\psi)$.
\end{enumerate}
This completes the definition of the evaluation game. 
\end{definition}

\begin{remark}\label{rem: evaluation games v.s. embedded games}
The evaluation games we have defined make use of special subgames---the embedded games. The reason for distinguishing these subgames from the evaluation games is purely technical and could be avoided. The distinction will help us organize and simplify proofs later on in the paper. The most natural way of thinking of the 
evaluation games and the embedded subgames is that they simply form a 
single game used for evaluating formulae of \ATL, and the separation of 
the embedded games matters only in our reasoning \emph{about} the evaluation games, not in \emph{using} these games for evaluating formulae.
\end{remark}

We can say that the embedded game for a formula $\coop{A}\psi \U\theta$ is an
\defstyle{eventuality game} and the embedded game for $\coop{A}\psi \Rl\theta$ a
\defstyle{safety game}.
%
%
%
%Some explanations.
%
%
%
Intuitively, the embedded game
$\embgame(\vrf,\ctr,A,q,\psi_{\ctr},\psi_{\octr})$
can be seen as a `simultaneous reachability game' where both players have a
goal they are trying to reach before the opponent reaches her/his goal.
The verifier $\vrf$
leads the coalition $A$ and the falsifier $\ovrf$
leads the opposing coalition $\overline A$.
The goal of each of $\vrf$ and $\ovrf$ is defined by a formula. 
When $\vrf = \ctr$, the goal of $\vrf$ is to \emph{verify} $\psi_{\ctr}$ and the goal of $\ovrf$ is to
\emph{falsify} $\psi_{\octr}$, where verifying $\psi_{\ctr}$ corresponds to 
reaching a state where $\psi_{\ctr}$ holds and falsifying $\psi_{\octr}$ corresponds to
reaching the complement of the set of states where $\psi_{\octr}$ holds.
When $\vrf=\octr$,
the goal of $\vrf$ is to verify $\psi_{\overline{\ctr}}$ and
that of $\ovrf$ is to falsify $\psi_{\ctr}$.
Both players $\vrf$ and $\ovrf$ have
the possibility to end the game when they believe that they have reached their goal. However,
the controller is responsible for ending the embedded game in finite time, and (s)he will lose if
the game continues infinitely long.
If both players reach their targets at the same time, the controller $\ctr$ has the priority to end the embedded game first.

Ending the embedded game can be seen as ``making a claim'' which would verify or falsify the temporal formula being evaluated in the embedded game. For example, in an embedded game for $\coop{A}\psi\U\theta$, the verifier $\vrf$ may claim at any state $q$ that $\theta$ is true at $q$ and thus $\psi\U\theta$ has been verified. The opponent $\ovrf$ can challenge this claim, and thus the claim is evaluated by continuing the evaluation game from the exit position $(\vrf,q,\theta)$. Similarly, $\ovrf$ may claim that $\psi$ \emph{is not true} at $q$ and thus $\psi\U\theta$ has been falsified. Then $\vrf$ must defend the truth of $\psi$ from the exit position $(\vrf,q,\psi)$.

It is worth noting that, even though the coalitions in \ATL operate concurrently, the verifier $\vrf$ in the embedded is forced to make the choices for her/his coalition in every round \emph{before} the falsifier $\ovrf$, and thus $\ovrf$ has the advantage of making her/his choices after $\vrf$ has made hers/his.
Therefore the evaluation games are fully \emph{turn-based} (which can be beneficial in some technical contexts). 
Let us consider how this is reflected in the standard compositional semantics of \ATL (Def~\ref{def: compositional semantics for ATL}).
If the formula $\coop{A}\Phi$ is true, then by compositional semantics $A$ has a collective strategy $S_A$ to enforce $\Phi$ regardless of the actions chosen by $\overline{A}$. Hence the strategy $S_A$ will work even if the the opposing coalition $\overline{A}$ could choose their actions after seeing the actions chosen by $A$. If $\coop{A}\Phi$ is \emph{not true}, then by compositional semantics such a strategy $S_A$ simply does not exist. However, in the embedded game, it is the ``responsibility'' of the falsifier $\ovrf$ to create a path, where $\Phi$ does not hold, by leading the opposing coalition $\overline{A}$ and being able to choose actions for it after seeing those chosen for $A$.

If we compare the intuitive properties of our evaluation game with those of the standard compositional semantics, we notice that the strategic abilities of coalitions $A\subseteq\Agt$ are analysed from a different perspective. The standard compositional semantics treats strategies explicitly by quantifying over collective strategies $S_A$. In contrast, the verifier and falsifier simply play the evaluation game step by step by controlling the coalitions $A$ and $\overline A$, respectively, and the verifier does not have to announce any strategy for $A$.
Still, in \GTS the strategies need to be treated explicitly, too, when we consider which of the players has a winning strategy in the evaluation game.

Lastly, in the compositional semantics, temporal formulae are always evaluated on infinite paths, while in the evaluation game a path is constructed only up to the extent needed for verifying (or falsifying) a formula. 
In the unbounded evaluation game, infinite paths may still be constructed, but only if the controller continues the embedded game for infinitely many rounds (and thus loses the game).

In the next subsection we define a \emph{bounded} version of the evaluation game in which always \emph{only finite} paths need to be constructed.

\subsection{Bounded evaluation games}
The difference between bounded and unbounded
evaluation games is that in the bounded case, embedded games
are associated with a \emph{time limit}. In the beginning of a bounded evaluation game,
the controller must announce some, possibly infinite, ordinal $\gamma$ which will decrease in each round. This will guarantee that the embedded game (and, in fact, the 
entire evaluation game) will end after a finite number of rounds.

Bounded evaluation games $\evalgame(\mathcal{M},\qzero,\varphi,\tlimbound)$
have an additional parameter $\tlimbound$, which is an ordinal
that fixes a strict upper bound for the ordinals
that the players can announce during the related embedded games.
As we will see further, different values of $\tlimbound$ give rise to different 
evaluation games and thus lead to different game-theoretic semantics.

\begin{definition}\label{def: Bounded evaluation game}
Let $\mathcal{M}$ be a \CGM, $\qzero\!\in\!\St$, $\varphi$ an \ATL-formula and $\tlimbound$ an ordinal. 
The \defstyle{bounded evaluation~game} 
(or $\tlimbound$-bounded evaluation game)
$\evalgame(\mathcal{M}, \qzero, \varphi,\tlimbound)$  is defined in 
the same way as the unbounded evaluation
game $\evalgame(\mathcal{M}, \qzero, \varphi)$, the 
only difference between the two games being the
treatment of until- and release-formulae. 
In the bounded case, these formulae  are treated as follows.

Let 
$
	\mathbf{G} = \embgame(\vrf,\ctr,A,\initst,\psi_{\ctr},\psi_{\octr})
$
be an embedded game that arises from a position $\pos$ in $\evalgame(\mathcal{M}, \qzero, \varphi)$.
In the same position $\pos$ in the bounded evaluation game $\evalgame(\mathcal{M}, \qzero, \varphi,\tlimbound)$, the player $\ctr$ first chooses some ordinal $\initlim<\tlimbound$
to be the \defstyle{initial time limit} for the embedded game $\mathbf{G}$.
This choice gives rise to a \defstyle{bounded embedded game} that is denoted by $\mathbf{G}[\initlim]$ and played in the way described below.

A  \defstyle{configuration} of $\mathbf{G}[\initlim]$ is a pair $(\tlim,q)$, where
$\tlim$ is a (possibly infinite) ordinal called the
\defstyle{current time limit} and $q\in \St$ a state
called the \defstyle{current state}. 
The bounded embedded game $\mathbf{G}[\initlim]$ starts from the \defstyle{initial
configuration} $(\initlim,\initst)$ and proceeds
from any configuration  $(\tlim,q)$ according to the following rules, applied in
the given order.
\begin{enumerate}[leftmargin=25pt]
\item[i)] If $\tlim=0$, the game ends at the exit position $(\vrf,q,\psi_{\ctr})$.
\item[ii)] $\ctr$ may end the game at the exit position $(\vrf,q,\psi_{\ctr})$.
\item[iii)] $\octr$ may end the game at the exit position $(\vrf,q,\psi_{\octr})$.
\item[iv)] If the game has not ended due to 
the previous rules, then $\step(\vrf,A,q)$ is played
in order to produce a resulting state $q'$.  
Then the bounded embedded game continues from the configuration $(\tlim',q')$,
where $\tlim'=\gamma\!-\!1$ if $\tlim$ is a successor ordinal,
and if $\tlim$ is a limit ordinal, then $\tlim'$ is an ordinal
smaller than $\tlim$ and chosen by \ctr.
\end{enumerate}

We denote the set of configurations in $\mathbf{G}[\initlim]$ by $\cnf_{\mathbf{G}[\initlim]}$. 
%When the initial time limit $\tlim$ is finite, we call the embedded game a   \defstyle{bounded embedded game}, while in the general case we will refer to it as the \defstyle{well-founded embedded game}. 
%
After the bounded embedded game $\mathbf{G}[\initlim]$ has reached an exit position---which it will, since ordinals are well-founded---the evaluation game resumes from the exit position of the embedded game.
\end{definition}

It is clear that bounded evaluation games end after a finite
number of rounds because bounded embedded games do.
Note, that if time limits are infinite ordinals,
they do not directly determine the number of rounds left in the game,
but instead they are related to the game duration in a more abstract way.
It is also worth noting here that different ways of using ordinals in game-theoretic considerations go way back. An example of an important and relatively early reference is \cite{cedric_1966} which contains references to even earlier related articles.
It is instructive to analyse embedded games independently of evaluation games. An embedded game of the form $\mathbf{G}=\embgame(\vrf,\ctr,A,\initst,\psi_{\ctr},\psi_{\octr})$ can be played without a time limit as in unbounded evaluation games, or it can be given some time limit $\initlim$ as a parameter, which leads to the related bounded embedded game $\mathbf{G}[\initlim]$.
When we use the plain notation $\mathbf{G}$ (as opposed to $\mathbf{G}[\gamma_0]$), we always assume that the embedded game $\mathbf{G}$ is not bounded. 
We sometimes emphasize this by calling $\mathbf{G}$ an \emph{unbounded} embedded game.
We will study the properties of embedded games in Section~\ref{sec:Analysing evaluation games}.
%
%later on, when we define \emph{canonical winning conditions} for them. 
%

%
Let $\omega$ denote the smallest infinite ordinal.
Evaluation games of the form
$$\evalgame(\mathcal{M},\qzero,\varphi,\omega)$$
constitute a particularly interesting subclass of 
bounded evaluation games. 
We call the games in this class \defstyle{finitely bounded evaluation games.}
In these games, only \emph{finite} time
limits can be announced for bounded embedded games.

Note that in Definition~\ref{def: Bounded evaluation game} we could handle successor ordinals in the same way as limit ordinals. That is, the controller should choose any ordinal $\gamma'$ smaller than the current time limit $\gamma$ also when $\gamma$ is a successor ordinal. However, it easy to see that in this case the best choice for the controller is to always choose $\gamma'=\gamma-1$ and thus this step in the game would be unnecessary. Furthermore, with \emph{finitely} bounded evaluation games, it is natural that the time limit is always lowered by one after every transition in the embedded game.

\begin{remark}
In this paper the temporal operators $\coop{A}\F$ and $\coop{A}\G$ are regarded as syntactic abbreviations, and therefore the rules associated with them can be extracted from those  
above. 
We now define alternative rules that could be directly given to $\coop{A}\F$ and $\coop{A}\G$ in the \emph{finitely bounded} evaluation games, resulting in an equivalent semantics. (The fact that this equivalence holds will ultimately be straightforward to observe.)  

\begin{itemize}[leftmargin=*]
\item Let $\pos_i=({\pl},q,\coop{A} \F\psi)$. 
First the player ${\pl}$ chooses some $n\in\mathbb{N}$ and then the players iterate $\step({\pl},A,q)$ for at most $n$ times. The player ${\pl}$ may decide to stop at the current state $q'$ after any number $m\leq n$ of iterations, and then the evaluation game is continued from $\pos_{i+1}=({\pl},q',\psi)$.
\item Let $\pos_i=({\pl},q,\coop{A} \G\psi)$.
First the player ${\opl}$ chooses some $n\in\mathbb{N}$ and then the players iterate $\step({\pl},A,q)$ for at most $n$ times. The player ${\opl}$ may decide to stop at the current state $q'$ after any number $m\leq n$ of iterations, and
the evaluation game is then continued from $\pos_{i+1}=({\pl},q',\psi)$.
\end{itemize}

Similar (but not identical) rules could also be given to $\coop{A}\F$ and $\coop{A}\G$
in the frameworks based on unbounded and on ordinal-bounded games.
\end{remark}

%\end{customexample}
%}

%%%%%%%%%%%%%%%%%%%%%%%%%%%%
\subsection{Game-theoretic semantics}

%In this section we define the notion of
%
%strategy for evaluation games. 
%
A strategy for a player $\pl\in\{\text{\bf A},\text{\bf E}\}$  will be defined below to be a function on game positions; in positions where the player $\pl$ is 
not required to make a move, the strategy of $\pl$ 
will output a special value ``$\void$", the meaning of which is `any possible value'. 
We will also use $\void$ for some other functions
when the output is not relevant, e.g., when defining a winning strategy, we may assign $\void$ for losing positions.
%
%, and sometimes we may assign any ordinal value to a timer.
%

\begin{definition}\label{def: Strategies}
Let $\mathbf{G}=\embgame(\vrf,\ctr,A,\initst,\psi_{\ctr},\psi_{\octr})$ be an embedded
game and $\pl\in\{\text{\bf A},\text{\bf E}\}$.
\defstyle{A~strategy for the player ${\pl}$ in $\mathbf{G}$}
is a function $\sigma_{\pl}$ whose domain is $\St$ and whose
range is specified below.

Firstly, for any $q\in \St$, it is possible
to define $\sigma_{\pl}(q)\in\{\psi_{\ctr},\psi_{\octr}\}$; then 
$\sigma_{\pl}$ instructs $\pl$ to end the
game at the state $q$.
Here it is required that if ${\pl}=\ctr$,
then $\sigma_{\pl}(q) =\psi_{\ctr}$ and if ${\pl}=\octr$, then $\sigma_{\pl}(q)=\psi_{\octr}$.
(Intuitively, if $\sigma_{\pl}(q)=\psi\in\{\psi_{\ctr},\psi_{\octr}\}$, then the strategy $\sigma_\pl$ instructs $\pl$ to claim that $\psi$ is true at $q$.)

If $\sigma_{\pl}(q)\not\in\{\psi_{\ctr},\psi_{\octr}\}$, then the following conditions hold.
\begin{itemize}[leftmargin=*]
\item If ${\pl}=\vrf$, then $\sigma_{\pl}(q)$ is a tuple of actions in $\avact(A,q)$ (to be chosen for $A$).
\item If ${\pl}=\ovrf$, then $\sigma_{\pl}(q)$ is
defined to be a \defstyle{response function} $f:\avact(A,q)\rightarrow
\avact(\overline A,q)$ that assigns a tuple of
actions for $\overline A$ as a response to any tuple of actions chosen for $A$ (by the opposing player).
% 
%(aka \defstyle{co-move} in \cite{GorDrim06})), 
\end{itemize}

Let $\initlim$ be an ordinal. A strategy $\sigma_\pl$ for ${\pl}$ in $\mathbf{G}[\initlim]$ is defined in the same way as a strategy in $\mathbf{G}$, but the domain of
this strategy is the set of all possible configurations $\cnf_{\mathbf{G[\initlim]}}$. 
\end{definition}

Note that strategies in embedded games are positional, i.e., 
they depend only on the current state in the unbounded case and on the
current configuration in the bounded case.
We will see later that if strategies were allowed to depend on more
information, such as the sequence of  states played,
the resulting semantic systems would be equivalent to the current ones.

Any strategy $\sigma_\pl$ for an unbounded embedded game $\mathbf{G}$ can also be used in any bounded embedded game $\mathbf{G}[\initlim]$: we simply use the same action $\sigma_\pl(q)$ for each configuration $(\tlim,q)\in\cnf_{\mathbf{G[\initlim]}}$. 
In general, this does not work the other way around, since a strategy $\sigma_\pl$ that is defined for configurations may give different values for $(\tlim,q)$ and $(\tlim',q)$.
However, if a strategy $\sigma_\pl$ for a bounded embedded game $\mathbf{G}[\initlim]$ is \emph{independent} of time limits (and thus depends on states only), it can also be used in the unbounded embedded game $\mathbf{G}$.
This observation will be crucial later when we compare bounded embedded games with the corresponding unbounded embedded games.

We next define the notion of strategy for evaluation games, using strategies for embedded games as sub-strategies.
We first present the definition for \emph{unbounded} evaluation games.
\begin{definition}\label{def: Strategies in unbounded evaluation games}
Let ${\pl}\in\{\text{\bf A},\text{\bf E}\}$. A \defstyle{strategy for player ${\pl}$ in an
unbounded evaluation game $\evalgame = \evalgame(\mathcal{M}, \qzero, \varphi)$}
is a function $\strev_{{\pl}}$ defined on the set of positions $\setpos$ of
$\evalgame$ (with the range specified below) satisfying the following conditions.
\begin{enumerate}[leftmargin=*]
\item If $\pos=({\pl},q,\psi\vee\theta)$, then $\strev_{{\pl}}(\pos)\in\{\psi,\theta\}$.
\item If $\pos=({\pl},q,\coop{A} \X\psi)$, then $\strev_{{\pl}}(\pos)$ is a tuple of
actions in $\avact(A,q)$ for the one step game $\step({\pl},A,q)$. 
\item If $\pos=({\opl},q,\coop{A} \X\psi)$, then $\strev_{{\pl}}(\pos)$ is a
response function $f:\avact(A,q)\rightarrow\avact(\overline{A},q)$ for the one step game $\step({\pl},A,q)$. 
\item Let $\pos=({\pl},q,\coop{A}\psi \mathsf{T}\theta)$ or $\pos=({\opl},q,\coop{A}\psi \mathsf{T}\theta)$, where $\mathsf{T}\in\{\U\!,\!\Rl\}$. Then $\strev_{\pl}(\pos)$ is a strategy $\sigma_\pl$ for $\pl$ in the respective embedded game
$\embgame(\vrf,\ctr,A,q,\theta,\psi)$. 
\item In all other cases, $\strev_\pl(\pos)=\void$ (the player $\pl$ is not required to make a move).
\end{enumerate}

We say that the player ${\pl}$ plays according to
the strategy $\strev_{{\pl}}$ in the evaluation
game $\evalgame$ if ${\pl}$ makes her/his 
choices in $\evalgame$ according to that strategy.
We then say that $\strev_{{\pl}}$ is a
\defstyle{winning strategy} for ${\pl}$ in $\evalgame$ if ${\pl}$
wins all plays of $\evalgame$
where (s)he plays according to that strategy.
\end{definition}

We now define strategies for \emph{bounded} evaluation games. This definition differs from the previous one in positions that lead to an embedded game. Here one of the players must announce a time limit for the embedded game. The player announcing the time limit must also decide how to lower the time limit when a limit ordinal is reached: this will be done by a
substrategy, called a  \emph{timer}. 
Strategies of the other player player may depend on the announced time limits.
\begin{definition}[c.f. Remark~\ref{On genuine positionality}]\label{def: Strategies in bounded evaluation games}
A \defstyle{strategy for player ${\pl}$ in a bounded evaluation game $\evalgame =
\evalgame(\mathcal{M}, \qzero, \varphi,\tlimbound)$} is defined as in Definition~\ref{def:
Strategies in unbounded evaluation games},
with the exception of positions with until- and release-formulae, which are
treated as follows.
\begin{enumerate}[leftmargin=*]
\item[(4)] Let $\pos=({\pl},q,\coop{A}\psi \mathsf{T}\theta)$ or $\pos=({\opl},q,\coop{A}\psi \mathsf{T}\theta)$, where $\mathsf{T}\in\{\U,\!\Rl\}$,  and let 
$\mathbf{G}=\embgame(\vrf,\ctr,A,q,\theta,\psi)$
denote the embedded game related to $\pos$.

If ${\pl}=\ctr$, then $\strev_{{\pl}}(\pos)=(\initlim,t,\sigma_{\pl})$ such that 
the following conditions hold.
\begin{itemize}[leftmargin=*]
\item
$\initlim<\tlimbound$ is an ordinal. It is the choice for the
initial time limit and leads to the bounded embedded game $\mathbf{G}[\initlim]$.
\item
$t$ is a function, called \defstyle{timer}, on pairs $(\tlim,q)$, where $\tlim\leq\initlim$ is a limit
ordinal and $q\in\St$. The value $t(\tlim,q)$ must be an ordinal less than $\tlim$.
The timer $t$ gives an instruction
how to lower the time limit $\tlim$
\emph{after} a transition to $q$ has been made.
\item
$\sigma_{\pl}$ is a strategy for ${\pl}$ in $\mathbf{G}[\gamma_0]$.
\end{itemize}
Finally, if
${\pl}\not=\ctr$, then $\strev_{{\pl}}(\pos)$ is a function that maps any ordinal $\initlim<\tlimbound$ to some strategy $\sigma_{\pl,\initlim}$ for ${\pl}$ in $\mathbf{G}[\initlim]$.
\end{enumerate}
\end{definition}
In finitely bounded evaluation games, only finite time limits $\initlim<\omega$ may be announced by the controller $\mathbf{C}$.
Since no limit ordinal can be reached, the timer $t$ can be omitted from the strategy in this special case.

\begin{remark}
As discussed in Remark~\ref{rem: evaluation games v.s. embedded games}, instead of considering embedded games as separate subgames of the evaluation game, they can be  merged into the evaluation game. Thus when considering the full evaluation game (including the embedded games), it is not necessary to make a distinction between positions and configurations. In this approach we could only consider the strategy $\Sigma_\pl$ for the evaluation game instead of having substrategies $\sigma_\pl$ for related embedded games. Moreover, in the bounded case, the timer $t$ is given as a separate component together with $\sigma_\pl$, but it could perhaps be more natural to make it a part of $\sigma_\pl$.

However, this separation of evaluation games and embedded games, as well as the corresponding strategies, allows us to analyse the properties of embedded games in a modular fashion. The separation will also be very useful when comparing \GTS with the standard compositional semantics since the verifier's strategies $\sigma_\vrf$ (without the timer) will then be very closely related to the collective strategies $S_A$ (recall Definition~\ref{def: compositional semantics for ATL}).
\end{remark}

\begin{remark}\label{On genuine positionality}
Note that in Definition~\ref{def: Strategies in bounded evaluation games} we allow, for technical convenience, the strategies for an embedded game $\embgame(\vrf,\ctr,A,\initst,\psi_{\ctr},\psi_{\octr})[\initlim]$ to depend on all parameters of the game; in addition to the configurations $(q,\tlim)$ that can occur in the embedded game. Naturally the parameters $\vrf$, $\ctr$, $A$, $\psi_{\ctr}$ and $\psi_{\octr}$ are all crucial information to the players, but for ``genuinely positional'' strategies, the \emph{initial} state $\initst$ and the \emph{initial} time limit $\initlim$ should only be available to a player when the game is in the initial configuration $(\initst,\initlim)$. However, we can show that $\initst$ and $\initlim$ are indeed unnecessary bits of information for the players when the game is not in the initial configuration. This follows from Proposition~\ref{the: Canonical strategy} which we present later.
\end{remark}

Different choices for time limit bounds $\tlimbound$ give rise to different semantic systems,  and most results in the next section will be proven for an arbitrary choice of $\tlimbound$.
However, in this paper we mainly focus on the cases $\tlimbound=\omega$ 
and $\tlimbound=2^\kappa$, where $\kappa$ is the cardinality of the model. 
We will prove later that time limit bounds greater than $2^\kappa$
are not needed, 
and in finite models time limit bounds of the size $\kappa$ suffice.

\begin{definition}
\label{def:game-theoretic semantics} 
Let $\mathcal{M}$
 %= (\Agt, \St, \Prop, \Act, d, o, v)$ 
 be a \CGM, $q\in\St$ and $\varphi$ an \ATL-formula. Let $\kappa$ be the cardinality of the model $\mathcal{M}$.
We define three different notions of truth of $\varphi$ in $\mathcal{M}$ and $q$, 
based on the three different evaluation games, thereby defining the
\defstyle{unbounded}, \defstyle{bounded} and \defstyle{finitely bounded semantics}, denoted respectively by $\models_{u}^{g}$, $\models_{b}^{g}$, and $\models_{f}^{g}$, as follows.
\begin{itemize}[leftmargin=*]
\item $\mathcal{M},q\models_{u}^{g}\varphi$
 \  iff \ {\bf E} has a winning strategy in $\evalgame(\mathcal{M}, q, \varphi)$. 
\item $\mathcal{M},q\models_{b}^{g}\varphi$
 \ iff \ {\bf E} has a winning strategy in $\evalgame(\mathcal{M}, q, \varphi,2^\kappa)$. 
\item $\mathcal{M},q\models_{f}^{g}\varphi$
 \ iff \ {\bf E} has a winning strategy in $\evalgame(\mathcal{M}, q, \varphi,\omega)$. 
\end{itemize}

We also write more generally that
$\mathcal{M},q\models_{\tlimbound}^{g}\!\varphi$
  iff  {\bf E} has a winning strategy in $\evalgame(\mathcal{M}, q, \varphi,\tlimbound)$.
This is called \defstyle{$\tlimbound$-bounded semantics}.
\end{definition}

Note that in the \emph{bounded} \GTS the time limit bound $\tlimbound$ is fixed based on the given model. On the other hand, $\tlimbound$-bounded \GTS has a fixed value of $\tlimbound$ that is used for every model (this kind of semantics may also be called \emph{uniformly bounded}.)

We will prove that both  the bounded and the unbounded semantics are 
equivalent to the standard compositional semantics 
of Definition \ref{def: compositional semantics for ATL}.
The finitely bounded semantics, on the other hand, will turn out non-equivalent to these, but  equivalent to a natural variant of the compositional semantics, introduced in Section \ref{sec:Comparing game-theoretic and compositional semantics}.
We will also see that each $\tlimbound$-bounded \GTS (for an arbitrary \emph{fixed} value of $\tlimbound$) is non-equivalent to the standard compositional semantics.
The following example shows that the finitely bounded \GTS differs from both the unbounded and bounded cases. In particular, the fixed point characterisation of the temporal operator $\F$, viz. 
$\coop{A} \F p \equiv p \lor \coop{A} \X\coop{A} \F p$, fails with finitely bounded \GTS, as seen by the example.
\begin{example}\label{ex: bounded vs. finitely bounded}
Consider the \CGM  $\mathcal{M}\!=\!(\{a\},\{q_0\}\cup\mathbb{N}
\times\mathbb{N},\{p\},\mathbb{N},d,o, v)$ in Figure~1, 
where
$v(p) = \{(i,i)\mid i\in\mathbb{N}\}$, 
$ d(a,q_0)=\mathbb{N}$, $d(a,(i,j))=\{0\}$, 
$o(q_0,i)=(i,0)$ and $o((i,j),0)=(i,j\!+\!1)$.

In this model
$\mathcal{M},q_0\not\models_f^g \coop{\emptyset} \F p$, because for every time limit $n<\omega$ chosen by Eloise, Abelard
may select the action $n$ in the first round for the agent $a$, so it 
will take $n\!+\!1$ rounds to reach a state where $p$ is true. 
On the other hand, $\mathcal{M},q_0\models_f^g \coop{\emptyset} \X\coop{\emptyset} \F p$ and therefore $\mathcal{M},q_0\models_f^g p \lor \coop{\emptyset} \X\coop{\emptyset} \F p$, because, after the first step the game
will be at a state $(i,0)$ for some $i\in\mathbb{N}$, whence Eloise can choose
any time limit $n\geq i$ and reach a state where $p$ is true before time runs out.

However, $\mathcal{M},q_0\models_b^g \coop{\emptyset} \F p$, since Eloise can choose $\omega$ as the time limit in the beginning of the game and then lower it to $i<\omega$ when the next state $(i,0)$ is reached. Also, $\mathcal{M},q_0\models_u^g \coop{\emptyset} \F p$ since a state where $p$ is true will always be reached in a finite number of steps. 

Despite these observations, we will show that the three semantics become equivalent over image-finite models.

\begin{figure}[h]
\begin{center}
\begin{tikzpicture}
	[
	place/.style={text width=0.5cm, align=flush center,
	circle,draw=black!100,fill=black!0,thick,inner sep=2pt,minimum size=4mm},scale=0.87]
	\node at (-2.1,2.8) {$\mathcal{M}$:};
%	\node at (5.7,2.6) (Agt) {$\Agt=\{a\}$};
%	\node at (6.2,2.1) (Act) {$\mathit{act}(q_0,a)\,=\,\mathbb{Z}_{+}$};
	\node at (-1.5,0) (00) [place] {$\neg p$};
	\node at (-2.5,0)  {$ q_0$};
	\node at (1,4) (41) [place] {$\neg p$};
	\node at (1,2) (31) [place] {$\neg p$};
	\node at (1,0) (21) [place] {$\neg p$};
	\node at (1,-2) (11) [place,fill=red!50] {$ p$};
	\node at (4,-2) (19) {$\cdots$};  		
	\draw [-latex,thick] (11) to node {} (19);	
	%
	%\draw [dashed,thick] (4,-3.2) to (4,-4.4);
	\draw [-latex,thick] (00) to node [below] {$ 0$\;\;} (11);
	\draw [-latex,thick] (00) to node [above] {$ 1$\;\;} (21);
	\draw [-latex,thick] (00) to node [above] {$ 2$\;\;} (31);
	\draw [-latex,thick] (00) to node [above] {$ 3$\;\;} (41);
	%\draw [-latex,thick] (00) to (-0.35,-1.25);
	\draw [dashed,thick] (-0.35,3) to (-0.35,4.2);
	\node at (2.5,4) (42) [place] {$ \neg p$};
	\node at (2.5,2) (32) [place] {$\neg p$};
	\node at (2.5,0) (22) [place,fill=red!50] {$ p$};
	\node at (5.5,0) (29) {$\cdots$};  		
	\draw [-latex,thick] (22) to node {} (29);	
	\node at (2.5,-2) (12) [place] {$\neg p$};
	\draw [-latex,thick] (11) to node [above] {$ 0$} (12);
	\draw [-latex,thick] (21) to node [above] {$ 0$} (22);
	\draw [-latex,thick] (31) to node [above] {$ 0$} (32);
	\draw [-latex,thick] (41) to node [above] {$ 0$} (42);
	\node at (4,4)  	(43) [place] {$ \neg p$};
	\node at (4,2)       (33) [place,fill=red!50] {$ p$}; 
	\node at (7,2) (39) {$\cdots$};  		
	\draw [-latex,thick] (33) to node {} (39);	
	\node at (4,0) 		(23) [place] {$\neg p$};
	\draw [-latex,thick] (22) to node [above] {$ 0$} (23);
	\draw [-latex,thick] (32) to node [above] {$ 0$} (33);
	\draw [-latex,thick] (42) to node [above] {$ 0$} (43);
	\node at (5.5,4) (44) [place,fill=red!50] {$ p$};
	\node at (5.5,2) (34) [place] {$\neg p$};
	\node at (8.5,4) (49) {$\cdots$};  		
	\draw [-latex,thick] (44) to node {} (49);	
	\draw [-latex,thick] (33) to node [above] {$ 0$} (34);
	\draw [-latex,thick] (43) to node [above] {$ 0$} (44);
	\node at (7,4) (45) [place] {$ \neg p$};
	\draw [-latex,thick] (44) to node [above] {$ 0$} (45);
	%\draw [dashed,thick] (7.5,-1.8) to (8.3,-2.3);
\end{tikzpicture}
\label{fig2}
\caption{$\coop{A} \F p \equiv p \lor \coop{A} \X\coop{A} \F p$ 
%fixed point characterisation of $\F$ 
fails in the finitely bounded \GTS.} 
\end{center}
\end{figure}
\end{example}
%

%
%%%%%%%%%%%%%%%%%%%%%%%%%%%%
%%%%%%%%%%%%%%%%%%%%%%%%%%%%
\section{Analysing embedded games}
\label{sec:Analysing evaluation games}
%%%%%%%%%%%%%%%%%%%%%%%%%%%%
%%%%%%%%%%%%%%%%%%%%%%%%%%%%
%

In this section we will examine the properties of different versions of embedded games that occur as parts of evaluation games. We associate with each state a \emph{winning time label} which describes how good that state is for the players. The optimal labels will be used to define a \emph{canonical strategy} which will be a winning strategy whenever there exists one. 
With these notions we shall prove positional determinacy of the embedded games. We will also show that if the players are allowed to announce sufficiently large ordinals as time limits, then bounded embedded games become essentially equivalent to corresponding unbounded embedded games.
Furthermore, we shall analyse how the sizes of the needed
ordinals depend on the \CGM in which the game is played.

%
%%%%%%%%%%%%%%%%%%%%%
\subsection{Winning time labels}

Different values of the time limit bound $\tlimbound$ correspond to different classes of bounded embedded games $\mathbf{G}[\initlim]$ where $\initlim<\tlimbound$. In this section we use a fixed (unless otherwise specified) value of $\tlimbound$ and will assume that all bounded embedded games are part of some evaluation game $\evalgame(\mathcal{M},\qzero,\varphi,\tlimbound)$. 
Since $\tlimbound$ could have any ordinal value, our results will hold, in particular, for both $2^\kappa$-bounded and finitely bounded semantics.

If $\mathbf{G}=\embgame(\vrf,\ctr,A,\initst,\psi_{\ctr},\psi_{\octr})$ is an embedded game and $q\in\St$, we write $\mathbf{G}[q]:=\embgame(\vrf,\ctr,A,q,\psi_{\ctr},\psi_{\octr})$. We also use the abbreviation $\mathbf{G}[q,\tlim]:=(\mathbf{G}[q])[\tlim]$. This notation is useful because, by the recursive nature of bounded embedded games, any configuration $(\tlim,q)$ of
$\mathbf{G}[\initlim]$ (where $\initlim<\tlimbound$) is the initial configuration of $\mathbf{G}[q,\tlim]$. 
%
%Note that since we require the players to use positional strategies, they do not ``see'' any difference between initial configurations and other configurations.
%

%
We next define winning strategies for embedded games. Here ``winning an embedded game'' intuitively means having a winning strategy in the \emph{evaluation game} that continues from the exit position of the embedded game.

\begin{definition}\label{def: Winning strategies in embedded game}
%
%Consider unbounded games. 
Let
$\mathbf{G}=\embgame(\vrf,\ctr,A,\initst,\psi_{\ctr},\psi_{\octr})$
be an embedded game. 

\begin{enumerate}[leftmargin=*]
\item We say that $\sigma_{{\pl}}$ is a \defstyle{winning strategy for the player $\pl$ in {\bf G}} if the following conditions hold:
\begin{enumerate}
\item[a)] infinite plays are possible with $\sigma_{\pl}$ only if ${\pl}\neq\ctr$.
\item[b)] the equivalence $\mathcal{M},q\models_u^g\atlb\Leftrightarrow{\pl} = \vrf$
holds for all exit positions $(\vrf,q,\psi)$ of $\mathbf{G}$ that
can be reached when ${\pl}$ plays using $\sigma_{\pl}$ . 
\end{enumerate}
\item Let $\initlim<\tlimbound$.
\begin{itemize}[leftmargin=*]
\item Suppose that $\pl=\ctr$. We say that the pair $(\sigma_{{\pl}},t)$ is a \defstyle{timed winning strategy for ${\pl}$ in $\mathbf{G}[\initlim]$} if the equivalence $\mathcal{M},q\models_\tlimbound^g\atlb\Leftrightarrow {\pl} = \vrf$ holds for all exit positions $(\vrf,q,\atlb)$ that can be encountered when ${\pl}$ plays using the strategy $\sigma_{\pl}$ and the timer $t$.
\item Suppose that $\pl\neq\ctr$. We say that $\sigma_{{\pl}}$ is a \defstyle{winning strategy for $\pl$ in $\mathbf{G}[\initlim]$} if the equivalence $\mathcal{M},q\models_\tlimbound^g\atlb\Leftrightarrow {\pl} = \vrf$ holds for all exit positions $(\vrf,q,\atlb)$ that  can occur when ${\pl}$ plays using $\sigma_{\pl}$.
\end{itemize}
\end{enumerate}
If the unbounded (respectively, bounded) embedded game ends in a position
where the equivalence $\mathcal{M},q\models_u^g\atlb
\Leftrightarrow {\pl} = \vrf$ 
(respectively, $\mathcal{M},q\models_\tlimbound^g\atlb
\Leftrightarrow {\pl} = \vrf$) holds,
we say that ${\pl}$ \defstyle{wins} the embedded game. 
In the unbounded case, $\octr$ also wins in the case where the play is infinite.
(Note that it now follows that, in every play of the embedded game, exactly one of the players wins.)
\end{definition}
We next define for an embedded game $\mathbf{G}=\embgame(\vrf,\ctr,A,\initst,\psi_{\ctr},\psi_{\octr})$ so called \defstyle{winning time labels}, $\mathcal{L}_{\pl}(q)$, for each $q\in\St$. The labels will indicate how good the state $q$ is for the player ${\pl}$ when different bounded embedded games $\mathbf{G}[q,\initlim]$ are played with different time limits $\initlim<\tlimbound$. If the label is ``\win"\ or ``\lose", then the state is a winning 
(respectively, losing) state for ${\pl}$, regardless of the time limit $\initlim$. 
If the label is an ordinal $\tlim<\tlimbound$, it means that $\tlim$ is the ``critical time limit'' for winning or losing the game. More precisely, if ${\pl}=\ctr$, $\tlim$ is the least time limit needed for $\pl$ to win from $q$, and if ${\pl}\neq\ctr$, then $\tlim$ is the least time limit such that ${\pl}$ can no longer guarantee that (s)he will not lose the game from $q$.

From now on we will often consider
separately the cases where the player ${\pl}$ is the
controlling player $\ctr$ and where her/his opponent ${\opl}$ is the
controlling player.
The former case correspond to the
situation where ${\pl}$ is the verifier in an
eventuality game or the falsifier in a safety game.
The latter case means that
${\pl}$ is the verifier in a safety game or is the
falsifier in an eventuality game.

\begin{definition}\label{def: Winning time labels}
Let $\mathbf{G} = 
\embgame(\vrf,\ctr,A,\initst,\psi_{\ctr},\psi_{\octr})$
be an embedded game and ${\pl}\in\{{\bf A},{\bf E}\}$.
The \defstyle{winning time label
 $\mathcal{L}_{\pl}(q)$ for ${\pl}$ in $\mathbf{G}$ at state $q\in\St$} is
defined as follows.

\medskip

\noindent
\textbf{Case 1}. Suppose  ${\pl}=\ctr$. Let $\sigma_{\pl}$ be a strategy for ${\pl}$. We
first define a \defstyle{strategy label} $l(q,\sigma_{\pl})$ as follows.
\begin{itemize}[leftmargin=*]
\item Set $l(q,\sigma_{\pl}):=\lose$ if $(\sigma_{\pl},t)$ is
not a timed winning strategy in $\mathbf{G}[q,\tlim]$
for any timer $t$ and $\tlim<\tlimbound$.
\item Else, set $l(q,\sigma_{\pl}) :=\tlim$, where $\tlim<\tlimbound$ is the least time limit for which there is a timer t such that $(\sigma_{\pl},t)$ is a timed winning strategy in $\mathbf{G}[q,\tlim]$.
\end{itemize}
When there exists at least one strategy $\sigma_{\pl}$ for $\pl$ such that $l(q,\sigma_{\pl})\neq \lose$, we define $\mathcal{L}_{\pl}(q)$ as the least ordinal value of strategy labels $l(q,\sigma_{\pl})$.
Else, we define $\mathcal{L}_{\pl}(q):=\lose$.

\medskip

\noindent
\textbf{Case 2}. Suppose ${\pl}\neq\ctr$. Let $\sigma_{\pl}$ be a strategy for ${\pl}$. We define $l(q,\sigma_{\pl})$ as follows.

\begin{itemize}[leftmargin=*]
\item 
Set $l(q,\sigma_{\pl}):=\win$ if $\sigma_{\pl}$ is a timed winning strategy in $\mathbf{G}[q,\tlim]$ for every time limit $\tlim<\tlimbound$. 
\item Else, set $l(q,\sigma_{\pl}):=\tlim$, where $\tlim<\tlimbound$ is the least time limit such that $\sigma_{\pl}$ is not a winning strategy in $\mathbf{G}[q,\tlim]$.
\end{itemize}
If $l(q,\sigma_{\pl})=\win$ for some $\sigma_{\pl}$, then set $\mathcal{L}_{\pl}(q):=\win$.
Else, set $\mathcal{L}_{\pl}(q)$ to be the least upper bound for the values $l(q,\sigma_{\pl})$.
\end{definition}
The following lemma shows that if the controller has a timed winning strategy within  some 
time limit, then (s)he has a timed winning strategy which is winning for all greater time limits as well.
This claim may seem obvious, but needs to be proven nevertheless, since there is no guarantee that a winning strategy for some smaller time limit would make good choices also at configurations with larger time limits. 
\begin{lemma}\label{the: Regular strategies}
Let 
$\mathbf{G}=\embgame(\vrf,\ctr,A,\initst,\psi_{\ctr},\psi_{\octr})$
be an embedded game. We assume that $\pl=\ctr$ and that $\pl$ has a timed winning strategy $(\sigma_\pl,t)$ in $\mathbf{G}[\initlim]$ for some $\initlim<\tlimbound$.
Then there is a pair $(\sigma_\pl',t')$ which is a timed winning strategy in $\mathbf{G}[\tlim]$ for any time limit $\tlim$ such that $\initlim\leq\tlim<\tlimbound$. 
\end{lemma}

\begin{proof}
The idea here is that we define the timed strategy $(\sigma_\pl',t')$ at a configuration $(\tlim,q)$ by using the choices given by $(\sigma_\pl,t)$ for the smallest $\tlim'\leq\tlim$ such that $(\sigma_\pl,t)$ is a winning strategy for $\mathbf{G}[q,\tlim']$.
See the appendix for a detailed proof.
\end{proof}

The following proposition relates
values of winning time labels to durations of
embedded games and existence of winning strategies.

\begin{proposition}
\label{the: Winning time labels}
Let $\mathbf{G}=\embgame(\vrf,\ctr,A,\initst,\psi_{\ctr},\psi_{\octr})$
be an embedded game, 
${\pl}\in\{{\bf A},{\bf E}\}$ and $q\in\St$.

\begin{enumerate}[leftmargin=*]
\item
If ${\pl}=\ctr$, then:  

i) $\mathcal{L}_{\pl}(q)=\tlim<\tlimbound$ iff there is a pair $(\sigma_\pl,t)$ that is a timed winning strategy in $\mathbf{G}[q,\tlim']$ for all $\tlim'$ s.t. $\tlim\leq\tlim'<\tlimbound$, but there is no timed winning strategy for $\pl$ in $\mathbf{G}[q,\tlim']$ for any $\tlim'<\tlim$.

ii)  $\mathcal{L}_{\pl}(q)=\lose$ iff there is no timed winning strategy $(\sigma_\pl,t)$ for ${\pl}$ in $\mathbf{G}[q,\tlim]$ for any $\tlim<\tlimbound$.

\item
If ${\pl}\neq\ctr$, then:  

i) $\mathcal{L}_{\pl}(q)=\tlim<\tlimbound$ iff
for every $\tlim'<\tlim$, there is some $\sigma_{\pl}$ which is a winning strategy for $\pl$ in $\mathbf{G}[q,\tlim']$, but there is no winning strategy for ${\pl}$ in $\mathbf{G}[q,\tlim']$ for any $\tlim'$ s.t. $\tlim\leq\tlim'<\tlimbound$.

ii) $\mathcal{L}_{\pl}(q)=\win$ iff there is a strategy $\sigma_\pl$ which is a winning strategy in $\mathbf{G}[q,\tlim]$ for every $\tlim<\tlimbound$.
\end{enumerate}
\end{proposition}

\begin{proof}
The proof follows quite directly from the definition of winning time labels. We also need to use Lemma~\ref{the: Regular strategies} when proving the case (1), i).
More details are given in the appendix. 
\end{proof}
Winning time labels $\mathcal{L}_\pl(q)$ of an embedded game
are either ordinals less than the time limit bound $\tlimbound$, or else, 
labels $\win$, $\lose$. 
If we increased the value of $\tlimbound$ to some $\tlimbound'>\tlimbound$
and considered the values of winning time labels of the corresponding embedded game within the evaluation game $\evalgame(\mathcal{M},\qzero,\varphi,\tlimbound')$, then some of the labels that originally were $\win$ or $\lose$, could now
obtain ordinal values $\tlim$ such that $\tlimbound\leq\tlim<\tlimbound'$.
Other kinds of changes of labels would also be possible because the ``truth sets'' of the goal formulae $\psi_{\ctr}$ and $\psi_{\octr}$ could change; see the following example.

\begin{example}
Consider the \CGM $\mathcal{M}\!=\!(\{a\},\{q_i\mid i\in \mathbb{N}\},\{p\},\{0\},d,o, v)$, where $d(q_i)=\{0\}$ for each $i$, $o(q_i,0)=q_{i+1}$ for each $i\in\mathbb{N}$ and $v(p)=\{q_3\}$ (see Figure~2). Let 
\[
	\psi:=\coop{\emptyset}\atlf p.
\]
We consider the embedded game related to $\psi$ with the time limit bounds $3$ and $4$. When $\tlimbound=3$, Eloise has the following winning time labels: the state $q_0$ has the label $\lose$, $q_1$ the label $2$, $q_2$ the label $1$, $q_3$ the label $0$ and all the other states have the label $\lose$. When $\tlimbound$ is increased to $4$, then the label of $q_0$ changes from $\lose$ to the value $3$, but the other labels remain the same. 

Let then
\[
	\varphi:=\coop{\emptyset}\atlf\psi \quad (=\coop{\emptyset}\atlf\coop{\emptyset}\atlf p).
\]
We then consider the embedded game related to $\varphi$ with the time limit bounds $3$ and $4$. When $\tlimbound=3$, Eloise's winning time labels are as follows: the state $q_0$ has the label $1$, the states $q_1$, $q_2$ and $q_3$ have the label $0$ and all the other states have the label $\lose$. When $\tlimbound$ is increased to $4$, then the label of $q_0$ is lowered from $1$ to $0$, but all the other labels remain the same.
\begin{figure}[h]
\begin{center}
\begin{tikzpicture}
	[place/.style={text width=0.5cm, align=flush center,
	circle,draw=black!100,fill=black!0,thick,inner sep=2pt,minimum size=4mm},scale=0.87]
	\node at (-0.5,0.5) (q0) {$q_0$};
	\node at (1.5,0.5) (q1) {$q_1$};
	\node at (3.5,0.5) (q2) {$q_2$};
	\node at (5.5,0.5) (q3) {$q_3$};
	\node at (7.5,0.5) (q4) {$q_4$};
	\node at (0,0) (1) [place] {$\neg p$};
	\node at (2,0) (2) [place] {$\neg p$};
	\node at (4,0) (3) [place] {$\neg p$};	
	\node at (6,0) (4) [place,fill=red!50] {$ p$};	
	\node at (8,0) (5) [place] {$\neg p$};	
	\node at (10,0) (6) {\dots};
	\draw [-latex,thick] (1) to node[below] {0} (2);	
	\draw [-latex,thick] (2) to node[below] {0} (3);
	\draw [-latex,thick] (3) to node[below] {0} (4);
	\draw [-latex,thick] (4) to node[below] {0} (5);
	\draw [-latex,thick] (5) to node[below] {0} (6);
\end{tikzpicture}
\label{fig 3}
\caption{An example of a game in which increasing the time limit bound lowers some winning time labels.} 
\end{center}
\vspace{-0,5cm}
\end{figure}
\end{example}

However, if all ordinal 
valued labels stay strictly below $\tlimbound$ in
all embedded games when going from $\Gamma$ to $\Gamma'$,
then each label in fact remains the same in the transition.
This is shown by the following proposition.

\begin{proposition}\label{the: stability of labels}
Let $\tlimbound,\tlimbound'>0$ be ordinals such that $\tlimbound<\tlimbound'$. Consider bounded evaluation games $\mathcal{G}_\tlimbound=\mathcal{G}(\mathcal{M},\qzero,\varphi,\tlimbound)$ and $\mathcal{G}_{\tlimbound'}=\mathcal{G}(\mathcal{M},\qzero,\varphi,\tlimbound')$. 
Suppose that all the ordinal valued winning time labels, for the embedded games of $\mathcal{G}_{\tlimbound'}$, are strictly smaller than $\tlimbound$. Then players have exactly the same winning time labels for the embedded games of $\mathcal{G}_{\tlimbound'}$ as as they have for the embedded games of $\mathcal{G}_{\tlimbound}$.
\end{proposition}

\begin{proof}
\textit{(Sketch)} 
Since winning time labels determine the time limits needed for winning, there is no need for the players to announce any higher ordinals than the winning time labels. Therefore, by the assumption that all winning time labels of the embedded games within $\mathcal{G}_{\tlimbound'}$ are below $\tlimbound$, we can show that $\tlimbound'$-bounded \GTS becomes equivalent with $\tlimbound$-bounded \GTS. From this it follows that, for any embedded game, the same winning time labels are constructed with respect to both $\tlimbound$ and $\tlimbound'$.
For more details, see the proof in the appendix.
\end{proof}

We say that an ordinal $\tlimbound$ is \defstyle{stable} for an
embedded game $\mathbf{G}$ if the winning time labels of $\mathbf{G}$
cannot be altered by increasing the time limit bound from $\Gamma$ to any higher ordinal.
We say that $\Gamma$ is \defstyle{globally stable} for a \CGM
$\mathcal{M}$ if $\Gamma$ is stable for all bounded embedded games within
all evaluation games $\evalgame(\mathcal{M},\qzero,\varphi,\tlimbound)$.
By Proposition~\ref{the: stability of labels}, it is equivalent to say that $\tlimbound$ is globally stable for $\mathcal{M}$ if, in the evaluation games for $\mathcal{M}$, we cannot create any labels with the ordinal value $\gamma\geq\tlimbound$ by increasing the value of the time limit bound $\tlimbound$ to any $\tlimbound'>\tlimbound$.

We will see later that 
there exists a globally stable time limit bound for every concurrent game model. 
When $\tlimbound$ is globally stable, its role is not so relevant any more, 
since players would not benefit from the ability to choose arbitrarily high time limits.
However, we always need some time limit bound in order to avoid strategies
becoming proper classes.

%%%%%%%%%%%%%%%%%%%%%
\subsection{Canonical strategies for embedded games}

Here we define so-called \emph{canonical strategies}. They are 
guaranteed to be winning strategies whenever a winning strategy exists.
In the following definition we first consider the case where the player $\pl$ is the controller $\ctr$.

\begin{definition}\label{def: Canonical strategy of the controller}
Let $\mathbf{G}=\embgame(\vrf,\ctr,A,\initst,\psi_{\ctr},\psi_{\octr})$
be an embedded game,
let $\pl\in\{{\bf A},{\bf E}\}$, and assume that $\pl = \ctr$.
We define a \defstyle{canonical strategy $\tau_{{\pl}}$} for ${\pl}$ in $\mathbf{G}$, for each $q\in\St$, as follows.

\begin{itemize}[leftmargin=*]
\item If $\mathcal{L}_{\pl}(q)=\tlim$, then $\tau_{{\pl}}(q)=\sigma_{\pl}(\tlim,q)$ for some strategy $\sigma_{\pl}$ for which there is a timer $t$ such that
$(\sigma_\pl,t)$ is a timed winning strategy in $\mathbf{G}[q,\tlim']$ for all $\tlim'$ such that $\tlim\leq\tlim'<\tlimbound$. 
(Note that such a strategy exists by Proposition~\ref{the: Winning time labels}).

\item If $\mathcal{L}_{\pl}(q)=\lose$, then we put $\tau_{{\pl}}(q) = \void$.
\end{itemize}

We define the \defstyle{canonical timer $t_{can}$} for ${\pl}$ in $\mathbf{G}$ for any pair $(\tlim,q)$ (where $\tlim<\tlimbound$ is a limit ordinal and $q\in\St$) as follows.
\begin{itemize}[leftmargin=*]
\item if $\mathcal{L}_{\pl}(q)\neq \lose$
and $\mathcal{L}_{\pl}(q)<\tlim$, then $t_{can}(\tlim,q)=\mathcal{L}_{\pl}(q)$.
\item otherwise $t_{can}(\tlim,q)=\void$.
\end{itemize}
We call the pair $(\tau_\pl,t_{can})$ a \defstyle{canonically timed strategy}
(for the controller).

\end{definition}

Note that $\tau_{{\pl}}$ is not necessarily unique
since we may have to choose one from several strategies.
However, these choices are all equally good for our purposes. 
Note also that the canonical strategy $\tau_\pl$ depends only on states and can thus be used in both unbounded and bounded embedded games. 
We will see that if $\ctr$ has a timed winning strategy in $\mathbf{G}[\initlim]$ for some $\initlim<\tlimbound$, then $\ctr$ wins $\mathbf{G}[\initlim]$ with $(\tau_\ctr,t_{can})$.
The canonical strategy of $\ctr$ can also be seen as optimal for winning the game as fast as possible.
This is because the canonical strategy always follows actions given by a strategy that is a winning strategy with the lowest possible time limit.

In the next definition we consider different types of canonical strategies
for the case where the player $\pl$ is not the controller $\ctr$.

\begin{definition}\label{def: Canonical strategy of the non-controller}
Let 
$\mathbf{G}=\embgame(\vrf,\ctr,A,\initst,\psi_{\ctr},\psi_{\octr})$
be an embedded game, let $\pl\in\{{\bf A},{\bf E}\}$ and 
assume that ${\pl}\neq\ctr$. 
We define a \defstyle{canonical strategy} $\tau_{{\pl}}$ for $\pl$ in $\mathbf{G}[\initlim]$ for all $\gamma_0 <\Gamma$, at every configuration $(\tlim,q)$ where $\tlim<\tlimbound$ and $q\in\St$, as follows.

\begin{itemize}[leftmargin=*]
\item If $\mathcal{L}_{\pl}(q)=\win$, then $\tau_{{\pl}}(\tlim,q)=\sigma_{\pl}(\tlim,q)$ for some $\sigma_{\pl}$ that is a winning strategy in $\mathbf{G}[q,\tlim]$ for every $\tlim<\tlimbound$ (such a strategy exists by Proposition \ref{the: Winning time labels}).

\item Else, if $\mathcal{L}_{\pl}(q)=\tlim'$ and $\tlim'>\tlim$, then $\tau_{{\pl}}(\tlim,q)=\sigma_{\pl}(\tlim,q)$ for some $\sigma_{\pl}$ that is a winning strategy in $\mathbf{G}[q,\tlim]$
(such a strategy exists by Proposition \ref{the: Winning time labels}). 

\item Otherwise we define $\tau_{\pl}(\tlim,q)=\void$.
\end{itemize}

We also define, for every $n<\omega$, the \defstyle{$n$-canonical strategy} $\tau_\pl^n$ for $\pl$ in $\mathbf{G}$, for each $q\in\St$, as follows.
\begin{itemize}[leftmargin=*]
\item If $\mathcal{L}_\pl(q)\geq\omega$ or $\mathcal{L}_\pl(q)=\win$, then $\tau_{\pl}^n(q)=\tau_{\pl}(n,q)$. 

\item Else, if $\mathcal{L}_{\pl}(q)=m>0$, then $\tau_{\pl}^n(q)=\sigma_{\pl}(m\!-\!1,q)$ for some $\sigma_\pl$ that is a winning strategy in $\mathbf{G}[q,m-1]$. (Such a strategy exists by Proposition \ref{the: Winning time labels}).

\item Otherwise $\tau_{\pl}^n(q)=\void$.
\end{itemize}

Finally, when $\tlimbound$ is a successor ordinal, we define the \defstyle{$\infty$-canonical strategy} $\tau_\pl^\infty$ for $\pl$ in~$\mathbf{G}$, for each $q\in\St$, as follows.
\begin{itemize}[leftmargin=*]
\item If $\mathcal{L}_{\pl}(q)=\win$, then $\tau_{\pl}^\infty(q)=\tau_{\pl}(\tlimbound\!-\!1,q)$. 

\item otherwise $\tau_{\pl}^\infty(q)=\void$.
\end{itemize}
%
%Note that to be able to define $\tau_\pl^\infty$, we have to
%assume that $\tlimbound$ is a successor ordinal.
\end{definition}

When $\pl\neq\ctr$, the canonical strategy $\tau_\pl$ depends on time limits, and thus it cannot be used in unbounded embedded games. However, both $n$-canonical and $\infty$-canonical strategies depend only on states.
We fix the notation such that hereafter $\tau_\pl$, $\tau_\pl^n$ and $\tau_\pl^\infty$ will always denote canonical strategies (of the respective type) for the player $\pl$.

As we will see, canonical strategies $\tau_{\octr}$ for $\octr$ are optimal in a sense that if $\octr$ has a winning strategy in $\mathbf{G}[\initlim]$ for some $\initlim<\tlimbound$, then $\octr$ wins $\mathbf{G}[\initlim]$ with $\tau_{\octr}$. We will see later that $n$-canonical strategies are important in the finitely bounded evaluation games and $\infty$-canonical strategies in evaluation games with sufficiently large time limit bounds $\tlimbound$.
Intuitively, the $\infty$-canonical strategy always assumes that the
highest possible time limit $\Gamma - 1$ has been set for every position. Thus, 
when $\octr$ uses an $\infty$-canonical strategy, (s)he plays as
carefully as possible, always assuming the `worst' possible time limit $\Gamma - 1$.
The $n$-canonical strategy behaves in a similar way, but under the assumption that the initial time limit $\tlim_0$ was set to at most $n$.

The following definition will be useful in our proofs later on.

\begin{definition}
Let $\mathbf{G}=\embgame(\vrf,\ctr,A,\initst,\psi_{\ctr},\psi_{\octr})$ be an
embedded game.
%
%in an evaluation game with time limit bound $\Gamma$.
%
Let $\sigma_\pl$ be a strategy in $\mathbf{G}[\initlim]$ $(\initlim<\tlimbound)$.
Consider a configuration $(\tlim,q)$ and suppose that $\sigma_{\pl}(\tlim,q)$ is either a tuple of actions for $A$ or some response function for $\overline{A}$.
We say that set $Q\subseteq\St$ is \defstyle{forced} by $\sigma_\pl(\tlim,q)$ if 
for each $q'\in\St$, it holds that 
$q'\in Q$ if and only if there is some play with $\sigma_\pl$ from $(\tlim,q)$
such that the next configuration is $(\tlim',q')$ for some $\tlim'$.
We use the same terminology for the set forced by $\sigma_\pl(q)$ when  $\sigma_\pl$ depends only on states.
\end{definition}

The following proposition shows that the canonical strategy (with the canonical timer) is
guaranteed to be a (timed) winning strategy always when such a strategy exists.
\begin{proposition}
\label{the: Canonical strategy}
Let  $\mathbf{G}=\embgame(\vrf,\ctr,A,\initst,\psi_{\ctr},\psi_{\octr})$
be an embedded game, $\pl\in\{{\bf A},{\bf E}\}$ and $\initlim<\tlimbound$.
\begin{enumerate}[leftmargin=*]
\item Suppose that $\pl=\ctr$. If ${\pl}$ has any timed winning strategy $(\sigma_\pl,t)$ in $\mathbf{G}[\initlim]$, then $(\tau_\pl,t_{can})$ is a timed winning strategy for $\pl$ in $\mathbf{G}[\initlim]$.
\item Suppose that $\pl\neq\ctr$. If ${\pl}$ has any winning strategy $\sigma_\pl$ in $\mathbf{G}[\initlim]$, then $\tau_{\pl}$ is a winning strategy for $\pl$ in $\mathbf{G}[\initlim]$.
\end{enumerate}
\end{proposition}

\begin{proof}
The proof follows quite routinely from Definitions \ref{def: Canonical strategy of the controller} and \ref{def: Canonical strategy of the non-controller}. 
See details in the Appendix. 
\end{proof}

Recall Remark~\ref{On genuine positionality} on ``genuinely positional'' strategies. Since canonically timed strategies and canonical strategies depend on neither the initial state $\initst$ nor the initial time limit $\initlim$, it follows that if a player has a winning strategy in an embedded game, then (s)he has a genuinely positional winning strategy. Also the converse clearly holds:  if (s)he has a genuinely positional winning strategy, she has a standard winning strategy.
By the first claim of the previous proposition, we see that it suffices to
consider those strategies of player $\ctr$ which are
independent of time limits in configurations. 
The following lemma
shows that the same holds for 
the player $\octr$ in bounded embedded games
with a \emph{finite} time limit.
The key here will be the use of $n$-canonical strategies which only depend on states.
%
%A proof is given in \!\cite{GTS-ATL-2015}.

%
\begin{lemma}\label{the: n-canonical strategy}
Let  $\mathbf{G}=\embgame(\vrf,\ctr,A,\initst,\psi_{\ctr},\psi_{\octr})$
be an embedded game, let $\pl\in\{{\bf A},{\bf E}\}$ and assume that $\pl\neq\ctr$.
For any $n<\omega$, if ${\pl}$ has a winning strategy $\sigma_\pl$ in $\mathbf{G}[m]$ for some $m\leq n$, then $\tau_\pl^n$ is a winning strategy in $\mathbf{G}[m]$.
\end{lemma}

\begin{proof}
We will prove by induction on $m\leq n$ that for any $q\in\St$, if $\pl$ has winning strategy in $\mathbf{G}[q,m]$, then $\tau_\pl^n$ is a winning strategy in $\mathbf{G}[q,m]$.
If $m=0$ and ${\pl}$ has a winning strategy $\sigma_\pl$ in $\mathbf{G}[q,0]$, then every strategy of $\pl$ will be a winning strategy in $\mathbf{G}[q,0]$. Hence, in particular, $\tau_\pl^n$ is a winning strategy in $\mathbf{G}[q,0]$.

Suppose then that the claim holds for $m\!-\!1$ and that $\pl$ has a winning strategy in $\mathbf{G}[q,m]$. Thus, by Proposition~\ref{the: Winning time labels}, we have $\mathcal{L}_\pl(q)>m$ or $\mathcal{L}_\pl(q)=\win$.

Suppose first that $\mathcal{L}_\pl(q)=m'<\omega$, and let $\sigma_\pl$ be a strategy for which $l(q,\sigma_\pl)=m'$
and $\tau_\pl^n(q)=\sigma_\pl(m'\!-\!1,q)$
(such a strategy $\sigma_{\pl}$ exists by the definition of $\tau_{\pl}^n$).
Let $Q\subseteq\St$ be the set of states forced by $\sigma_\pl(m'\!-\!1,q)$. Since $m'>m$, the strategy $\sigma_\pl$ must be a winning strategy in $\mathbf{G}[q,m]$, 
and thus it will also be a winning strategy in 
$\mathbf{G}[q',m\!-\!1]$ for every $q'\in Q$. 
Thus, by the inductive hypothesis, $\tau_\pl^n$ is a winning strategy in $\mathbf{G}[q',m\!-\!1]$ for every $q'\in Q$. Therefore we 
observe that $\tau_{\pl}^n$ will also be a winning strategy in $\mathbf{G}[q',m]$.
Suppose then that $\mathcal{L}_\pl(q)\geq\omega$ or $\mathcal{L}_\pl(q)=\win$, and let $\sigma_\pl$ be a strategy such that $l(q,\sigma_\pl)\in\{\win\}\cup\{\tlim<\tlimbound\mid \tlim>n\}$ and $\tau_\pl^n(q)=\tau_\pl(n,q)=\sigma_\pl(n,q)$.
(Recall Definition 
\ref{def: Canonical strategy of the non-controller}; the strategy $\sigma_{\pl}$ exists by the
definitions of $\tau_{\pl}^n$ and $\tau_\pl$.)
Let $Q\subseteq\St$ be the set of states that is forced by $\sigma_\pl(n,q)$. Since $m\leq n$, the strategy $\sigma_\pl$ is a winning strategy in $\mathbf{G}[q,m]$, and thus it is also a winning strategy in $\mathbf{G}[q',m\!-\!1]$ for every $q'\in Q$. Hence, by the inductive hypothesis, $\tau_\pl^n$ is a winning strategy in $\mathbf{G}[q',m\!-\!1]$ for every $q'\in Q$. Thus, we observe that $\tau_\pl^n$ is a winning strategy in $\mathbf{G}[q,m]$.
\end{proof}
\begin{example}
If $\mathcal{L}_{\octr}(q)=\omega$, the player $\octr$ can win the game with any time limit $n<\omega$,
but there is no single strategy that would win for every $n$. But if $\octr$ knows that the initial time limit is (at most) $m$, then (s)he knows that the $m$-canonical strategy will a winning strategy for her/him. Therefore, intuitively, $\octr$ only needs to know the time limit when selecting the strategy, but not when using it (since $n$-canonical strategies depend on states only).
\end{example}
%

%%%%%%%%%%%%%%%%%%%%%

\subsection{Determinacy of bounded embedded games}

The correspondence in the following proposition between the winning time labels of $\ctr$ and $\octr$
will be the key for proving determinacy of bounded embedded games.
The main idea for proving the proposition is similar to the 
one in the standard proof of the Gale-Stewart Theorem.

\begin{proposition}\label{the: Correspondence of the winning time labels}
Let
$\mathbf{G}=\embgame(\vrf,\ctr,A,\initst,\psi_{\ctr},\psi_{\octr})$ be an embedded game.
Then for each state $q\in\St$ and each ordinal $\tlim<\tlimbound$, it holds that 
$
	\mathcal{L}_{\ctr}(q)\!=\!\tlim \text{ iff } 
	\mathcal{L}_{\octr}(q)\!=\!\tlim
$
\end{proposition}

\begin{proof}
\textit{(Sketch)} 
We can prove the claim by transfinite induction on $\tlim$. The case $\tlim=0$ is clear, since if $\ctr$ cannot win the game with time limit $0$, then $\octr$ will win it automatically. We then suppose that the claim holds for every $\tlim'<\tlim$ and prove the equivalence for $\tlim$. If $\mathcal{L}_\ctr(q)=\gamma$, then $\ctr$ has a winning strategy in $\mathbf{G}[q,\tlim]$ and thus $\octr$ cannot have a winning strategy in that game. Hence by Proposition~\ref{the: Winning time labels} we have $\mathcal{L}_{\octr}\leq\tlim$.
By the inductive hypothesis, $\mathcal{L}_{\octr}\neq\tlim'$ for every $\tlim'<\tlim$ and thus $\mathcal{L}_{\octr}(q)=\gamma$.
Suppose then that $\mathcal{L}_{\octr}(q)=\gamma$. If there existed some $\sigma_{\octr}$, $\tlim'<\tlimbound$ and $Q\subseteq\St$ forced by $\sigma_{\octr}(\tlim',q)$ such that $\mathcal{L}_{\octr}(q')\geq\tlim$ for every $q'\in Q$, then we could have used $\sigma_{\octr}$ to construct a winning strategy for $\octr$ in $\mathbf{G}[q,\tlim]$. This is not possible since we have $\mathcal{L}_{\octr}(q)=\gamma$. 
Thus it can be shown that $\ctr$ can play so that for all possible successor states $q'$, we have $\mathcal{L}_{\octr}(q')<\gamma$, whence by the inductive hypothesis $\mathcal{L}_{\ctr}(q')<\gamma$.
Hence we can construct a timed winning strategy for $\ctr$ in $\mathbf{G}[q,\tlim]$,  and thus infer that $\mathcal{L}_\ctr(q)=\gamma$.
For a detailed proof, see the appendix. 
\end{proof}

 Recall that apart from ordinal values that are less than the bound $\tlimbound$, the only possible winning time label for $\ctr$ is the label $\lose$ (see Definition~\ref{def: Winning time labels}).
Similarly for $\octr$, the only non-ordinal value for the labels is $\win$. 
Hence by the previous proposition, we also have
\[
	\mathcal{L}_{\ctr}(q)=\lose \,\text{ iff }\,\mathcal{L}_{\octr}(q)=\win.
\]
It is now easy to show that all bounded embedded games are (positionally) determined.

\begin{corollary}
\label{the: Determinacy of the bounded embedded game}
The controller $\ctr$ has a timed winning strategy in a bounded embedded game 
$\embgame(\vrf,\ctr,A,\initst,\psi_{\ctr},\psi_{\octr})[\initlim]$
iff $\octr$ does not have a winning strategy in that game.
\end{corollary}

\begin{proof}
If $\mathcal{L}_\ctr(\initst)=\lose$, then $\mathcal{L}_{\octr}(\initst)=\win$, whence by Proposition~\ref{the: Winning time labels}, the player $\octr$ has a winning
strategy and $\ctr$ does not have a timed winning strategy. Else $\mathcal{L}_\ctr(\initst)=\tlim$ for some $\tlim<\tlimbound$. Now, by Proposition~\ref{the: Correspondence of the winning time labels}, also $\mathcal{L}_{\octr}(\initst)=\tlim$. If $\tlim\leq\initlim$, then by Proposition~\ref{the: Winning time labels} the player $\ctr$ has a timed winning strategy, while $\octr$ does not have a
winning strategy. Analogously, if $\tlim>\initlim$, then $\octr$ has a winning strategy, while $\ctr$ does not have a timed winning strategy.
\end{proof}

Due to the positional determinacy of bounded embedded games, we can prove that also bounded \emph{evaluation} games are positionally determined.

\begin{proposition}\label{the: Determinacy of the evaluation games}
Let $\mathcal{M}$ be a \CGM, $q\in\St$, $\varphi$ an \ATL-formula, $\tlimbound>0$ an ordinal and $\pl\in\{{\bf A},{\bf E}\}$. Then 
$\pl$ has a winning strategy in $\mathcal{G}(\mathcal{M},q,\varphi,\tlimbound)$ if and only if $\opl$ does not have a winning strategy in $(\mathcal{G},q,\varphi,\tlimbound)$.
\end{proposition}

\begin{proof}
(\textit{Sketch})
Since embedded games are positionally determined, we can prove the positional determinacy of the evaluation games by a simple induction on $\varphi$. 
%
%It is important to note here that both bounded and unbounded evaluation games are games of perfect information. 
%Hence the proof can be done in a similar way as the Gale-Stewart Theorem.
%
For a detailed proof, see the appendix. 
\end{proof}

Consequently the following equivalence holds: \quad
$\mathcal{M},q\models_\tlimbound^g\neg\varphi \,\text{ iff }\, \mathcal{M},q\not\models_\tlimbound^g\varphi$.

%%%%%%%%%%%%%%%%%%%%%
\subsection{Finding stable time limit bounds}\label{stablelimitbounds}

\begin{definition}\label{def: Image degree}
Let $\mathcal{M}$ be a \CGM and let $q\in\St$. We define the \defstyle{branching degree of $q$}, $\bd(q)$, to be the cardinality of the set of states accessible from $q$ with a single transition:
$$\bd(q):=\card(\{o(q,\vec\alpha)\mid \vec\alpha\in\avact(\Agt,q)\}).$$
%
%
%
%We define the \defstyle{infinite branching bound of $\mathcal{M}$}, $\ibb(\mathcal{M})$, as the smallest infinite cardinal $\kappa$ such that $\kappa > \bd(q)$ for every $q\in\St$.
%
%
%
We define the \defstyle{regular branching bound of $\mathcal{M}$} to be the \emph{smallest} cardinal $\rbb(\mathcal{M})$ for which 
the following conditions hold:
\begin{enumerate}[leftmargin=*]
\item $\rbb(\mathcal{M}) > \bd(q)$ for every $q\in\St$.
\item $\rbb(\mathcal{M})$ is infinite.
\item $\rbb(\mathcal{M})$ is a \emph{regular} cardinal.
\end{enumerate}
\end{definition}
We will see that the value of $\rbb(\mathcal{M})$ is always a globally stable time limit bound for $\mathcal{M}$. We will also see that the regular branching bound is, in a sense, the smallest possible time limit bound that is guaranteed to be stable (see the claim of Proposition~\ref{prop:exII}).
\begin{remark}\label{rbb}
Let $\kappa$ be the smallest infinite cardinal for which $\bd(q)<\kappa$ for every $q\in\St$. Since successor cardinals are always regular, we
must have $\rbb(\mathcal{M})\leq\kappa^+$ ($\kappa^+$ here is the successor
\emph{cardinal} of $\kappa$). However, in a typical case with infinite models, we have $\rbb(\mathcal{M})=\kappa$. In particular
\[
	\rbb(\mathcal{M})=\omega 
	\;\; \text{ iff  } \;\; \mathcal{M} \text{ is \defstyle{image-finite} (that is, $\bd(q)<\omega$ for every $q\in\St$)}.
\]
Also, if $\card(\Act)\leq\kappa$ or $\card(\mathcal{\St})\leq\kappa$ for an infinite $\kappa$, then necessarily $\rbb(\mathcal{M})\leq\kappa^+$.
\end{remark}
The following lemma shows an important correspondence between the canonical strategies and the winning time labels of the controller.
\begin{lemma}
\label{the: Winning time labels 2}
Let  
$\mathbf{G}=\embgame(\vrf,\ctr,A,\initst,\psi_{\ctr},\psi_{\octr})$
be an embedded game, $\pl\in\{{\bf A},{\bf E}\}$ and ${\pl}=\ctr$. 
The following holds for every $q\in\St$:
if $\mathcal{L}_{\pl}(q)=\tlim>0$ and $Q\subseteq\St$ is forced by $\tau_{\pl}(q)$, 
then 
$\mathcal{L}_{\pl}(q')<\tlim$ for every $q'\in Q$, and
\begin{itemize}[leftmargin=*]
\item $\max\{\mathcal{L}_{\pl}(q')\mid q'\in Q\}=\tlim-1$, if $\tlim$ is a successor ordinal,
\item $\sup\{\mathcal{L}_{\pl}(q')\mid q'\in Q\}=\tlim$, if $\tlim$ is a limit ordinal.
\end{itemize}
\end{lemma}

\begin{proof} 
(\textit {Sketch})
When $\mathcal{L}_{\pl}(q)=\tlim>0$, by Proposition~\ref{the: Canonical strategy}, $(\tau_\pl,t_{can})$ is a timed winning strategy in $\mathbf{G}[q,\tlim]$. 
Therefore every winning time label in the set $Q$ forced by $\tau_\pl(q)$ must be an
ordinal less than $\tlim$. If $\tlim$ is a successor ordinal, then there must be some state with label $\tlim\!-\!1$
in $Q$, and if $\tlim$ is a limit ordinal, then $\tlim$ must be the supremum of the labels in $Q$ (else there would be a winning strategy
for $\ctr$ in $\mathbf{G}[q,\tlim']$ for some $\tlim'<\tlim$).
For more details, see the proof in the appendix. 
\end{proof}

The following lemma shows that if a certain ordinal-valued winning time label exist for an embedded game, then all the smaller winning time labels must exist for that game, as well.
\begin{lemma}\label{the: Downwards existence}
Let $\mathbf{G}=\embgame(\vrf,\ctr,A,\initst,\psi_{\ctr},\psi_{\octr})$ be an embedded game and $\tlim<\tlimbound$ an ordinal. Assume that $\mathcal{L}_\pl(q)=\tlim$ for some $q\in\St$ and $\pl\in\{{\bf A},{\bf E}\}$. Then for every $\delta\leq\tlim$ there is a state $q_\delta$ for which $\mathcal{L}_\ctr(q_\delta)=\delta$.
\end{lemma}

\begin{proof}
We prove the claim by transfinite induction on $\tlim<\tlimbound$ for every $q\in\St$. Let the inductive hypothesis be that the claim holds for every $\tlim'<\tlim$ and suppose that either of the players has winning time label $\tlim$ at some state $q$. By Proposition~\ref{the: Correspondence of the winning time labels}, we have $\mathcal{L}_\ctr(q)=\tlim$.

Assume that $\delta<\tlim$. If $\tlim$ is a successor ordinal, then by Lemma~\ref{the: Winning time labels 2} there is a state $q'\in\St$ such that  $\mathcal{L}_\ctr(q')=\tlim\!-\!1$. Since $\delta\leq\tlim\!-\!1$, by the inductive hypothesis there is a state $q_\delta$ for which $\mathcal{L}_\ctr(q_\delta)=\delta$.
Suppose then that $\tlim$ is a limit ordinal. By Lemma~\ref{the: Winning time labels 2}\,, there must be a state $q'\in\St$ such that $\mathcal{L}_\ctr(q')=\tlim'$ for some ordinal $\tlim'$ such that $\delta<\tlim'<\tlim$. Hence by the inductive hypothesis there is a state $q_\delta$ for which $\mathcal{L}_\ctr(q_\delta)=\delta$.
\end{proof}
In the following proposition, we show that in finite models, all winning time labels are strictly smaller than the cardinality of the model.

\begin{proposition}
\label{prop: Stable bound for finite models}
Let $\mathcal{M}$ be a \CGM. If $\card(\mathcal{M})=n<\omega$, then $n$ is a globally stable time limit bound for $\mathcal{M}$. 
\end{proposition}

\begin{proof}
(Recall that, by Proposition~\ref{the: stability of labels}, if all possible winning time labels are below an ordinal $\tlimbound$, for an arbitrary time limit bound $\tlimbound'$, then $\tlimbound$ is globally stable for $\mathcal{M}$.)
If there was some state with a winning time label $\tlim\geq n$, then by Lemma~\ref{the: Downwards existence}, there would be a state $q\in\St$ for which $\mathcal{L}_\ctr(q)=n$. Further, by Lemma~\ref{the: Downwards existence}, we would now find states with winning time labels $n\!-\!1,n\!-\!2,\dots,0$. But since each state may only have a single label, this would mean that $\card(\mathcal{M})\geq n+1 > n$, a contradiction.
(This claim is also quite obvious by the observation that the controller can only win the embedded game by reaching a state in the truth set of the formula $\psi_\ctr$. Hence it would not be beneficial for the controller to go in cycles.)
\end{proof}

The following result shows how to find an upper bound for the values of possible winning time labels by just looking at the regular branching bound of a model. Recall that a cardinal $\kappa$ is \defstyle{regular} if $\kappa$ is equal to its own \emph{cofinality}, i.e., there is no set of less than $\kappa$ many ordinals, each less than $\kappa$, with supremum $\kappa$.
\begin{proposition}\label{the: Stable time limit bound for a model}
Let $\mathcal{M}$ be a \CGM.
Then $\rbb(\mathcal{M})$ is globally stable for $\mathcal{M}$.
\end{proposition}

\begin{proof}
For the sake of contradiction, suppose that there is $\tlimbound'>\rbb(\mathcal{M})$ and an embedded game $\mathbf{G}$ within a bounded evaluation game $\evalgame(\mathcal{M},\qzero,\varphi,\tlimbound')$ such that in $\mathbf{G}$ at least one of the players has winning time labels that are greater or equal to $\rbb(\mathcal{M})$.
By Lemma~\ref{the: Downwards existence}, there is $q\in\St$ for which $\mathcal{L}_\ctr(q)=\rbb(\mathcal{M})$. Let $Q\subseteq\St$ be the set of states that is forced by $\tau_\ctr(q)$.
Now $\card(Q)\leq\bd(q)<\rbb(\mathcal{M})$. 
By Lemma~\ref{the: Winning time labels 2}, $\mathcal{L}_\ctr(q')<\rbb(\mathcal{M})$ for every state $q'\in Q$, and furthermore, since  
$\rbb(\mathcal{M})$ is a limit ordinal,
$\rbb(\mathcal{M})$ must be the supremum of the winning time labels of the states in $Q$,
i.e. $\sup\{\mathcal{L}_\ctr(q')\mid q'\in Q\}=\rbb(\mathcal{M})$.

Now every winning time label in $Q$ is smaller than $\rbb(\mathcal{M})$ and the cardinality of $Q$ is less than $\rbb(\mathcal{M})$. Because $\rbb(\mathcal{M})$ is a regular cardinal, it is equal to its own cofinality. Hence we must have $\sup\{\mathcal{L}_\ctr(q')\mid q'\in Q\}<\rbb(\mathcal{M})$. This is a contradiction and thus $\rbb(\mathcal{M})$ must be globally stable for $\mathcal{M}$  (recall Proposition~\ref{the: stability of labels} on global stability).
\end{proof}

From Proposition~\ref{the: Stable time limit bound for a model} we obtain the following corollary on the stability of the time limit bounds used with the bounded and finitely bounded \GTS.
\begin{corollary}\label{cor: Stable time limit bound for a model}
Let $\mathcal{M}$ be any \CGM. 

\begin{enumerate}[leftmargin=*]
\item If $\card(\mathcal{M})=\kappa$, then $2^\kappa$ is a globally stable time limit bound for $\mathcal{M}$. 

\item  If $\mathcal{M}$ is image-finite, then $\omega$ is a globally stable time limit bound for $\mathcal{M}$.
\end{enumerate}
\end{corollary}
\begin{proof}
(1) Let  $\card(\mathcal{M})\!=\!\kappa$.
If $\kappa<\omega$, then by Proposition~\ref{prop: Stable bound for finite models}, $\kappa$ is globally stable for $\mathcal{M}$, whence also $2^\kappa$ is globally stable for $\mathcal{M}$. 
And if $\kappa\geq\omega$, then $\rbb(\mathcal{M})\leq\kappa^+\leq 2^\kappa$ (see Remark~\ref{rbb}) and thus, by Proposition~\ref{the: Stable time limit bound for a model}, $2^\kappa$ is a globally stable time limit bound for $\mathcal{M}$.

(2) If $\mathcal{M}$ is image-finite, $\rbb(\mathcal{M})=\omega$. By Proposition~\ref{the: Stable time limit bound for a model}, $\omega$ is globally stable for $\mathcal{M}$.
\end{proof}

Note that we could replace $2^\kappa$ with $\kappa^+$ in the case (1) above.
This gives us a stronger result in the models of set theory that do not satisfy the 
generalized continuum hypothesis. Also note that since $\rbb(\mathcal{M})$ is determined only by the size of the branchings in the model, the cardinality $\kappa$ of $\mathcal{M}$ might be much greater than $\rbb(\mathcal{M})$ (which is globally stable for $\mathcal{M}$).

As seen by Proposition~\ref{the: Winning time labels 2}, only sufficiently large branchings can generate winning time labels of higher cardinality. Thus, when estimating the value of a stable time limit bound, it suffices check the sizes of the large branchings in the model. We know by Proposition~\ref{the: Stable time limit bound for a model}, that $\rbb(\mathcal{M})$ is guaranteed to be stable for a \CGM $\mathcal{M}$, but can we give any better estimate of a stable bound by just looking at the branchings in $\mathcal{M}$? The answer is negative, in the sense of the following Proposition~\ref{prop:exII}. Indeed, it implies that only knowing the least infinite strict upper bound of the branchings $\bd(q)$ of all states $q$ of a \CGM $\mathcal{M}$, no lower time limit bound than $\rbb(\mathcal{M})$ is guaranteed to be globally stable for $\mathcal{M}$.
Before the proposition, we give some auxiliary definitions.

Let $\mathcal{M}$ be a $\CGM$ and let $\kappa$ be an infinite cardinal such that the branching degrees in $\mathcal{M}$ are strictly bounded by $\kappa$, that is
$
	\bd(q) < \kappa \; \text{ for every } q\in\St.
$
Then we say that $\mathcal{M}$ is \defstyle{less than $\kappa$-branching}.
For an ordinal $\gamma$, we say that a \CGM $\mathcal{N}$
\defstyle{realizes} $\gamma$ if there is some state of $\mathcal{N}$
that realizes the winning time label $\gamma$ in some
embedded game.

%\begin{customexample}{II}
\begin{proposition}
\label{prop:exII}
Let $\kappa$ be an infinite cardinal.
Then there exists a less than $\kappa$-branching model $\mathcal{M}$
that realizes every $\gamma<\rbb(\mathcal{M})$ and where
all branchings strictly less than $\kappa$ occur.
\end{proposition}
\begin{proof}
We begin the proof by constructing by transfinite induction a model $\mathcal{M}_{\gamma}$
for each $\gamma$ such that $\mathcal{M}_{\gamma}$ 
realizes exactly all ordinals less or equal to $\gamma$, and
furthermore, the following conditions hold.
\begin{enumerate}[leftmargin=*]
\item
If $\gamma$ is a finite ordinal, then the 
branching degree $\bd(q)$ of each state $q$ in $\mathcal{M}_{\gamma}$ is $1$.
\item
If $\gamma$ is an infinite ordinal, then $\card(\gamma)\leq\gamma$ is an upper bound for
the branchings $\bd(q)$ of the states $q$ in $\mathcal{M}_{\gamma}$.  Furthermore,
the following conditions hold.
\begin{enumerate}
\item
If $\card(\gamma)$ is a regular cardinal, then $\mathcal{M}_{\gamma}$ 
has a node with branching $\card(\gamma)$.
\item
If $\card(\gamma)$ is a singular (i.e., non-regular) cardinal,
then all states in $\mathcal{M}_{\gamma}$ 
have branching strictly less than $\card(\gamma)$.
\end{enumerate}
\end{enumerate}
All such models $\mathcal{M}_{\gamma}$
will be of the form $\mathcal{M}_\tlim=(\Agt, \St, \Prop, \Act, d, o, v)$, where $\Agt=\{a\}$, $q_0\in\St$, $\Prop=\{p\}$ and $\Act=\{\delta\mid\delta<\gamma\}$. We will always
establish that the winning time label for Eloise at the state $q_0$ in the embedded game $\mathbf{G}=({\bf E},{\bf E},\emptyset,q_0,p,\top)$ (which arises when verifying the formula $\varphi=\coop{\emptyset}\F p$ at $q_0$) is equal to $\gamma$.

If $\gamma=0$, we define $\mathcal{M}_0=(\Agt, \St, \Prop, \Act, d, o, v)$
where $\St=\{q_0\}$, $d(q_0,a)=\{0\}$, $o(q_0,0)=q_0$
and $v(p)=\{q_0\}$.
Now $\bd(q_0)=1$ and $\mathcal{L}_{\bf E}(q_0) = 0$.
%
%the branchings of $\mathcal{M}_0$ are
%bounded by $\kappa$. Also, there are no ordinals less than $0$ so
%trivially all ordinals less than $\gamma$ are realized in $\mathcal{M}_{\gamma}$.
%Also, clearly $\mathcal{L}_{\bf E}(q_0)=0$.

Suppose then that $\tlim$ is a
successor ordinal.
%By the inductive hypothesis, there is a model
Consider the model $\mathcal{M}_{\tlim\!-\!1}$ with $\mathcal{L}_{\bf E}(q_0)=\tlim\!-\!1$ obtained
by the induction hypothesis. Let $\mathcal{M}_{\tlim\!-\!1}'=(\Agt, \St', \Prop, \Act, d', o', v')$ be an isomorphic copy of $\mathcal{M}_{\tlim\!-\!1}$ in
which the state $q_0$ is replaced by a new
state $q'$. Let $\mathcal{M}_{\tlim}=(\Agt, \St, \Prop, \Act, d, o, v)$, where we define $\St:=\St'\cup\{q_0\}$,  $d:=d'\cup\{((q_0,a),\{0\})\}$, $o:=o'\cup\{((q_0,0),q')\}$ and $v:=v'$. Since $\bd(q_0)=1$,  the branchings of $\mathcal{M}_\tlim$
are bounded in the desired way by the induction hypothesis.
It is easy to see that $\mathcal{L}_{\bf E}(q_0)=\tlim$ , 
and all ordinals less than $\gamma$ are realized in $\mathcal{M}_{\gamma}$ by
the induction hypothesis.
Suppose then that $\tlim$ is a limit ordinal. We next construct, for 
later use, a set of ordinals
$\Psi\subseteq\{\delta\mid\delta<\tlim\}$ such that $\sup(\Psi)=\tlim$.
If $\card(\gamma)$ is a regular cardinal,
then we set $\Psi = \{\delta\, |\, \delta<\gamma\, \}$.
Otherwise, when $\card(\gamma)$ is a singular cardinal, we do
the following. Let $\mu$ be the cofinality of $\gamma$, whence there exists some set
$S$ of ordinals strictly less than $\tlim$ such that $\sup(S)=\tlim$
and $\card(S) = \mu$. Let $\Psi := S$.
Note that the following conditions hold.
\begin{enumerate}[leftmargin=*]
\item
If $\card(\gamma)$ is regular, then $\card(\Psi) = \card(\gamma)$. 
\item
If $\card(\gamma)$ is singular, then $\card(\Psi) < \card(\gamma)$. 
(This holds for the following reason.
The cofinality operator $\mathit{cf}$ satisfies $\mathit{cf}(\mathit{cf}(\alpha))
= \mathit{cf}(\alpha)$ for
each ordinal $\alpha$. Therefore, if $\mathit{cf}(\gamma) = \card(\gamma)$, we
obtain a contradiction  as follows: $\card(\gamma) = \mathit{cf}(\gamma)
= \mathit{cf}(\mathit{cf}(\gamma))
= \mathit{cf}(\card(\gamma))<\card(\gamma)$,
where the last inequality is due to $\card(\gamma)$ being singular.
Thus $\mathit{cf}(\gamma)<\card(\gamma)$,
whence $\card(\Psi) = \mathit{cf}(\gamma)<\card(\gamma)$, as required.)
\end{enumerate}
%

%
%We have now constructed the set $\Psi$ for both cases,
%the case where $\gamma$ is a singular cardinal and the case where it is not.
%We continue discussing these two cases in parallel.
%

%
By the inductive hypothesis, for every ordinal $\delta\in\Psi$, there is a
suitable model $\mathcal{M}_{\delta}$
%which is less than $\kappa$-branching
where $\mathcal{L}_{\bf E}(q_0)=\delta$.
We build an isomorphic copy $\mathcal{M}_{\delta}'=(\Agt, \St_\delta, \Prop, \Act_{\delta}, d_\delta, o_\delta, v_\delta)$ of every model $\mathcal{M}_\delta$ so that
the sets $\St_\delta$  are disjoint and such that the state $q_0$ is replaced with a new state $q_\delta$ in every model $\mathcal{M}_{\delta}'$.
Now let $\mathcal{M}_{\tlim}=(\Agt, \St, \Prop, \Act, d, o, v)$, where: 
\begin{itemize}
\item $\St=\bigcup_{\delta\in\Psi}\St_\delta\cup\{q_0\}$,  

\item $d=\bigcup_{\delta\in\Psi}d_\delta\cup\{((q_0,a),\{\delta\mid \delta\in\Psi\})\}$,

\item $o=\bigcup_{\delta\in\Psi}o_\delta\cup\bigcup_{\delta\in\Psi}\{((q_0,\delta),q_\delta)\}$, 

\item $v := \{(p,\ \bigcup_{\delta\in\Psi}v_\delta(p))\}$, i.e., $v$ is the function 
that maps $p$ to the union of the sets $v_{\delta}(p)$.
\end{itemize}

\noindent
%Because $\bd(q_0)=\card(\Psi)<\kappa$ and the branchings in $\mathcal{M}_\delta$ are less %than by $\kappa$ for all $\delta\in\Psi$, the branchings in $\mathcal{M}_\tlim$ are also less than %$\kappa$.
Since $\sup(\Psi)=\tlim$, it is easy to see that $\mathcal{L}_{\bf E}(q_0)=\tlim$.
Clearly $\mathcal{M}_{\gamma}$ has the required properties,
whence the transfinite induction for constructing the models $\mathcal{M}_{\gamma}$ has now
been completed. 
%
%\end{customexample}
%
%
%

%
Before completing our proof, we define a further 
class of models $\mathcal{N}_{\mu}$, for each infinite cardinal $\mu$, 
such that $\mathcal{N}_{\mu}$ has a node $q'$ with $\bd(q') = \mu$.
These models $\mathcal{N}_{\mu}$ are defined simply for the sake of simplifying
our subsequent arguments. Intuitively, the models $\mathcal{N}_{\mu}$ will be
used in order to directly force the existence of sufficiently large branchings.
Formally, we define $\mathcal{N}_{\mu}
=(\Agt, \St_\mu, \Prop, \Act_{\mu}, d_\mu, o_\mu, v_\mu)$ as follows.
We let $\Agt = \{a\}$
and $\St_\mu := \{q'\}\cup\{q_{\delta}\, |\, \delta<\mu,\text{ $\delta$ is an ordinal}\}$.
We also define $\Pi = \{p\}$ and let $\Act_{\mu}$ be the set of 
ordinals smaller than $\mu$. We define the set of actions available to
agent $a$ at $q'$ to be the set $\Act_{\mu}$, while the set of
actions available elsewhere is defined arbitrarily. The outcome of $a$
choosing the action $\delta$ at $q'$ will lead to $q_{\delta}$.
At other states, the outcomes of actions are defined arbitrarily.
Finally, we define $v_{\mu}(p) = \emptyset$.

To complete our proof, we now argue as follows.
Let $\kappa$ be an infinite cardinal.
Assume first that $\kappa$ is regular. Let $\mathcal{M}$ be the 
disjoint union of the 
models $\mathcal{M}_{\gamma}$ and $\mathcal{N}_{\gamma}$ for all $\gamma<\kappa$
constructed above.
Now, $\mathcal{M}$ is less than $\kappa$-branching,
and since $\kappa$ is regular, we have $\rbb({\mathcal{M}}) = \kappa$.
Clearly $\mathcal{M}$ realizes exactly all
ordinals $\gamma<\rbb(\mathcal{M})=\kappa$, as required.
(Recall that $\rbb(\mathcal{M})$ is a globally stable limit bound and 
thus no labels greater or equal to $\rbb(\mathcal{M})$ can be
realized in $\mathcal{M}$.)
Assume then that $\kappa$ is singular. Now let $\mathcal{M}$ be the 
disjoint union of the models $\mathcal{M}_{\gamma}$ for all $\gamma<\kappa^+$ and the
models $\mathcal{N}_{\gamma}$ for all $\gamma<\kappa$.
Again $\mathcal{M}$ is less than $\kappa$-branching.
Since $\kappa$ is not regular, we have $\rbb({\mathcal{M}}) = \kappa^+$,
and thus $\mathcal{M}$ realizes exactly all
ordinals $\gamma<\rbb(\mathcal{M})$, as required.
\end{proof}

\begin{example}
By the proof of Proposition~\ref{prop:exII}, we can construct a model $\mathcal{M}$ whose  branchings are strictly below the \emph{singular} cardinal $\aleph_\omega$, but nevertheless, all the winning time labels $\tlim<\aleph_{\omega}^+=\rbb(\mathcal{M})$ are realized in $\mathcal{M}$.
Thus, if we know only that the branchings of an arbitrary model $\mathcal{N}$ are 
bounded by the strict least upper bound $\aleph_{\omega}$, 
but know nothing else about the model,
we cannot give any better estimate than $\rbb(\mathcal{N})=\aleph_\omega^+$ for a stable time limit bound for $\mathcal{N}$.
\end{example}

%%%%%%%%%%%%%%%%%%%%%

\subsection{Unbounded vs bounded embedded games}\label{ssec: Unbounded vs. bounded}

As mentioned earlier, $\infty$-canonical strategies $\tau_\pl^\infty$ (recall Definition~\ref{def: Canonical strategy of the non-controller}) are important in evaluation games with sufficiently large time limit bounds $\tlimbound$. We will see that if the time limit bound $\tlimbound$ is stable and there exists a winning strategy for each $\tlim<\tlimbound$, then $\tau_\pl^\infty$ will be a winning strategy for \emph{each} $\tlim<\tlimbound$.
The intuition behind $\infty$-canonical strategies is that the player does not know the current time limit $\tlim$ (since the strategy depends on states only), but (s)he is playing  ``defensively'' by always assuming $\tlim$ to have the highest possible value (namely $\tlimbound-1$).

The following lemma shows that when $\tlimbound$ satisfies certain conditions, then, if $\pl$ uses an $\infty$-canonical strategy $\tau_\pl^\infty$ and begins from a state with the label $\win$, $\pl$ will always stay in states with the label $\win$. 

\begin{lemma}\label{the: Preserving the label win}
Let 
$\mathbf{G}=\embgame(\vrf,\ctr,A,\initst,\psi_{\ctr},\psi_{\octr})$
be an embedded game, $\pl\in\{{\bf A},{\bf E}\}$ and $\pl\neq\ctr$. Assume that the time limit bound $\tlimbound$ is a successor ordinal and $\tlimbound\!-\!1$ is stable for $\mathbf{G}$. Now, for every  $q\in\St$, if  $\mathcal{L}_\pl(q)=\win$ and $Q\subseteq\St$ is forced by $\tau_\pl^\infty$, then $\mathcal{L}_\pl(q')=\win$ for every $q'\in Q$.
\end{lemma}

\begin{proof}
Suppose that $\mathcal{L}_\pl(q)=\win$ and $Q\subseteq\St$ is forced by $\tau_\pl^\infty(q)$. Let $\sigma_\pl$ be a strategy for which $l(q,\sigma_\pl)=\win$ and $\tau_\pl^\infty(q)=\sigma_\pl(\tlimbound\!-\!1,q)$ (such a strategy exists by Definition~\ref{def: Canonical strategy of the non-controller}). For the sake of contradiction, suppose that there is some $q'\in Q$ for which $\mathcal{L}_\pl(q')\neq\win$ and thus $\mathcal{L}_\pl(q')=\tlim$ for some $\tlim<\tlimbound$. Since $\tlimbound\!-\!1$ is stable for $\mathbf{G}$, we must also have $\tlim<\tlimbound\!-\!1$. Now there is a play of $\mathbf{G}[q,\tlimbound\!-\!1]$ in which $\pl$ is using $\sigma_\pl$ and the configuration $(\tlim,q')$ follows $(\tlimbound\!-\!1,q)$. But since $\mathcal{L}_\pl(q')=\tlim$, the strategy $\sigma_\pl$ cannot be a winning strategy in $\mathbf{G}[q',\tlim]$. Hence $\sigma_\pl$ is not a winning strategy in $\mathbf{G}[q,\Gamma - 1]$, which is a contradiction since $l(q,\sigma_\pl)=\win$.
\end{proof}

The following proposition shows that when the time limit bound $\tlimbound$ is stable, then bounded embedded games become essentially equivalent to unbounded embedded games.
For the proof, we use concepts that we have defined before. The key here is to use a canonically timed strategy when $\pl=\ctr$ and an $\infty$-canonical strategy when $\pl\neq\ctr$.
\begin{proposition}\label{the: Bounded vs. unbounded} 
Let $\mathbf{G}=\embgame(\vrf,\ctr,A,\initst,\psi_{\ctr},\psi_{\octr},)$ be an embedded game and let $\pl\in\{{\bf A},{\bf E}\}$. Suppose $\tlimbound$ is stable for $\mathbf{G}$. Then the following equivalences hold---under the exceptional assumption that the winning condition of the exit positions of the unbounded embedded game $\mathbf{G}$  (recall Definition~\ref{def: Winning strategies in embedded game}) is defined by using the semantics $\models_\tlimbound^g$ instead of $\models_u^g$ \footnote{However, we will see (by Theorem~\ref{the: Bounded semantics for a stable time limit bound vs. unbounded semantics}) that when $\tlimbound$ is globally stable for $\mathcal{M}$, then these two notions of truth become equivalent, and thus this exceptional assumption here becomes irrelevant. (Also note that the exceptional assumption can only be considered relevant  in situations where another ``nested'' embedded game can be played after playing $\mathbf{G}$).}.
\begin{itemize}[leftmargin=*]
\item If ${\pl}=\ctr$, then there is a winning strategy $\sigma_\pl$ in $\mathbf{G}$ iff there is $\initlim<\tlimbound$ and a timed winning strategy $(\sigma_\pl',t)$ in $\mathbf{G}[\initlim]$.

\item If ${\pl}\neq\ctr$, then there is a winning strategy $\sigma_\pl$ in $\mathbf{G}$ iff there is $\sigma_\pl'$ which is a winning strategy in $\mathbf{G}[\initlim]$ for every $\initlim<\tlimbound$.
\end{itemize}
\end{proposition}
\begin{proof}
Suppose first that ${\pl}=\ctr$. If $(\sigma_{\pl}',t)$ is a timed winning strategy in $\mathbf{G}[\initlim]$ for some time limit $\initlim<\tlimbound$, then by Proposition~\ref{the: Canonical strategy}, $(\tau_\pl,t_{can})$ is a timed winning strategy in $\mathbf{G}[\initlim]$. Now the strategy $\tau_\pl$ will also be a winning strategy in $\mathbf{G}$ (note that infinite plays are impossible with $\tau_\pl$ and that $\tau_\pl$ depends on states only). 

For the other direction, suppose that there is a winning strategy $\sigma_{\pl}$ for $\pl$ in $\mathbf{G}$. We assume, for the sake of contradiction, that $\pl$ does not have a timed winning strategy in $\mathbf{G}[\initlim]$ for any $\initlim<\tlimbound$. Hence, by Proposition~\ref{the: Winning time labels}, we have $\mathcal{L}_{\pl}(\initst)=\lose$. As consequence of Proposition~\ref{the: Correspondence of the winning time labels}, we must now have $\mathcal{L}_{\opl}(\initst)=\win$.
Let $\tlimbound':=\tlimbound\!+\!1$. We  construct an $\infty$-canonical strategy $\tau_\opl^\infty$ for $\opl$ by using the strategies that correspond to embedded games with the time limit bound $\tlimbound'$. Since $\tlimbound$ is stable, the winning time labels of the states will not change when $\tlimbound$ is increased and, in particular, the state $q$ will still have the value $\win$ for $\opl$. Now, the assumptions of Lemma~\ref{the: Preserving the label win} hold, and thus we can use it to deduce that all (finite) plays of $\mathbf{G}$ with $\tau_\opl^\infty$ will end at a state that has the label $\win$ for $\opl$. But in order to lose $\mathbf{G}$, the player $\opl$ should end up at a state with the label $0$. Hence $\tau_\opl^\infty$ must be a winning strategy for $\opl$ in $\mathbf{G}$ (note here that $\tau_\pl^\infty$ depends on states only and thus may be used in $\mathbf{G}$). This is a contradiction since we assumed that $\pl$ has a winning strategy in $\mathbf{G}$.

%\medskip

Suppose then that ${\pl}\neq\ctr$. If there is a winning strategy $\sigma_{\pl}$ for $\pl$ in $\mathbf{G}$, then we can define $\sigma_{\pl}'(\gamma,q)=\sigma_\pl(q)$ for every $\gamma<\Gamma$, whence $\sigma_\pl'$ will be a winning strategy in $\mathbf{G}[\initlim]$ for every time limit $\initlim<\tlimbound$. 

For the other direction, suppose that there is a strategy $\sigma_\pl'$ which is a winning strategy in $\mathbf{G}[\initlim]$ for every time limit $\initlim<\tlimbound$. Then, by Proposition~\ref{the: Winning time labels}, we have $\mathcal{L}_{\pl}(\initst)=\win$. As above, we can now increase $\tlimbound$ to $\tlimbound+1$ in order to construct an $\infty$-canonical strategy $\tau_\pl^\infty$ for $\pl$ in $\mathbf{G}$. With the same reasoning as above, $\tau_\pl^\infty$ will be a winning strategy for $\pl$ in $\mathbf{G}$.
\end{proof}

By using Proposition~\ref{the: Bounded vs. unbounded}, we can now prove the equivalence of unbounded and $\tlimbound$-bounded game-theoretic semantics for a stable time limit bound $\tlimbound$.

\begin{theorem}\label{the: Bounded semantics for a stable time limit bound vs. unbounded semantics}
Let $\mathcal{M}$ be a \CGM, $q\in\St$ and $\varphi$ an \ATL-formula. Suppose that $\tlimbound>0$ is a globally stable time limit bound for $\mathcal{M}$. Then we have
\[\mathcal{M},q\models_u^g\varphi \, \text{ iff } \, \mathcal{M},q\models_{\tlimbound}^g\varphi.\]
\end{theorem}

\begin{proof}
We prove by induction on $\varphi$ that Eloise has a winning strategy in $\evalgame(\mathcal{M},q,\varphi)$ if and only if she has a winning strategy in $\evalgame(\mathcal{M},q,\varphi,\tlimbound)$. Since the rules for unbounded and bounded evaluation games differ only for positions that lead to embedded games, we only need to consider the cases where $\varphi=\psi\U\theta$ or $\varphi=\psi\Rl\theta$.

Let $\mathbf{G}$ be the unbounded embedded game that arises from $\varphi$. Since $\tlimbound$ is globally stable for the model $\mathcal{M}$, it is stable for $\mathbf{G}$. 
By the inductive hypothesis, the equivalence of the two semantics holds for the formulae $\psi$ and $\theta$. We may thus assume that the winning condition of the exit positions of $\mathbf{G}$ has been defined by using $\tlimbound$-bounded \GTS  instead of the unbounded \GTS (recall the corresponding assumption in Proposition~\ref{the: Bounded vs. unbounded}). Since $\tlimbound$-bounded evaluation games are determined (Proposition~\ref{the: Determinacy of the evaluation games}), it now follows that having a winning strategy in $\mathbf{G}$---or in $\mathbf{G}[\initlim]$ for some $\initlim<\tlimbound$---guarantees having a winning strategy in the corresponding evaluation game that is continued (recall Definition~\ref{def: Winning strategies in embedded game}).
The inductive step for $\varphi$ can thus be proven easily by using Proposition~\ref{the: Bounded vs. unbounded}.

If Eloise is the controller in $\mathbf{G}$, then by Proposition~\ref{the: Bounded vs. unbounded} she has a winning strategy in $\mathbf{G}$ if and only if there is some time limit $\tlim<\tlimbound$ such that she has a (timed) winning strategy in $\mathbf{G}[\tlim]$. And if Eloise is not the controller in $\mathbf{G}$, then by Proposition~\ref{the: Bounded vs. unbounded}, she has a winning strategy in $\mathbf{G}$ if and only if she has a winning strategy in $\mathbf{G}[\tlim]$ for every $\tlim<\tlimbound$ (chosen by the Abelard).
The claim for $\varphi$ thus follows.
\end{proof}

Since $\tlimbound$-bounded evaluation games are positionally determined for any time limit bound $\tlimbound>0$ (by Proposition~\ref{the: Determinacy of the evaluation games}), in particular they are determined for such $\tlimbound$ that are globally stable for a given model $\mathcal{M}$. Hence Theorem~\ref{the: Bounded semantics for a stable time limit bound vs. unbounded semantics} implies that unbounded evaluation games (and unbounded embedded games) are positionally determined too.
By these results, we see that even if we had defined memory-based strategies for bounded or unbounded evaluation games, the resulting semantics would have remain equivalent to the current one.

As a direct corollary of Theorem~\ref{the: Bounded semantics for a stable time limit bound vs. unbounded semantics} we obtain the equivalence of unbounded and bounded game-theoretic semantics.

\begin{corollary}\label{the: Equivalence of unbounded and bounded semantics}
Let $\mathcal{M}$ be a \CGM, $q\in\St$ and $\varphi$ an \ATL-formula. Then 
\[
	\mathcal{M},q\models_u^g\varphi \,\text{ iff }\, \mathcal{M},q\models_b^g\varphi.
\]
\end{corollary}

\begin{proof}
Assume that $\card(\mathcal{M})=\kappa$. By Corollary~\ref{cor: Stable time limit bound for a model}(1), the ordinal $2^\kappa$ is globally stable for $\mathcal{M}$ and thus the claim follows from Theorem~\ref{the: Bounded semantics for a stable time limit bound vs. unbounded semantics}.
\end{proof}

The proof above relies on the (global) stability of the time limit bound $\tlimbound=2^{\card(\mathcal{M})}$. We know by Proposition~\ref{the: Stable time limit bound for a model} that the time limit bound $\tlimbound=\rbb(\mathcal{M})$ would also suffice for this proof. But by Proposition~\ref{prop:exII}, for any time limit bound $\tlimbound<\rbb(\mathcal{M})$, there are cases when $\tlimbound$-bounded semantics is not equivalent with the unbounded \GTS. Therefore, we cannot use any \emph{fixed} value of $\tlimbound$ in order to obtain this equivalence.

Even though the finitely bounded semantics is not equivalent to the unbounded semantics (see Example ~\ref{ex: bounded vs. finitely bounded}),
the three semantics become equivalent on a natural class of models.
\begin{theorem}\label{the: Equivalence of finitely bounded and bounded semantics}
Let $\mathcal{M}$ be an image-finite \CGM, $q\in\St$ and $\varphi$ an \ATL-formula. Then 
\[
	\mathcal{M},q\models_f^g\varphi 
	\,\text{ iff }\, \mathcal{M},q\models_u^g\varphi 
	\,\text{ iff }\, \mathcal{M},q\models_b^g\varphi.
\]
\end{theorem}

\begin{proof}
Since $\mathcal{M}$ is image-finite, by Corollary~\ref{cor: Stable time limit bound for a model}(2), the ordinal $\omega$ is a globally stable time limit bound for $\mathcal{M}$ and thus the claim follows from Theorem~\ref{the: Bounded semantics for a stable time limit bound vs. unbounded semantics} and Corollary \ref{the: Equivalence of unbounded and bounded semantics}.
\end{proof}

Note that in image-finite models all ordinal-valued winning time labels are finite. Thus the controller would gain nothing from being able to use infinite ordinals in embedded games.

%%%%%%%%%%%%%%%%%%%%%
\section{Game-theoretic vs compositional semantics}
\label{sec:Comparing game-theoretic and compositional semantics}
%%%%%%%%%%%%%%%%%%%%%

\subsection{Equivalence between the unbounded GTS and the compositional semantics}

We show that the unbounded game-theoretic semantics for \ATL is equivalent to the standard compositional semantics of \ATL.
\begin{theorem}\label{the: Equivalence of unbounded compositional and game-theoretic semantics}
Let $\mathcal{M}$ be a \CGM, $\qzero\in\St$ and $\varphi$ an \ATL-formula. 
Then 
\[\mathcal{M},\qzero\models\varphi \,\text{ iff }\,  \mathcal{M},\qzero\models_{u}^g\varphi.\]
\end{theorem}

\begin{proof}
(\textit{Sketch})
The equivalence can be proven by induction on $\varphi$. The cases $\varphi=p$ and $\varphi=\psi\vee\theta$ are easy. For the case $\varphi=\neg\psi$ we need to use the determinacy of the evaluation games from Proposition~\ref{the: Determinacy of the evaluation games}. In the cases for strategic operators $\coop{A}\X\psi$, $\coop{A}\psi\U\theta$ and $\coop{A}\psi\Rl\theta$, we construct a strategy $\sigma_\pl$ for the embedded (or one step) game using the collective strategy $S_A$ and vice versa.
For a detailed proof, see the appendix. 
\end{proof}

By Theorem~\ref{the: Equivalence of unbounded and bounded semantics} we immediately obtain the following corollary. 
\begin{corollary}
Let $\mathcal{M}$ be a \CGM, $\qzero\in\St$ and $\varphi$ an \ATL-formula. 
Then 
\[\mathcal{M},\qzero\models\varphi \,\text{ iff }\,  \mathcal{M},\qzero\models_{b}^g\varphi.\]
\end{corollary}

However, as we have discussed in Section~\ref{ssec: Unbounded vs. bounded}, $\tlimbound$-bounded \GTS is not equivalent with the unbounded \GTS for any fixed value of the time limit bound $\tlimbound$. Therefore the different values of $\tlimbound$ lead to different semantic systems that are all non-equivalent to the standard compositional semantics.

%%%%%%%%%%%%%%%%%%%%%
\subsection{Finitely bounded compositional semantics}
\label{ssec: Finitely bounded compositional semantics}

As shown earlier (Example~\ref{ex: bounded vs. finitely bounded}), the finitely bounded game-theoretic semantics is not equivalent to the standard compositional semantics of \ATL. However, it can still be shown equivalent to a natural semantics,
to be defined next, which we call \emph{finitely bounded compositional semantics}.

\begin{definition}\label{finboundcomposemantics}
Let $\mathcal{M} = (\Agt, \St, \Prop, \Act, d, o, v)$ be a \CGM, $q\in\St$ and $\varphi$ an \ATL-formula. The truth of $\varphi$ in $\mathcal{M}$ at $q$ according to \defstyle{finitely bounded compositional semantics}, denoted by $\mathcal{M},q\models_{f}\varphi$, is defined recursively as follows:
\begin{itemize}[leftmargin=*]
\item The truth conditions for $p\in\Prop$, $\neg\psi$, $\psi\vee\theta$ and $\coop{A} \X\psi$ are as in the standard compositional semantics of \ATL (see Definition~\ref{def: compositional semantics for ATL}).
\item $\mathcal{M},q\models_{f}\coop{A}\psi \U\theta$ iff there exists $n<\omega$ and a collective strategy $S_A$ such that for each $\Lambda\in\paths(q,S_A)$, there is some $i\leq n$\, such that $\mathcal{M},\Lambda[i]\models_{f}\theta$ and $\mathcal{M},\Lambda[j]\models_{f}\psi$ for every $j < i$.
\item $\mathcal{M},q\models_{f}\coop{A}\psi \Rl\theta$ iff for every $n<\omega$, there exists  a collective strategy $S_{A,n}$ such that for each $\Lambda\in\paths(q,S_{A,n})$ and $i\leq n$, either $\mathcal{M},\Lambda[i]\models_{f}\theta$ or there is $j < i$ such that $\mathcal{M},\Lambda[j]\models_{f}\psi$.
\end{itemize}
\end{definition}
For $\F$ and $\G$ we obtain the following derived truth conditions:
\begin{itemize}[leftmargin=*]
\item $\mathcal{M},q\models_b\coop{A} \F\psi$ iff there exists $n<\omega$ and  a collective strategy $S_A$ such that for each $\Lambda\in\paths(q,S_A)$ there is $i\leq n$\, such that $\mathcal{M},\Lambda[i]\models_{f}\psi$.
\item $\mathcal{M},q\models_{f}\coop{A} \G\psi$ iff for every  $n\geq 0$\, there exists  a collective strategy $S_{A,n}$ such that for each $\Lambda\in\paths(q,S_{A,n})$ and $i\leq n$ we have $\mathcal{M},\Lambda[i]\models_{f}\psi$.
\end{itemize}
There exists an interesting conceptual link between
\emph{for-loops and while-loops} on one
hand  and the finitely bounded and unbounded semantics of \ATL on the other hand.
This link is perhaps easiest to
understand via the game-theoretic approaches to \ATL semantics, but it ultimately
makes no difference whether compositional or game-theoretic semantics is considered.
The principal difference between the two types of loops is that for-loops require an
input parameter, let us call it an \emph{iterator}, which determines how many times the
loop is executed. The iterator can be an input from elsewhere in the program or
even a fixed constant (although in the latter case, it is not 
clear whether the construct in question should be regarded as a genuine loop).
Similarly, in the finitely bounded semantics, the execution of strategic formulae
requires a finite input integer which limits the number of steps to be taken in the model and is
clearly analogous to the iterator. While-loops and the unbounded semantics are
similarly analogous to each other, both being essentially iterative processes but without an
iterator bounding the number of executions.

%%%%%%%%%%%%%%%%%%%%%
%\subsection{Understanding the finitely bounded compositional semantics}

%
It turns out that the finitely bounded compositional semantics simplifies, in a sense, the meaning of the temporal operators $\atlu$ and $\atlr$, and in particular of $\atlf$ and $\atlg$, by reducing them to simple infinitary disjunctions and conjunctions of iteration patterns using 
$\atlx$ (see below). The iteration 
patterns correspond to finite approximations of the respective fixed point
characterisations (cf. \cite{GorDrim06}). Recall that the difference between finitely bounded and
the standard compositional semantics is essentially
described by Example~\ref{ex: bounded vs. finitely bounded}.  
%without proofs, which will be provided in a follow-up work. 

To characterize the finitely bounded semantics, we now
define series of 'iteration operators'. For the sake of better illustration, we will only treat the cases of $\atlg$ and $\atlu$ here.

In what follows $A$ denotes an arbitrary coalition of agents.  
First, we define the following operators on \ATL formulae: 
\begin{itemize}[leftmargin=*]
\item $\mathbf{G}_{A;\theta}(\varphi) := \theta \land \coop{A}\X \varphi$ 
\item  $\mathbf{U}_{A;\psi,\theta}(\varphi) :=  \theta \lor (\psi \land \coop{A}\X \varphi)$ 
\end{itemize}
%
%\begin{itemize}[leftmargin=*]
%\item $\mathbf{G}_{A}(\theta) = \theta \land \coop{A}\X  \theta$. 
%\item  $\mathbf{U}_{A}(\psi,\theta) =  \psi \lor (\theta \land \coop{A}\X \psi \U \theta$.  
%\end{itemize}
%
Now, we define the following recursively over $n \in \bbN$: 
\begin{itemize}[leftmargin=*]
\item $\mathbf{G}^{0}_{A}(\theta) := \theta$;  \ 
$\mathbf{G}^{n+1}_{A}(\theta) := \mathbf{G}_{A;\theta}(\mathbf{G}^{n}_{A}(\theta))$  

\item $\mathbf{U}^{0}_{A}(\psi,\theta) := \theta$; \  
$\mathbf{U}^{n+1}_{A}(\psi, \theta) := \mathbf{U}_{A;\psi,\theta}(\mathbf{U}^{n}_{A}(\psi,\theta))$.  
\end{itemize}

\begin{lemma} 
\label{lem1}
For every \CGM $\mathcal{M}$ and $q\in \mathcal{M}$, the following hold: 

\begin{enumerate}[leftmargin=*]
\item 
$\mathcal{M},q\models_{f}\coop{A}\G  \theta$ \ iff \ 
 $\mathcal{M},q\models_{f} \mathbf{G}^{n}_{A}(\theta)$ for every $n \in \bbN$.

Furthermore, if $\theta$ only contains occurrences of
strategic operators of the type $\coop{B}\atlx$, \\ then  \quad
 $\mathcal{M},q\models_{f} \coop{A}\atlg \theta $ \ iff \ 
 $\mathcal{M},q\models \mathbf{G}^{n}_{A}(\theta)$ for every $n \in \bbN$.

\item 
$\mathcal{M},q\models_{f}\coop{A}\psi \U \theta$ \ iff \  $\mathcal{M},q\models_{f} \mathbf{U}^{n}_{A}(\psi, \theta)$ for some $n \in \bbN$.

Furthermore, if $\psi$ and $\theta$ only contain occurrences of strategic operators of the type $\coop{B}\atlx$, then \quad
$\mathcal{M},q\models_{f} \coop{A}\psi\atlu \theta$ \ iff \
$\mathcal{M},q\models \mathbf{U}^{n}_{A}(\psi, \theta)$ for some $n \in \bbN$.

\end{enumerate}
\end{lemma}

The equivalences above suggest how to modify the axioms characterising these operators provided in \cite{GorDrim06}, viz. to replace the biconditionals 

\smallskip
 \textbf{FP$_{\atlg}$:}
$(\varphi\wedge\llangle A\rrangle \X\llangle A\rrangle \G\varphi)\leftrightarrow
\llangle A\rrangle \G\varphi$
\ \ \ and 

\smallskip
\textbf{FP$_{\atlu}$:}
$\llangle A\rrangle \varphi \U\psi \leftrightarrow (\psi\vee(\varphi\wedge\llangle A\rrangle \X\llangle A\rrangle \varphi \U\psi))$ 

\smallskip
by the respective implications 
\smallskip

\textbf{PreFP$_{\atlg}$:}
$(\varphi\wedge\llangle A\rrangle \X\llangle A\rrangle \G\varphi)\rightarrow
\llangle A\rrangle \G\varphi$ \ \ \ and 

\smallskip
\textbf{PostFP$_{\atlu}$:}
$\llangle A\rrangle \varphi \U\psi\rightarrow(\psi\vee(\varphi\wedge\llangle A\rrangle \X\llangle A\rrangle \varphi \U\psi))$.

\medskip 

Lemma \ref{lem1} has also partly been discussed in the follow-up work \cite{FBS-ATL-2018}, where we present a sound and complete infinitary axiomatization for the finitely bounded compositional semantics of \ATL
by adding respective infinitary axiom schemes and inference rules.
%We also conjecture that the the resulting modified axiomatic system based on \cite{GorDrim06} is sound and complete with respect to the finitely bounded compositional semantics. 

%%%%%%%%%%%%%%%%%%%%%
\subsection{Equivalence between the finitely bounded compositional and game-theoretic semantics}

To prove equivalence between the finitely bounded compositional semantics and the finitely bounded \GTS, we  need to show that it is sufficient to consider strategies in the embedded games that depend on states only. This property will be needed because collective strategies for coalitions ($S_A$) in the compositional semantics of \ATL formulae are positional (according to Definition~\ref{def: compositional semantics for ATL}).

\begin{lemma}
\label{the: Time uniform strategies}
If Eloise has a winning strategy $\strev_{\bf E}$ in a finitely bounded evaluation game $\evalgame(\mathcal{M}, \qzero, \varphi, \omega)$, then she has a winning strategy $\strev_{\bf E}'$ which uses, in every bounded embedded game, a strategy $\sigma_{\bf E}$ that depends only on states.
\end{lemma}

\begin{proof}
In summary, the idea is that if $\bf E\!=\!\ctr$ in some embedded game, Eloise may play with the canonical strategy $\tau_{\bf E}$ that depends only on states. If $\bf E\!\neq\!\ctr$ and Abelard chooses $n<\omega$ as the time limit, Eloise may play with the $n$-canonical strategy $\tau_{\bf E}^n$ that only depends on states. A more detailed explanation follows.

Suppose that $\strev_{\bf E}$ is a winning strategy for Eloise in a finitely bounded evaluation game $\evalgame(\mathcal{M}, \qzero, \varphi, \omega)$. 
%We first observe that since $\strev_{\bf E}$ is a winning strategy, all strategies $\sigma_{\bf E}$ that are assigned by $\strev_{\bf E}$ must be winning strategies in the corresponding embedded games.
Let $\mathbf{G}=\embgame(\vrf,\ctr,A,\initst,\psi_{\ctr},\psi_{\octr})$ be an embedded game that is related to some position $\pos$ in the evaluation game $\evalgame(\mathcal{M}, \qzero, \varphi, \omega)$.  
Suppose first that ${\bf E}=\ctr$. Then $\strev_{\bf E}(\pos)=(n,\sigma_{\bf E})$ for some $n<\omega$ (the timer $t$ is not used in the finitely bounded case, and thus it can be omitted here). Since $\strev_{\bf E}$ is a winning strategy in the evaluation game, $\sigma_{\bf E}$ must be winning strategy in the embedded game $\mathbf{G}[n]$. By Proposition~\ref{the: Canonical strategy}, the canonical strategy $\tau_{\bf E}$ is a winning strategy in $\mathbf{G}[n]$ (the canonical timer is not needed here).
Suppose now that ${\bf E}\neq\ctr$. Then $\strev_{\bf E}(\pos)$ maps every $n<\omega$ to some strategy $\sigma_{{\bf E},n}$. Since $\strev_{\bf E}$ is a winning strategy in the evaluation game, $\sigma_{{\bf E},n}$ must be winning strategy in the corresponding bounded embedded game $\mathbf{G}[n]$. By Lemma~\ref{the: n-canonical strategy}\,, for every $n<\omega$, the $n$-canonical strategy $\tau_{\bf E}^n$ is a winning strategy in $\mathbf{G}[n]$.

Thus, it is easy to see that we can construct a strategy $\strev_{\bf E}'$ for Eloise in such a way that it only uses canonical strategies when ${\bf E=\ctr}$ and maps all $n<\omega$ to $n$-canonical strategies when ${\bf E}\neq\ctr$. Since these strategies depend on states only, the claim now follows for $\strev_{\bf E}'$.
\end{proof}

Using the previous lemma, we can now prove the equivalence between the finitely
bounded compositional and game-theoretic semantics using a similar induction as the one in the proof of Theorem~\ref{the: Equivalence of unbounded
compositional and game-theoretic semantics}.
See the proof in the appendix.

\begin{theorem}\label{the: Equivalence of finitely bounded compositional and game-theoretic semantics}
Let $\mathcal{M}$ be a \CGM, $\qzero\in\St$ and $\varphi$ an \ATL-formula.  Then 
\[\mathcal{M},\qzero\models_{f}\varphi \,\text{ iff }\,  \mathcal{M},\qzero\models_{f}^g\varphi.\] 
\end{theorem}

As we have seen in Example~\ref{ex: bounded vs. finitely bounded}, the fixed point characterisation of the temporal operator $\F$ fails with finitely bounded semantics (both compositional and game-theoretic). But since finitely bounded \GTS is equivalent with the standard
semantics in image-finite models (and in particular, finite models), the formula $\neg\coop{A} \F p \wedge (p \lor \coop{A} \X\coop{A} \F p)$ is satisfied only in infinite models. Therefore we obtain the following result:

\begin{proposition}
\ATL with finitely bounded semantics lacks the finite model property.
\end{proposition}

It is well-known that \ATL with the standard semantics has the finite model property. By this fact, and the equivalence of the finitely bounded and the standard semantics in finite models, we obtain the following interesting consequence.

\begin{corollary}
All validities of \ATL with finitely bounded semantics are also valid with the standard semantics.
\end{corollary}

%%%%%%%%%%%%%%%%%%%%%
\subsection*{Concluding remarks}
\label{sec: Conclusion}

Game-theoretic approaches have proved to be a natural and fruitful idea in many areas of logic. 
The principal aim of this paper has been to develop a conceptual and technical framework for
systems of game-theoretic semantics for \ATL. 
%We have shown that \GTS is particularly natural and suitable to apply to the logic \ATL which is itself %designed to reason about game-based models. 
%
%
%
The game-theoretic perspective appears in the setting of \ATL on two semantic levels: on the object level, in the standard semantics of the strategic operators, and on the meta-level, where game-theoretic logical semantics is applied to formulae. We have unified these two perspectives in the semantic evaluation games designed here for \ATL and argued that the resulting versions of game-theoretic semantics are conceptually and technically natural from logical as well as game-theoretic perspectives. Our bounded \GTS provides a new framework where truth of \ATL-formulae can be established in finite time---even on infinite models---despite the fact that \ATL-formulae involve temporal operators that refer to infinite computations.

The ideas and results in the present work can be naturally extended to other richer languages. In particular, the recent paper \cite{GKR-AAMAS2017} introduces  \GTS for the extended logic \ATLplus. Because of the more expressive language of  \ATLplus, which generally requires memory-based strategies, the \GTS for that logic is considerably more complex, while formally extending the \GTS for \ATL developed here.
%This is a good illustration of the connection between expressiveness and conceptual and technical %complexity of computational procedures solving logical decision problems. 
%
Another recent paper, \cite{mucalc}, introduces a bounded semantics for
the modal $\mu$-calculus, directly inspired by the approach in the current paper
based on ordinals. This leads to games alternative to the usual parity games
where the durations of plays are always guaranteed to be finite. 
However, \cite{mucalc} discusses the case for standard modalities rather than the 
strategic operators used in the current article. 
%
%Further extensions of variations of \GTS and related bounded semantics, e.g. to \ATLs, the %alternating mu-calculus and strategy logics \cite{tocl/MogaveroMPV14},  would certainly be %interesting and desirable.  

The finitely bounded compositional semantics for \ATL and other expressive branching time logics, naturally emerging from the present work, is of particular interest for several reasons and deserves a separate study of its own. As a first step in that direction, in a recent work \cite{GorankoKR17}, we have studied in more detail the finitely bounded semantics for the computational tree logic \CTL (which can be seen as a single-agent variant of \ATL).  In particular, we have presented there sound and complete tableaux systems for checking the validities of \CTL with finitely bounded semantics, 
one of which is infinitary, while the other is finitary and provides an EXPTIME decision procedure for the satisfiability problem for that logic. Thus we establish EXPTIME-completess of the logic, the same complexity as that of standard \CTL. In \cite{FBS-ATL-2018}, we extend these results and also provide an infinitary complete axiomatization for \ATL with finitely bounded semantics.

We will now briefly address two important issues on how the present work links to other research.
%The first issue was already mentioned in the introduction and was
%originally pointed out by a reviewer of \cite{GKR-AAMAS2016}.
The first issue is---as already pointed out in the introduction---that some of our technical results could have alternatively been derived using some general results for coalgebraic modal logic. This is because concurrent game models can be viewed as coalgebras for a \emph{game functor} defined  in \cite{vene2011}, and the fixed-point extension of the coalitional coalgebraic modal logic for this functor links to \ATL in a natural way.
Game-theoretic semantics has been developed for coalgebraic fixed-point logics in \cite{vene2006,pattinson2009,vene2010}
and could indeed be used to obtain some of our results concerning the unbounded game-theoretic semantics. However, as general and powerful that approach is, it would not be very helpful for readers not familiar with coalgebras and coalgebraic modal logic, so the more direct and self-contained approach adopted in the present paper has its benefits.
Moreover, our work on bounded and finitely bounded semantics is not directly related to existing work in coalgebraic modal logic, though even there 
some natural shortcuts based on background theory could have been used.
For example, using K\"{o}nig's Lemma, it is easy to prove that the finitely bounded and unbounded game-theoretic semantics are equivalent on image-finite models.

The second issue concerns links to alternating Turing-machines and automata, 
%other game-like formalisms 
and some possible applications. In addition to defining variants of \GTS for \ATL, we have also identified a number of related useful concepts such as, e.g., $n$-canonical and $\infty$-canonical strategies as well as winning time labels for embedded games. All these could be useful also outside the context of \ATL. In particular, we believe that the finitely bounded game-theoretic semantics can be used to connect different kinds of extensions of \ATL to alternating Turing-machines (and alternating automata) as well as different kinds of games. Recalling that the size of the model domain is a sufficient initial time limit for semantic games in finite models, it is easy to see that evaluation games directly translate in polynomial time to equivalent alternating reachability games, thereby providing a direct proof of the $\mathrm{PTIME}$ upper bound of the model
checking problem of \ATL. A similar approach works also in the context of various extensions of \ATL, such as  \ATLplus, as demonstrated in \cite{GKR-AAMAS2017}, where some known and some new results on the complexity of model checking \ATLplus and some of its fragments have been obtained by using these links.

\subsubsection*{Acknowledgments}
The work of Valentin Goranko was supported by a research grant 2015-04388 of the Swedish Research Council.
The work of Antti Kuusisto was supported by the ERC grant 647289 ``CODA."
We thank the anonymous reviewers for helpful comments and references.

% Bibliography
\bibliographystyle{plain}
\bibliography{VG-ATL}
                             % Sample .bib file with references that match those in
                             % the 'Specifications Document (V1.5)' as well containing
                             % 'legacy' bibs and bibs with 'alternate codings'.
                             % Gerry Murray - March 2012

% Appendix
%\appendix
\section*{APPENDIX: SOME TECHNICAL PROOFS}
\label{sec:Appendix}

\medskip

\setcounter{section}{1}

\subsection*{Proof of Lemma~\ref{the: Regular strategies}}

We define the strategy $\sigma_\pl'$ for any
configuration $(\tlim,q)$, where $\tlim<\tlimbound$ and $q\in\St$, as follows.
\begin{itemize}[leftmargin=*]
\item If $(\sigma_\pl,t)$ is not a timed winning strategy in $\mathbf{G}[q,\tlim']$ for any $\tlim'\leq\tlim$, set $\sigma_\pl'(\tlim,q) = \void$.	
\item Else, set $\sigma_\pl'(\tlim,q) = \sigma_\pl(\tlim',q)$, where $\tlim'\leq\tlim$ is the smallest ordinal such that $(\sigma_\pl,t)$ is a timed winning strategy in $\mathbf{G}[q,\tlim']$.
\end{itemize}
We define the timer $t'$ for any pair $(\tlim,q)$, where $\tlim<\tlimbound$ is a
limit ordinal and $q\in\St$, in the following way.
\begin{itemize}[leftmargin=*]
\item If $(\sigma_\pl,t)$ is not a timed winning strategy in $\mathbf{G}[q,\tlim']$ for any $\tlim'<\tlim$, set $t'(\tlim,q) = \void$.
\item Else $t'(\tlim,q)=\tlim'$, where $\tlim'$ is any ordinal such that  $\tlim'<\tlim$ and $(\sigma_\pl,t)$ is a timed winning strategy in $\mathbf{G}[q,\tlim']$.
\end{itemize}
We prove by transfinite induction on $\tlim<\tlimbound$ for every $q\in\St$ that if $(\sigma_\pl,t)$ is a timed winning strategy in $\mathbf{G}[q,\delta]$ for some $\delta\leq\tlim$, then $(\sigma_\pl',t')$ is a timed winning strategy in $\mathbf{G}[q,\tlim]$. The claim then follows. 

Let the inductive hypothesis be that the claim holds for every $\tlim'<\tlim$ and suppose that $(\sigma_\pl,t)$ is a timed winning strategy in $\mathbf{G}[q,\delta]$ for some $\delta\leq\tlim$. Let $\delta'<\tlimbound$ be the smallest ordinal such that  $(\sigma_\pl,t)$ is a timed winning strategy in $\mathbf{G}[q,\delta']$. Now $\delta'\leq\delta\leq\tlim$ and $\sigma_\pl'(\tlim,q)=\sigma_\pl(\delta',q)$.

Suppose first that $\delta'=0$, whence $(\sigma_\pl,t)$ is a timed winning strategy in $\mathbf{G}[q,0]$. Now $\sigma_\pl'(\tlim,q)=\sigma_\pl(0,q)=\psi_\ctr$ and thus $(\sigma_\pl',t')$ is a timed winning strategy in $\mathbf{G}[q,\tlim]$.
Suppose then that $\delta'>0$, whence we must have $\sigma_\pl(\delta',q)\neq\psi_\ctr$. Let $q'\in\St$ be any possible successor state of $q$ when $\pl$ follows $\sigma_\pl(\delta',q)$.
Now $(\sigma_\pl,t)$ must be a timed winning strategy in
$\mathbf{G}[q',\delta'']$ for some ordinal $\delta''<\delta'$ (if $\delta'$ is a limit ordinal, then $\delta''=t(\delta',q')$, and if $\delta'$ is a successor ordinal, then $\delta''\!=\delta'\!-\!1$). Since $\delta'\leq\tlim$, we have $\delta''<\tlim$.

For the inductive step, suppose first that $\tlim$ is a successor ordinal. 
Since we have $\delta''<\tlim$, we infer by the inductive hypothesis that $(\sigma_\pl',t')$ is a timed winning strategy in $\mathbf{G}[q',\tlim\!-\!1]$. Hence we see that $(\sigma_\pl',t')$ must be a timed winning strategy in $\mathbf{G}[q,\tlim]$.
Suppose then that $\tlim$ is a limit ordinal. 
Since $(\sigma_\pl,t)$ is a timed winning strategy in $\mathbf{G}[q',\delta'']$, the value of $t'(\tlim,q')$ is defined such that $t'(\tlim,q')<\tlim$ and $(\sigma_\pl,t)$ is a timed winning strategy in $\mathbf{G}[q',t'(\tlim,q')]$. Thus, by the inductive hypothesis, $(\sigma_\pl',t')$ is a timed winning strategy in $\mathbf{G}[q',t'(\tlim,q')]$. Hence we see that $(\sigma_\pl',t')$ must be a timed winning strategy in $\mathbf{G}[q,\tlim]$.
\qed

\subsection*{Proof of Proposition~\ref{the: Winning time labels}}

\noindent
1. We first consider  the case where ${\pl}=\ctr$. 

\smallskip
i) Suppose first that $\mathcal{L}_{\pl}(q)=\tlim<\tlimbound$. By Definition~\ref{def: Winning time labels} there is some strategy $\sigma_{\pl}$ whose strategy label $l(q,\sigma_{\pl})$ is $\tlim$. Thus there is some timer $t$ such that the pair $(\sigma_{\pl},t)$ is a timed winning strategy for $\pl$ in $\mathbf{G}[q,\tlim]$. By 
Lemma~\ref{the: Regular strategies} there is a pair $(\sigma_\pl',t')$ which is a timed winning strategy in $\mathbf{G}[q,\tlim']$ for any $\gamma'$ such that  $\tlim\leq\tlim'<\tlimbound$. 
If there existed some timed winning strategy $(\sigma_\pl'',t'')$ for $\pl$ in $\mathbf{G}[q,\tlim']$ for some $\tlim'<\tlim$, then we would have $l(q,\sigma_\pl'')\leq\tlim'$ and thus $\mathcal{L}_\pl(q)\leq\tlim'<\tlim$, which is a contradiction.

For the other direction, suppose that there is a pair $(\sigma_\pl,q)$ which is a timed winning strategy in $\mathbf{G}[q,\tlim']$ for any $\tlim'$ such that $\tlim\leq\tlim'<\tlimbound$, but there is no timed winning strategy for $\pl$ in $\mathbf{G}[q,\tlim']$ for any $\tlim'<\tlim$. Now $l(q,\sigma_\pl)=\tlim$, and for any other strategy $\sigma_\pl'$, we have either $l(q,\sigma_{\pl}')=\lose$ or $l(q,\sigma_{\pl}')\geq\tlim$. Hence the smallest ordinal value for the strategy labels at $q$ is the ordinal $\tlim$, and thus we have $\mathcal{L}_\pl(q)=\tlim$.

\smallskip
ii) If $\mathcal{L}_{\pl}(q)=\lose$, then $l(q,\sigma_{\pl})=\lose$ for every strategy $\sigma_{\pl}$ of $\pl$. Hence none of the strategy-timer pairs $(\sigma_\pl,t)$ is a timed winning strategy for $\pl$ in $\mathbf{G}[q,\tlim]$ for any $\tlim<\tlimbound$. Conversely, if there is no timed winning strategy $(\sigma_\pl,t)$ in $\mathbf{G}[q,\tlim]$ for any $\tlim<\tlimbound$, then we have $l(q,\sigma_{\pl})=\lose$ for every $\sigma_{\pl}$ and thus $\mathcal{L}_{\pl}(q)=\lose$.

\medskip

\noindent
2. We now consider the case where ${\pl}\neq\ctr$. 

\smallskip
i) If $\mathcal{L}_{\pl}(q)=\win$, then $l(q,\sigma_{\pl})=\win$ for some
strategy $\sigma_{\pl}$, i.e., $\sigma_{\pl}$ is a winning strategy in $\mathbf{G}[q,\tlim]$ for any time limit $\tlim<\tlimbound$. Conversely, if there is some $\sigma_{\pl}$ which is a winning strategy in $\mathbf{G}[q,\tlim]$ for every time limit $\tlim<\tlimbound$, then $l(q,\sigma_{\pl})=\win$ and thus $\mathcal{L}_{\pl}(q)=\win$.

\smallskip 
ii) Suppose that $\mathcal{L}_{\pl}(\initst)=\tlim<\tlimbound$, i.e., $\tlim$ is the supremum of the strategy labels $l(q,\sigma_{\pl})$. Suppose first that there is some $\sigma_{\pl}$ for which $l(\initst,\sigma_{\pl})=\tlim$. Now $\sigma_\pl$ is a winning strategy in $\mathbf{G}[q,\tlim']$ for any $\tlim'<\tlim$. Suppose then that there is no maximum value for the labels $l(q,\sigma_{\pl})$, whence $\tlim$ must be a limit ordinal. Let $\tlim'<\tlim$. Since $\tlim$ is the
least upper bound for the strategy labels $l(q,\sigma_{\pl})$, there must be
some strategy $\sigma_{\pl}'$ for which $l(q,\sigma_{\pl}')\geq\tlim'\!+\!1$, as otherwise $\tlim'\!+\!1$ would be a lower upper bound for the strategy labels.
We now observe that $\sigma_{\pl}'$ is a winning strategy in $\mathbf{G}[q,\tlim']$.

If there existed a winning
strategy $\sigma_{\pl}'$ for $\pl$ in $\mathbf{G}[q,\tlim']$ for some $\tlim'\geq\tlim$, then we would have $l(q,\sigma_{\pl}')>\tlim$, and thus $\mathcal{L}_{\pl}(q)>\tlim$. Hence there cannot be any winning strategy for ${\pl}$ in $\mathbf{G}[q,\tlim']$ for any $\tlim'$ such that $\tlim\leq\tlim'<\tlimbound$.

For the other direction, assume that for every $\tlim'<\tlim$, there exists a winning strategy for $\pl$ in $\mathbf{G}[q,\tlim']$, but there exists no winning strategy for ${\pl}$ in $\mathbf{G}[q,\tlim']$ for any $\tlim'$ such that $\tlim\leq\tlim'<\tlimbound$. 
If we had $\mathcal{L}_\pl(q)=\win$, we would end up with a contradiction by the (already proved) result concerning the label $\win$.
If we had $\tlim<\mathcal{L}_{\pl}(\initst)<\tlimbound$, then, by the (already proved) other direction of the current claim, there would have existed a winning strategy for $\pl$ in $\mathbf{G}[q,\tlim]$, which is a contradiction. 
And if we had $\mathcal{L}_{\pl}(q)<\tlim$, then, once again by the other direction of the current claim, there would not have been any winning strategy for $\pl$ in 
$\mathbf{G}[q,\mathcal{L}_\pl(q)]$, again a contradiction. 
Hence, the only possibility left is that $\mathcal{L}_{\pl}(\initst)=\tlim$.
\qed

\subsection*{Proof of Proposition~\ref{the: stability of labels}}

In order to simplify this proof, we will use the determinacy of $\tlimbound$-bounded evaluation games (Proposition~\ref{the: Determinacy of the evaluation games}). 
By observing our proofs, we can confirm that this does not lead to any circular argumentation. 
This is not surprising as Proposition~\ref{the: Determinacy of the evaluation games} is proven for an arbitrary time limit bound $\tlimbound>0$ without any assumptions on its stability. The claim of the current proposition (\ref{the: stability of labels}) is only needed for showing that certain time limit bounds are guaranteed to be stable (Propositions~\ref{prop: Stable bound for finite models} and \ref{the: Stable time limit bound for a model}). Note that by assuming the determinacy of bounded evaluation games, a player wins a bounded \emph{embedded} game if and only if (s)he has a winning strategy in the bounded \emph{evaluation} game that is continued (recall Definition~\ref{def: Winning strategies in embedded game}).

Let $\mathbf{G}$ be an (unbounded) embedded game that can occur within $\mathcal{G}(\mathcal{M},\qzero,\varphi)$. Now the corresponding bounded embedded games $\mathbf{G}[\tlim]$ (for different time limits $\tlim$) can occur within $\mathcal{G}_{\tlimbound}$ and $\mathcal{G}_{\tlimbound'}$. We need to show that the winning time labels for player $\pl$ in $\mathbf{G}$ are the same for all $q\in\St$ with respect to both $\tlimbound$ and $\tlimbound'$. 
We first show that any exit position of $\mathbf{G}$ is a winning position for $\pl$ in $\mathcal{G}_{\tlimbound}$ if and only if it is a winning position for $\pl$ in $\mathcal{G}_{\tlimbound'}$
(note that from this equivalence it follows that $\mathbf{G}$ has the same winning condition with respect to both $\tlimbound$ and $\tlimbound'$).

Let $\pos=(\pl',q,\psi)$ be a possible exit position of $\mathbf{G}$. Suppose first that $\pos$ is a winning position for $\pl$ in $\mathcal{G}_{\tlimbound'}$. By Proposition~\ref{the: Winning time labels}, the values of the winning time labels for the controller determine the lowest sufficiently large time limit for winning the corresponding embedded game---if there is one. 
We may thus assume that, within her/his winning strategy from $\pos$ in $\mathcal{G}_{\tlimbound'}$, the player $\pl$ always selects time limits corresponding to her/his winning time labels for the initial positions of the embedded games. Let $\Sigma_\pl$ denote such a strategy of $\pl$ for $\mathcal{G}_{\tlimbound'}$. Since, by the asumption, all the winning time labels of $\pl$ are below the time limit bound $\tlimbound$, it is possible for $\pl$ to follow $\Sigma_\pl$ from $\pos$ also within $\mathcal{G}_\tlimbound$ (note that the rules for $\mathcal{G}_{\tlimbound}$ and in $\mathcal{G}_{\tlimbound'}$ differ only when a time limit for an embedded game is chosen). It is now easy to see that $\Sigma_\pl$ a winning strategy from $\pos$ in $\mathcal{G}_{\tlimbound}$. 

Suppose then that $\pos$ is not a winning position for $\pl$ in $\mathcal{G}_{\tlimbound'}$. By the determinacy of bounded embedded games, $\pos$ must be a winning position for $\opl$ in $\mathcal{G}_{\tlimbound'}$. Hence, we can argue exactly as above, that $\pos$ is a winning position for $\opl$ in $\mathcal{G}_{\tlimbound}$ and thus it cannot be a winning position for $\pl$ in $\mathcal{G}_{\tlimbound}$.
We can thus conclude that $\pos$ is a winning position for $\pl$ in $\mathcal{G}_{\tlimbound}$ if and only it is a winning position for $\pl$ in $\mathcal{G}_{\tlimbound'}$.

By the equivalence that was proven above, it is easy to see by Definition~\ref{def: Winning time labels}, that the bounds $\tlimbound$ and $\tlimbound'$  cannot create two different \emph{ordinal valued} winning time labels for $\pl$ in $\mathbf{G}$. 
The remaining possibility is that there is a state $q\in\St$ which has a label $\lose/\win$ with the time limit bound $\tlimbound$ and some ordinal value $\tlim$ with the time limit bound $\tlimbound'$. But, again by Definition~\ref{def: Winning time labels}, we see that this is possible only if $\gamma\geq\tlimbound$, which would contradict the assumption that he winning time labels of $\pl$ in $\mathbf{G}$ with respect to $\tlimbound'$ are all strictly below $\tlimbound$. This concludes the proof.
\qed

\subsection*{Proof of Proposition~\ref{the: Canonical strategy}}

\noindent
1. We first consider the case where ${\pl}=\ctr$.
\smallskip

We will prove by transfinite induction on $\gamma<\tlimbound$
that for every $q\in\St$, if ${\pl}$ has a timed winning strategy in $\mathbf{G}[q,\tlim]$, then $(\tau_\pl,t_{can})$ is a timed winning strategy in $\mathbf{G}[q,\tlim]$.
Let the inductive hypothesis be that the
claim holds for every $\tlim'<\tlim$ and suppose that ${\pl}$ has a
timed winning strategy in $\mathbf{G}[q,\tlim]$.
By Proposition~\ref{the: Winning time labels}, we have $\mathcal{L}_\pl(q)=\delta$ for some $\delta\leq\tlim$. Let $\sigma_\pl'$ be a
strategy for which $\tau_\pl(q)=\sigma_\pl'(\delta,q)$ and there is a timer $t'$ such that $(\sigma_\pl',t')$ is a timed
winning strategy in $\mathbf{G}[q,\delta']$
for all $\delta'$ such that $\delta\leq\delta'<\tlimbound$
(such a strategy exists by the definition of the canonical strategy $\tau_{\pl}$).

Suppose first that $\delta=0$, whence $(\sigma_\pl',t')$ is a timed winning strategy in $\mathbf{G}[q,0]$. Now we have $\tau_\pl(q)=\sigma_\pl'(0,q)=\psi_\ctr$ and thus $(\tau_\pl,t_{can})$ is a timed winning strategy in $\mathbf{G}[q,\tlim]$.
Suppose then that $\delta>0$, whence $\sigma'_\pl(\delta,q)$ must be either some tuple of actions for the coalition $A$ or some response function for the coalition $\overline A$. Let $Q\subseteq\St$ be the set of states that is forced by $\sigma'_\pl(\delta,q)$ and let $q'\in Q$.  Since $(\sigma_\pl',t')$ is a timed winning strategy in $\mathbf{G}[q,\delta]$, there is $\delta'<\delta$ such that $(\sigma_\pl',t')$  is a timed winning strategy in $\mathbf{G}[q',\delta']$ (if $\delta$ is a limit ordinal, then $\delta'=t'(\delta,q')$, and else $\delta'=\delta-1$). Since $\delta\leq\tlim$, we have $\delta'<\tlim$.

For the inductive step, suppose first that $\tlim$ is a successor ordinal. Since we have $\delta'\leq\tlim\!-\!1$, there is a timed winning strategy $(\sigma_\pl'',t'')$ in $\mathbf{G}[q',\tlim\!-\!1]$ by Lemma~\ref{the: Regular strategies}. Thus, by the inductive hypothesis, $(\tau_\pl,t_{can})$ is a timed winning strategy in $\mathbf{G}[q',\tlim\!-\!1]$. Since this holds for every $q'\in Q$, we see that $(\tau_\pl,t_{can})$ is a timed winning strategy in $\mathbf{G}[q,\tlim]$.

Suppose now that $\tlim$ is a limit ordinal. Since $(\sigma_\pl',t')$  is a timed winning strategy in $\mathbf{G}[q',\delta']$, by Proposition~\ref{the: Winning time labels} we must have $\mathcal{L}_\pl(q')\leq\delta'<\tlim$. Thus, by the definition of the
canonical timer, $t_{can}(\tlim,q')=\mathcal{L}_\pl(q')$. By Proposition~\ref{the: Winning time labels}, there is a timed winning strategy $(\sigma_\pl'',t'')$ in $\mathbf{G}[q',\mathcal{L}_\pl(q')]$. Thus by the inductive hypothesis $(\tau_\pl,t_{can})$ is a timed winning strategy in $\mathbf{G}[q',\mathcal{L}_\pl(q')]$. Since this holds for every $q'\in Q$, and $t_{can}(\tlim,q')=\mathcal{L}_\pl(q')$ for every $q'\in Q$, we see that $(\tau_\pl,t_{can})$ is a timed winning strategy in $\mathbf{G}[q,\tlim]$.

\medskip

\noindent
2. We then consider the case where ${\pl}\neq\ctr$. 

\smallskip

We will prove by transfinite induction on $\gamma<\tlimbound$ that for every $q\in Q$: if ${\pl}$ has a winning strategy in $\mathbf{G}[q,\tlim]$, then $\tau_\pl$ is a winning strategy in $\mathbf{G}[q,\tlim]$. 
Suppose first that $\tlim=0$ and that ${\pl}$ has a winning strategy $\sigma_\pl$ in $\mathbf{G}[q,0]$. Now, since with the time limit $0$ the game will end at $q$ immediately, every strategy of $\pl$ will be a winning strategy in $\mathbf{G}[q,0]$. Hence, in particular, $\tau_\pl$ is a winning strategy in $\mathbf{G}[q,0]$.

For the inductive step, suppose that the claim holds for every $\tlim'<\tlim$ and that ${\pl}$ has a winning strategy $\sigma_\pl$ in $\mathbf{G}[q,\tlim]$. 
By Proposition~\ref{the: Winning time labels}, we have either $\mathcal{L}_\pl(q)=\win$ or $\mathcal{L}_\pl(q)>\tlim$. 

Assume first that $\mathcal{L}_\pl(q)=\win$. Let $\sigma_\pl$ be a strategy for which $l(q,\sigma_\pl)=\win$ and $\tau_\pl(\tlim,q)=\sigma_\pl(\tlim,q)$ (such a strategy exists by the definition of $\tau_\pl$).
Let $Q\subseteq\St$ be the set of states forced by $\sigma_\pl(\tlim,q)$ and let $q'\in Q$. 
Since $l(q,\sigma_\pl)=\win$, the strategy $\sigma_\pl$ is a winning strategy in $\mathbf{G}[q,\delta]$ for every $\delta<\tlimbound$, and therefore, as  we have $\tlim<\tlimbound$, the strategy $\sigma_\pl$ must also be a
winning strategy in $\mathbf{G}[q',\tlim']$ for every $\tlim'<\tlim$.
Thus there is a winning strategy in $\mathbf{G}[q',\tlim']$ for every $\tlim'<\tlim$ and
every $q'\in Q$. Hence, by the inductive hypothesis, $\tau_\pl$ is a
winning strategy in $\mathbf{G}[q',\tlim']$ for every $\tlim'<\tlim$ and $q'\in Q$.
Therefore we observe that $\tau_\pl$ is also a winning strategy in $\mathbf{G}[q,\tlim]$.
Assume now that $\mathcal{L}_\pl(q)=\delta>\tlim$. Let $\sigma_\pl$ be a strategy for which $l(q,\sigma_\pl)>\tlim$ and $\tau_\pl(\tlim,q)=\sigma_\pl(\tlim,q)$ (such a strategy exists by the definition of $\tau_\pl$). Let $Q\subseteq\St$ be the set of states that is forced by $\sigma_\pl(\tlim,q)$ and let $q'\in Q$. Since $\sigma_\pl$ is a winning strategy in $\mathbf{G}[q,\delta']$ for every $\delta' < l(q,\sigma_\pl)$
and since we have $\tlim < l(q,\sigma_\pl)$, the strategy $\sigma_\pl$ must also be a winning strategy in $\mathbf{G}[q',\tlim']$ for every $\tlim'<\tlim$. Hence we can deduce, as before, that $\tau_\pl$ is a winning strategy in
the games for the configurations over $Q$ that follow $(\gamma,q)$, and thus $\tau_{\pl}$ is also a winning strategy in $\mathbf{G}[q,\tlim]$.
\qed

\subsection*{Proof of Proposition~\ref{the: Correspondence of the winning time labels}}

We prove this claim by transfinite induction on the ordinal $\tlim<\tlimbound$. We first prove the special case where $\tlim=0$. Let $q\in\St$ and suppose first that $\mathcal{L}_{\ctr}(q)=0$, whence by Proposition~\ref{the: Winning time labels} the player $\ctr$ has a timed winning strategy $(\sigma_\ctr,t)$ in $\mathbf{G}[q,0]$. This is possible only if the exit position $(\ctr,q,\psi_\ctr)$ is a winning position for $\ctr$. In that case $\octr$
loses $\mathbf{G}[q,\gamma']$
with any time limit $\gamma'$ and thus $\mathcal{L}_{\octr}(q)=0$. Suppose then that $\mathcal{L}_{\octr}(q)=0$, whence
there no winning strategy for $\octr$ in $\mathbf{G}[q,0]$.
This is possible only if $\ctr$ wins at the exit position $(\ctr,q,\psi_\ctr)$, whence $\mathcal{L}_{\ctr}(q)=0$.
Now, let $\tlim>0$. The inductive hypothesis is that the claim holds for every ordinal $\tlim'<\tlim$.  
Suppose first that $\mathcal{L}_{\ctr}(q)=\tlim$. By Proposition~\ref{the: Canonical strategy} there is a timed winning strategy $(\sigma_\pl,t)$ in $\mathbf{G}[q,\tlim]$.
Hence $\octr$ cannot have a winning strategy in $\mathbf{G}[q,\gamma]$, and thus by Proposition~\ref{the: Winning time labels} we must have $\mathcal{L}_{\octr}(q)\leq\tlim$. If $\mathcal{L}_{\octr}(q)<\tlim$, then, by the inductive hypothesis,
we have $\mathcal{L}_{\ctr}(q)=\mathcal{L}_{\octr}(q)<\tlim$, which is a contradiction.
Therefore it must be that $\mathcal{L}_{\octr}(q)=\tlim$.
Suppose then that $\mathcal{L}_{\octr}(q)=\tlim$. We will next show that for any strategy $\sigma_{\octr}$ and any set $Q\subseteq\St$ forced by $\sigma_{\octr}(\tlim,q)$, there is a state $q'\in Q$ for which $\mathcal{L}_{\octr}(q')<\tlim$.
For the sake of contradiction, suppose that there is a strategy $\sigma_{\octr}$ such that for the set $Q\subseteq\St$ forced by $\sigma_{\octr}(\tlim,q)$, we have $\mathcal{L}_{\octr}(q')\geq\tlim$ 
or $\mathcal{L}_{\octr}(q') = \win$  for every $q'\in Q$. 
We construct the following strategy $\sigma_{\octr}'$ for $\octr$ in the embedded game $\mathbf{G}[q,\gamma]$:
\begin{align*}
	&\sigma_{\octr}'(\delta,q)=\sigma_{\octr}(\tlim,q) \text{ for every $\delta\leq\tlim$}, \\[-0,1cm]
	&\sigma_{\octr}'(\delta,q')=\tau_{\octr}(\delta,q') \text{ for every $\delta\leq\tlim$ and $q'\in\St\setminus\{q\}$.}
\end{align*}

Since $\mathcal{L}_{\octr}(q')\geq\tlim$ or $\mathcal{L}_{\octr}(q') = \win$ 
for every $q'\in Q$, by Proposition~\ref{the: Winning time labels}, the canonical strategy $\tau_{\octr}$ is a winning strategy in $\mathbf{G}[q',\delta]$ for any $q'\in Q$ and $\delta<\tlim$. Thus it is easy to see that $\sigma_{\octr}'$ is a winning strategy in $\mathbf{G}[q,\tlim]$. Hence, again by Proposition~\ref{the: Winning time labels},  we must have $\mathcal{L}_{\octr}(q)>\gamma$ or $\mathcal{L}_{\octr}(q) = \win$, which is a contradiction.
Therefore, we infer that 
\begin{align*}
	&\text{For any strategy $\sigma_{\octr}$ and $Q\subseteq\St$ forced by $\sigma_{\octr}(\tlim,q)$, }  \tag{$\star$}
	\\[-0,1cm]
	&\text{there is some $q'\in Q$ such that $\mathcal{L}_{\octr}(q')<\tlim$.}
\end{align*}
Let $Q':=\{q'\in\St \mid \mathcal{L}_{\octr}(q')<\tlim\}$. By the inductive hypothesis,  $\mathcal{L}_{\ctr}(q')<\tlim$ for every $q'\in Q'$.
We will show that $\ctr$ can play in such a way at $q$ that all possible
successor states will be in the set $Q'$.
Suppose first that $\ctr=\ovrf$. Since for every $\vec\alpha\in\avact(q,A)$, there is some strategy $\sigma_{\octr}$ such that  $\sigma_{\octr}(\tlim,q)=\vec\alpha$, we infer by ($\star$) that there is some response function for $\overline A$ which forces the next state to be in $Q'$.
Suppose then that $\ctr=\vrf$. If for every $\vec\alpha\in\avact(q,A)$ there existed some $\vec\beta\in\avact(q,\overline A)$ such that the outcome state of these actions was not in $Q'$, then there would be some strategy $\sigma_{\octr}$ such that the set forced by $\sigma_{\octr}(\tlim,q)$ would not intersect $Q'$. This is a contradiction with ($\star$), and thus there is some tuple of actions for $A$ at $q$ such that all possible successor states will be in $Q'$.

We next construct a strategy $\sigma_\ctr$ for $\ctr$ in $\mathbf{G}[q,\tlim]$. By the description above, we can define $\sigma_\ctr$ for every configuration $(\tlim',q)$, where $\tlim'\leq\tlim$, in such a way that the set forced by $\sigma_\ctr(\tlim',q)$ is a subset of $Q'$. For all configurations $(\tlim',q')$, where $\tlim'\leq\tlim$ and $q'\in\St\setminus\{q\}$, we define $\sigma_\ctr(\tlim',q')=\tau_\ctr(q')$.
Since $\mathcal{L}_\ctr(q')<\tlim$ for every $q'\in Q'$, 
we infer that  $t_{can}(\tlim,q')=\mathcal{L}_\ctr(q')<\tlim$ for every $q'\in Q$. By Propositions~\ref{the: Winning time labels} and \ref{the: Canonical strategy}\,, $(\tau_{\ctr},t_{can})$ is a timed winning strategy in $\mathbf{G}[q',t_{can}(\tlim,q')]$ for any $q'\in Q'$. Thus, it is easy to see that $(\sigma_{\ctr},t_{can})$ is a timed winning strategy in $\mathbf{G}[q,\tlim]$. However, since $\mathcal{L}_{\octr}(q)=\gamma$, 
we conclude, using the inductive hypothesis, that there cannot be a timed winning strategy for $\ctr$ in $\mathbf{G}[q,\tlim']$ for any $\tlim'<\tlim$. Hence by Proposition~\ref{the: Winning time labels} we must have $\mathcal{L}_{\ctr}(q)=\gamma$.
\qed

\subsection*{Proof of Proposition \ref{the: Determinacy of the evaluation games}}

Since both of the players cannot have a winning strategy, the implication from left to right is immediate. We prove the implication from right to left by induction on $\varphi$.
\begin{itemize}[leftmargin=*]
\item Suppose that $\varphi=p$ (where $p\in\Prop$).
By the rules of the game, it is easy to see that each ending position is a winning position for either of the players. The claim follows from this.

\item Suppose that $\varphi=\neg\psi$.
Since negation just swaps the verifier and falsifier for the position, this case follows directly from the inductive hypothesis.

\item Suppose that $\varphi=\psi\vee\theta$.
Suppose that $\opl$ does not have a winning strategy in $\mathcal{G}(\mathcal{M},q,\psi\vee\theta)$. 
Suppose first that $\opl={\bf A}$, whence it is not possible that both $({\bf E},q,\psi)$ and $({\bf E},q,\theta)$ are winning positions for $\opl$. Thus, by the inductive hypothesis, at least one of these positions must be a winning position for ${\bf E}=\pl$, and thus $\pl$ has a winning strategy in $\mathcal{G}(\mathcal{M},q,\psi\vee\theta)$.
Suppose then that $\opl={\bf E}$, whence neither of the positions $({\bf E},q,\psi)$ nor $({\bf E},q,\theta)$ can be a winning position for $\opl$. Thus, by the inductive hypothesis, both of these positions must be winning positions for ${\bf A}=\pl$, and thus $\pl$ has a winning strategy in $\mathcal{G}(\mathcal{M},q,\psi\vee\theta)$.

\item Suppose that $\varphi=\coop{A}\X\psi$.
This case can be proven by an analogous argument as the one used in the inductive step of Proposition~\ref{the: Correspondence of the winning time labels}. Note that the rule related to the position $({\bf E},q,\coop{A}\X\psi)$ is just a special case of the bounded embedded game $\embgame({\bf E},{\bf E},A,q,\psi,\bot)[1]$ with the restriction that neither of the players may end the game before the time runs out.

\item Suppose that $\varphi=\coop{A}\psi\U\theta$.
Let $\mathbf{G}$ be the embedded game that arises from $\varphi$.
From the inductive hypothesis it follows that the player who has a winning strategy in $\mathbf{G}$ also has a winning strategy in the evaluation game that is continued (recall Definition~\ref{def: Winning strategies in embedded game}). 

Suppose first that $\opl={\bf A}$, whence there must be (at least one) $\tlim<\tlimbound$ s.t. $\opl$ does not have winning strategy in $\mathbf{G}$ with the time limit $\tlim$. Since bounded embedded games are determined (by Corollary~\ref{the: Determinacy of the bounded embedded game}), it follows from the inductive hypothesis, that ${\bf E}=\pl$ must have a (timed) winning strategy in $\mathbf{G}[\tlim]$. Hence $\pl$ has a winning strategy in $\mathcal{G}(\mathcal{M},q,\coop{A}\psi\U\theta)$ ($\pl$ first chooses the time limit to be $\tlim$ and then follows her timed winning strategy in $\mathbf{G}[\tlim]$).

Suppose then that $\opl={\bf E}$, whence $\opl$ cannot have a (timed) winning strategy  in $\mathbf{G}$ with any  time limit $\tlim<\tlimbound$. Since bounded embedded games are determined (by Corollary~\ref{the: Determinacy of the bounded embedded game}), it follows from the inductive hypothesis, that ${\bf A}=\pl$ must have a winning strategy in $\mathbf{G}[\tlim]$, for every $\tlim<\tlimbound$. Hence $\pl$ has a winning strategy in $\mathcal{G}(\mathcal{M},q,\coop{A}\psi\U\theta)$ (after $\opl$ has chosen a time limit $\tlim$, then $\pl$ just follows his winning strategy in $\mathbf{G}[\tlim]$).

\item The case $\varphi=\coop{A}\psi\Rl\theta$ is proven dually to the case $\varphi=\coop{A}\psi\U\theta$. 
\qed
\end{itemize}

\subsection*{Proof of Lemma \ref{the: Winning time labels 2}}

Suppose that $\mathcal{L}_{\pl}(q)=\tlim>0$. Since $\mathcal{L}_{\pl}(q)\neq 0$, canonically timed strategy $(\tau_{\pl},t_{can})$ is not a timed winning strategy for $\pl$ in $\mathbf{G}[q,0]$. Therefore $\tau_\pl(q)$ is either some tuple of actions for $A$ or some response function for $\overline A$. Let $Q\subseteq\St$ be the set of states that is forced by $\tau_{\pl}(q)$. 
We first show that $\mathcal{L}_{\pl}(q')<\tlim$ for every $q'\in Q$. Since $(\tau_{\pl},t_{can})$ is a timed winning strategy in $\mathbf{G}[q,\tlim]$, it must also be a timed winning strategy in $\mathbf{G}[q',t_{can}(\tlim,q')]$ for every $q'\in Q$. Hence by the definition of canonical timer $\tlim>t_{can}(\tlim,q')=\mathcal{L}_{\pl}(q')$ for every $q'\in Q$.

Suppose first that $\tlim$ is a successor ordinal. If we would have $\mathcal{L}_{\pl}(q')<\tlim\!-\!1$ for every $q'\in Q$, then $(\tau_{\pl},t_{can})$ would be a winning strategy in $\mathbf{G}[q,\tlim\!-\!1]$, and thus we would have $\mathcal{L}_{\pl}(q)\leq\tlim\!-\!1$. Hence $\tlim>\mathcal{L}_{\pl}(q')\geq\tlim\!-\!1$ for some $q'\in Q$ and thus $\max\{\mathcal{L}_{\pl}(q')\mid q'\in Q\}=\tlim\!-\!1$.

Suppose then that $\tlim$ is a limit ordinal. If $\tlim'<\tlim$ would be an upper bound for the winning time labels in $Q$, then we would have $\mathcal{L}_\pl(q')<\tlim'\!+\!1$ for every $q'\in Q$. Hence $(\tau_{\pl},t_{can})$ would be a timed winning strategy in $\mathbf{G}[q,\tlim'\!+\!1]$, and thus we would have $\mathcal{L}_{\pl}(q)\leq\tlim'\!+\!1<\tlim$. This is impossible and thus $\sup\{\mathcal{L}_{\pl}(q')\mid q'\in Q\}=\tlim$.

\subsection*{Proof of Theorem \ref{the: Equivalence of unbounded compositional and game-theoretic semantics}}

We show by induction on $\varphi$ that the following equivalence holds for every $q\in\St$.
\[
	\mathcal{M},q\models\varphi \;\text{ iff }\; \mathcal{M},q\models_u^g\varphi.
\]
\begin{itemize}[leftmargin=*]
\item Suppose that $\varphi=p$ ($p\in\Prop$).
This case is trivial since the winning condition for the ending position $({\bf E},q,p)$ corresponds to the truth condition for $p$ in the compositional semantics.

\item Suppose that $\varphi=\neg\psi$.

Suppose first that $\mathcal{M},q\models\neg\psi$, i.e. $\mathcal{M},q\not\models\psi$. Therefore, by the inductive hypothesis, $\mathcal{M},q\not\models_u^g\psi$. Since, by Proposition~\ref{the: Determinacy of the evaluation games}, evaluation games are determined, $({\bf E},q,\psi)$ is a winning position for Abelard. Therefore, by the rule for negation, $({\bf A},q,\neg\psi)$ is also a winning position for Abelard. Since the all rules of the game for the players are symmetric, dually $({\bf E},q,\neg\psi)$ is a winning position for Eloise and thus $\mathcal{M},q\models_u^g\neg\psi$.

Suppose then that $\mathcal{M},q\models_u^g\neg\psi$. Now, by the rule for negation, $({\bf A},q,\psi)$ is a winning position for Eloise and thus dually $({\bf E},q,\psi)$ is a winning position for Abelard. But since only one of the the players can have a winning strategy from the position $({\bf E},q,\psi)$, we must have $\mathcal{M},q\not\models_u^g\psi$. Hence, by the inductive hypothesis, $\mathcal{M},q\not\models\psi$, i.e. $\mathcal{M},q\models\neg\psi$.

\item Suppose that $\varphi=\psi\vee\theta$.

Suppose first that $\mathcal{M},q\models\psi\vee\theta$, i.e. $\mathcal{M},q\models\psi$ or $\mathcal{M},q\models\theta$. If the former holds, we define $\strev_{\bf E}({\bf E},q,\psi\vee\theta)=\psi$ and else define $\strev_{\bf E}({\bf E},q,\psi\vee\theta)=\theta$. Now, by the inductive hypothesis, the next position of the game must a winning position for Eloise and thus $\mathcal{M},q\models_u^g\psi\vee\theta$.

Suppose then that $\mathcal{M},q\models_u^g\psi\vee\theta$, i.e. Eloise has a winning strategy $\strev_{\bf E}$ from the position $({\bf E},q,\psi\vee\theta)$.
Now $\strev_{\bf E}$ picks either $\psi$ or $\theta$.
If $\strev_{\bf E}({\bf E},q,\psi\vee\theta)=\psi$, then by the inductive hypothesis we must have $\mathcal{M},q\models\psi$. And if $\strev_{\bf E}({\bf E},q,\psi\vee\theta)=\theta$, then we analogously have $\mathcal{M},q\models\theta$.
Therefore $\mathcal{M},q\models\psi\vee\theta$.

\item Suppose that $\varphi=\coop{A}\atlx\psi$.

Suppose first that $\mathcal{M},q\models\coop{A} \X\psi$, i.e., there exists a collective strategy $S_A$ such that for each $\Lambda\in\paths(q,S_A)$, we have $\mathcal{M},\Lambda[1]\models\psi$. 
Let $\strev_{\bf E}({\bf E},q,\coop{A} \X\psi)$ be the tuple in $\avact(A,q)$ which is determined by $S_A$ at $q$. Now, regardless of the actions chosen by Abelard for the  agents in $\overline{A}$, the resulting state $q'$ must be $\Lambda[1]$ for some $\Lambda\in\paths(q,S_A)$ and thus $\mathcal{M},q'\models\psi$. 
By the inductive hypothesis $\mathcal{M},q'\models_u^g\psi$, i.e. Eloise has a winning strategy from the  position $({\bf E},q',\psi)$. Therefore we have $\mathcal{M},q\models_u^g\coop{A} \X\psi$.

Suppose then that $\mathcal{M},q\models_u^g\coop{A} \X\psi$, i.e. Eloise has a winning strategy $\strev_{\bf E}$ from the position $({\bf E},q,\coop{A} \X\psi)$.
Now $\strev_{\bf E}$ assigns some tuple of choices for the agents in $A$. We can now construct  a related collective strategy $S_A$  by using those choices at $q$; the choices at other states may be arbitrary. Let $\Lambda\in\paths(q,S_A)$. Now Abelard can choose such actions for $\overline A$ that the resulting state is $\Lambda[1]$. Since $\strev_{\bf E}$ is a winning strategy, the position $({\bf E},\Lambda[1],\psi)$ is a winning position for Eloise. Hence $\mathcal{M},\Lambda[1]\models_u^g\psi$ and thus, by the inductive hypothesis, $\mathcal{M},\Lambda[1]\models\psi$. Therefore $\mathcal{M},q\models\coop{A} \X\psi$.

\item Suppose that $\varphi=\coop{A}\psi\atlu\theta$.

Suppose first that $\mathcal{M},q\models\coop{A}\psi \U\theta$, i.e., there exists $S_A$ such that for each $\Lambda\in\paths(q,S_A)$, there is some $i\geq 0$\, such that $\mathcal{M},\Lambda[i]\models\theta$ and $\mathcal{M},\Lambda[j]\models\psi$ for every $j < i$. Let $\strev_{\bf E}({\bf E},q,\coop{A}\psi\U\theta)$ be the strategy $\sigma_{\bf E}$, defined as follows: 
let $\sigma_{\bf E}(q')=\theta$ for each $q'\in\St$ where $\mathcal{M},q'\models\theta$, and for all other states $q'\in\St$, let $\sigma_{\bf E}(q')$ be the tuple of actions for the agents in $A$ determined by $S_A$. Now, regardless of the actions of Abelard, all of the states that are reached in the embedded game must be states $\Lambda[i]$ for some $\Lambda\in\paths(q,S_A)$ and $i\geq 0$. 
Thus, when Eloise uses $\sigma_{\bf E}$, a state $q'$ where $\mathcal{M},q'\models\theta$ is reached in a finite number of rounds.  If Abelard ends the game before that at some state $q'$, then we have $\mathcal{M},q'\models\psi$. By the inductive hypothesis, either of these cases will result in an exit position which is a winning position for Eloise. Therefore we have $\mathcal{M},q\models_u^g\coop{A}\psi \U\theta$.

%\medskip

Suppose then that $\mathcal{M},q\models_u^g\coop{A}\psi \U\theta$, i.e. Eloise has a winning strategy $\strev_{\bf E}$ from the position $({\bf E},q,\coop{A}\psi \U\theta)$.
Now $\strev_{\bf E}({\bf E},q,\coop{A}\psi \U\theta)$ is some strategy $\sigma_{\bf E}$ for the corresponding embedded game.
We can now construct  a collective strategy $S_A$ that is related to the strategy $\sigma_{\bf E}$: for any state where $\sigma_{\bf E}$ assigns some tuple of actions for agents in $A$, we define the same actions for $S_A$. For states where $\sigma_{\bf E}$ instructs to end the game, we may define arbitrary actions for $S_A$. We will use this same method from now on, when we define collective strategies $S_A$ related to the strategies of $\vrf$ in an embedded game.

Let $\Lambda\in\paths(q,S_A)$. Now, when Eloise uses $\sigma_{\bf E}$ 
there will be actions of Abelard such that the states of the
embedded game are on $\Lambda$ until some configuration $\Lambda[i]$ at which Eloise ends the game at the exit position $({\bf E},\Lambda[i],\theta)$. (Note that since $\strev_{\bf E}$ is a winning strategy, Eloise must always end the embedded game after finitely many steps.)
Since $\strev_{\bf E}$ is a winning strategy, we must have $\mathcal{M},\Lambda[i]\models\theta$ by the inductive hypothesis. Let then $j < i$. Now Abelard can end the game after $j$ rounds at the position $({\bf E},\Lambda[j],\psi)$. Since that is a winning position for Eloise, we must have $\mathcal{M},\Lambda[j]\models\psi$ by the inductive hypothesis. Therefore we conclude that $\mathcal{M},q\models\coop{A}\psi \U\theta$.

\item Suppose that $\varphi=\coop{A}\psi\atlr\theta$.

Suppose that $\mathcal{M},q\models\coop{A}\psi \Rl\theta$, i.e. there exists a strategy $S_A$ such that for each $\Lambda\in\paths(q,S_A)$ and $i\geq 0$ either $\mathcal{M},\Lambda[i]\models\theta$ or there is  $j < i$\, such that $\mathcal{M},\Lambda[j]\models\psi$. Let $\strev_{\bf E}({\bf E},q,\coop{A}\psi\Rl\theta)$ be the strategy $\sigma_{\bf E}$ that is defined as follows: $\sigma_{\bf E}(q')=\psi$ for each $q'\in\St$ where $\mathcal{M},q'\models\psi$ and for all other states $q'\in\St$, let $\sigma_{\bf E}(q')$ be the tuple of actions for the agents in $A$ determined by $S_A$. Now, all states that are reached in the embedded game must be states $\Lambda[i]$ for some $\Lambda\in\paths(q,S_A)$ and $i\geq 0$. Thus, when Eloise uses $\sigma_{\bf E}$, she will either stay at states $q'$ where $\mathcal{M},q'\models\theta$ for infinitely long or reach a state $q'$ where $\mathcal{M},q'\models\psi$ at some point. Hence either 1) the embedded game continues infinitely long, whence Eloise wins the whole evaluation game, or 2)  by the inductive hypothesis, the exit position is a winning position for Eloise. Therefore we have $\mathcal{M},q\models_u^g\coop{A}\psi \Rl\theta$.

%\medskip

Suppose then that $\mathcal{M},q\models_u^g\coop{A}\psi \Rl\theta$, i.e. Eloise has a winning strategy $\strev_{\bf E}$ from the position $({\bf E},q,\coop{A}\psi \Rl\theta)$.
Now $\strev_{\bf E}({\bf E},q,\coop{A}\psi\Rl\theta)$ is a strategy $\sigma_{\bf E}$.
Let $S_A$ be the collective strategy that is related to $\sigma_{\bf E}$ 
and let $\Lambda\in\paths(q,S_A)$. Now, when Eloise uses $\sigma_{\bf E}$, there
exist actions of Abelard such that the states of the embedded game are on $\Lambda$ (until a state is reached where Eloise ends the game---if such a state exists).
We need to show that for every $i\geq 0$ either $\mathcal{M},\Lambda[i]\models\theta$ or there is $j<i$ such that  $\mathcal{M},\Lambda[j]\models\psi$. Let $i\geq 0$. If Eloise ends the game before $i$ rounds, the game ends at an exit position $({\bf E},\Lambda[j],\psi)$ for some $j < i$, and we can conclude that $\mathcal{M},\Lambda[j]\models\psi$ by the  inductive hypothesis. If Eloise does not end the game before $i$ rounds have been played, then Abelard can end it in the exit position $({\bf E},\Lambda[i],\theta)$. We can then conclude that  $\mathcal{M},\Lambda[i]\models\theta$ by the inductive hypothesis. Therefore $\mathcal{M},q\models\coop{A}\psi \Rl\theta$.
\end{itemize}
This completes the proof of the theorem. \qed

\subsection*{Proof of Theorem \ref{the: Equivalence of finitely bounded compositional and game-theoretic semantics}}

By Lemma~\ref{the: Time uniform strategies}, we may assume for this proof that all of the (winning) strategies that Eloise uses in embedded games depend on states only. This amounts to assuming that their domain is the set of states instead of configurations. We also recall that timers are not needed in the finitely bounded case.
We show by induction on $\varphi$ that the following holds for every $q\in\St$.
\[
	\mathcal{M},q\models_f\varphi \;\text{ iff }\; \mathcal{M},q\models_f^g\varphi.
\]
\begin{itemize}[leftmargin=*]
\item The cases where $\pos=({\pl},q,p)$ ($p\in\Prop$),
$\pos=({\pl},q,\neg\psi)$, $\pos=({\pl},q,\psi\vee\theta)$
or $\pos=({\pl},q,\coop{A} \X\psi)$ are treated exactly as in the proof of
Theorem~\ref{the: Equivalence of unbounded
compositional and game-theoretic semantics}.

\item Suppose that $\varphi=\coop{A}\psi\U\theta$.

Suppose first that $\mathcal{M},q\models_{f}\coop{A}\psi \U\theta$, i.e. there exist some $n<\omega$ and $S_A$ such that for each $\Lambda\in\paths(q,S_A)$, there is some $i\leq n$\, such that  $\mathcal{M},\Lambda[i]\models_{f}\theta$ and $\mathcal{M},\Lambda[j]\models_{f}\psi$ for every $j < i$.
Let $\strev_{\bf E}({\bf E},q,\coop{A}\psi \U\theta)=(n,\sigma_{\bf E})$, where $\sigma_{\bf E}(q')=\theta$ for each $q'\in\St$ where $\mathcal{M},q'\models_{f}\theta$ and for all other states $q'\in\St$ let $\sigma_{\bf E}(q')$ be the tuple of actions for the agents in $A$ chosen according to $S_A$.
Now, when Eloise chooses the time limit to be $n$ and uses $\sigma_{\bf E}$, then, regardless of the actions of Abelard, all states that are reached in the game must be states $\Lambda[i]$ for some $\Lambda\in\paths(q,S_A)$ and $i\leq n$. Thus Eloise can reach a state $q'$ where $\mathcal{M},q'\models_f\theta$ in $n$ rounds; if Abelard ends the game before that at some state $q'$, then $\mathcal{M},q\models_f\psi$. By the inductive hypothesis, either of these cases will result in an exit position which is a winning position for Eloise. Therefore we have $\mathcal{M},q\models_f^g\coop{A}\psi \U\theta$.

%\medskip

Suppose then that $\mathcal{M},q\models_f^g\coop{A}\psi \U\theta$,  i.e. Eloise has a winning strategy $\strev_{\bf E}$ from the position $({\bf E},q,\coop{A}\psi \U\theta)$. 
Now $\strev_{\bf E}({\bf E},q,\coop{A}\psi\U\theta)=(n,\sigma_{\bf E})$ where $n<\omega$. Let $S_A$ be the collective strategy that is related to the strategy $\sigma_{\bf E}$ (see the corresponding part in the proof of Theorem~\ref{the: Equivalence of unbounded compositional and game-theoretic semantics}). Let $\Lambda\in\paths(q,S_A)$. Now, when Eloise uses $\sigma_{\bf E}$ there exist actions of Abelard such that the states of the embedded game will be on $\Lambda$ until some configuration $(n\!-\!i,\Lambda[i])$ (with $i\leq n$) at which Eloise ends the game at the position $({\bf E},\Lambda[i],\theta)$. (If she does not end the game, then it will automatically end at the exit position $({\bf E},\Lambda[n],\theta)$.) 
Since $\strev_{\bf E}$ is a winning strategy, we have $\mathcal{M},\Lambda[i]\models_f\theta$ by the inductive hypothesis. Let then $j < i$. Since Abelard can end the game after $j$ rounds at the position $({\bf E},\Lambda[j],\psi)$, by the inductive hypothesis, we have $\mathcal{M},\Lambda[j]\models_f\psi$. Hence $\mathcal{M},q\models_f\coop{A}\psi \U\theta$.

\item Suppose that $\varphi=\coop{A}\psi\Rl\theta$.

Suppose first that $\mathcal{M},q\models_{f}\coop{A}\psi \Rl\theta$, i.e., for all $n<\omega$, there exists a collective strategy $S_{A,n}$ such that for each $\Lambda\in\paths(q,S_{A,n})$ and $i\leq n$, we have either $\mathcal{M},\Lambda[i]\models_{f}\theta$ or there is some $j < i$\, such that $\mathcal{M},\Lambda[j]\models_{f}\psi$. Let $\strev_{\bf E}({\bf E},q,\coop{A}\psi\U\theta)$ be a function that maps every $n<\omega$ to $\sigma_{{\bf E},n}$, where $\sigma_{{\bf E},n}$ is defined as follows.
Let $\sigma_{{\bf E},n}(q')=\psi$ for each $q'\in\St$ where $\mathcal{M},q'\models_{f}\psi$. For all other states $q'\in\St$, let $\sigma_{{\bf E},n}(q')$ be the tuple of actions for the agents in $A$ chosen according to $S_{A,n}$. Now, when Eloise uses $\sigma_{{\bf E},n}$, all states that can be  reached must be states $\Lambda[i]$ for some $\Lambda\in\paths(q,S_A)$ and $i\leq n$. Thus Eloise will either stay at states $q'$ where $\mathcal{M},q'\models_f\theta$ for $n$ rounds or reach a state $q'$ where $\mathcal{M},q'\models_f\psi$ while maintaining the truth of $\theta$. By the inductive hypothesis, either of these cases will result in an exit position which is a winning position for Eloise. Therefore we have $\mathcal{M},q\models_f^g\coop{A}\psi \Rl\theta$.

%\medskip

Suppose then that $\mathcal{M},q\models_f^g\coop{A}\psi \Rl\theta$,  i.e. Eloise has a winning strategy $\strev_{\bf E}$ from the position $({\bf E},q,\coop{A}\psi \Rl\theta)$. 
Now $\strev_{\bf E}({\bf E},q,\coop{A}\psi\Rl\theta)$ assigns some strategy $\sigma_{{\bf E},n}$ for every $n<\omega$. Let $n<\omega$ and let $S_{A,n}$ be the collective strategy that is related to $\sigma_{{\bf E},n}$. Let $\Lambda\in\paths(q,S_{A,n})$. Now, when Eloise plays using $\sigma_{{\bf E},n}$, there are some actions of Abelard such that the states of the embedded game are on $\Lambda$ until Eloise ends the game or $n$ rounds have lapsed. We need to show that for every $i\leq n$ either $\mathcal{M},\Lambda[i]\models_f\theta$ or there is some $j<i$ such that  $\mathcal{M},\Lambda[j]\models_f\psi$. Let $i\leq n$. If Eloise ends the game before $i$ rounds have gone, the game ends at the position $({\bf E},\Lambda[j],\psi)$ for some $j < i$. Then, by the inductive hypothesis, $\mathcal{M},\Lambda[j]\models_f\psi$. If Eloise does not end the game before $i$ rounds have lapsed, then Abelard may end it at the position $({\bf E},\Lambda[i],\theta)$, whence, by the inductive hypothesis, $\mathcal{M},\Lambda[i]\models_f\theta$. Therefore $\mathcal{M},q\models_f\coop{A}\psi \Rl\theta$.
\end{itemize}
This completes the proof of the theorem. \qed

\end{document}